\documentclass[11 pt]{amsart}

\usepackage{amsmath,amssymb,amsthm,amsfonts,setspace,hyperref,dsfont,yhmath,euscript}

\usepackage{hyperref}\hypersetup{colorlinks=true, citecolor=green, linkcolor=blue, hypertexnames=false}

\usepackage[letterpaper,margin=1.4in]{geometry}

\usepackage{color}

\newcommand{\NN}{\mathbb{N}}
\newcommand{\RR}{\mathbb{R}}
\newcommand{\R}{\mathbb{R}}

\newcommand{\CC}{\mathbb{C}}

\newcommand{\TT}{\mathbb{T}}

\newcommand{\ZZ}{\mathbb{Z}}
\newcommand{\Z}{\mathbb{Z}}

\newcommand{\norm}[1]{\lVert#1\rVert}

\newtheorem{theorem}{Theorem}[section]
\newtheorem{corollary}[theorem]{Corollary}

\newtheorem{lemma}[theorem]{Lemma}
\newtheorem{proposition}[theorem]{Proposition}

\newtheorem{conjecture}[theorem]{Conjecture}

\newcommand{\comment}[1]{}
\theoremstyle{definition}

\newtheorem{remark}[theorem]{Remark}

\numberwithin{equation}{section}

\makeatletter
\renewcommand*\env@matrix[1][*\c@MaxMatrixCols c]{%
  \hskip -\arraycolsep
  \let\@ifnextchar\new@ifnextchar
  \array{#1}}
\makeatother

\begin{document}
\title[Local rigidity]
{Global periodic-data rigidity for irreducible toral automorphisms}

\thanks{ $^1$ Based on research supported by NSF grant DMS-2452194}


\author[]{ Zhenqi Jenny Wang$^1$ }

\address{Department of Mathematics\\
        Michigan State University\\
        East Lansing, MI 48824,USA}
\email{wangzq@math.msu.edu}

\begin{abstract}
We prove a global \(C^{1+\text{H\"older}}\)-rigidity theorem for Anosov
diffeomorphisms of tori with irreducible linearization. Let
\(f:\mathbb T^N\to\mathbb T^N\) be a \(C^2\) Anosov diffeomorphism with
linearization \(A\in GL(N,\mathbb Z)\), and assume that \(A\) is irreducible.
If, for every periodic point \(p=f^n p\), the linear maps \(Df_p^n\) and
\(A^n\) are conjugate, then the Franks--Manning conjugacy between \(f\) and
\(A\) is \(C^{1+\text{H\"older}}\). Thus, in the irreducible
case, periodic data completely characterize global
\(C^{1+\text{H\"older}}\)-rigidity.

The proof does not assume conformality, uniform quasiconformality, simplicity
of the spectrum, or any restriction on Lyapunov multiplicities. The main
ingredient is a new partial-to-global rigidity mechanism combining geometric
and analytic arguments. We first obtain partial cocycle rigidity on canonical conformal
layers inside the Lyapunov blocks by geometric methods, and then promote this partial rigidity to full
regularity of the conjugacy along the Lyapunov blocks by analytic methods. The same method
yields a local rigidity theorem for \(C^1\)-small
\(C^{1+\text{H\"older}}\) perturbations of \(A\).

\end{abstract}


\maketitle

\setcounter{tocdepth}{2}

\tableofcontents
\section{Introduction}
Hyperbolic toral automorphisms are the fundamental algebraic models of
uniformly hyperbolic dynamics. By the Franks--Manning classification theorem \cite{Franks}, \cite{man},
every Anosov diffeomorphism \(f:\mathbb T^N\to\mathbb T^N\) is topologically
conjugate to its linearization \(A\in GL(N,\mathbb Z)\). More precisely, there exists a homeomorphism \(H\) such that
\begin{align*}
 H\circ f=A\circ H.
\end{align*}
Any two such conjugacies differ by an affine automorphism of $\TT^N$
commuting with $A$ \cite{W};  in particular, they have the same regularity. The conjugacy is always bi-H\"older, but in general it need not be $C^1$.

The central rigidity problem is to decide when this topological conjugacy has
higher regularity. The most elementary \(C^1\)-obstruction is given by
periodic data. Suppose that \(H\) is a \(C^1\) conjugacy between \(f\) and
\(A\). If \(p=f^n p\), then \(H(p)\) is an \(n\)-periodic point of \(A\), and
differentiating
\[
        H\circ f^n=A^n\circ H
\]
at \(p\) gives
\[
        DH_p\circ Df_p^n=A^n\circ DH_p .
\]
Thus \(Df_p^n\) and \(A^n\) must be linearly conjugate. We say that \(f\) and
\(A\) have \emph{the same periodic data} if this condition holds for every
periodic point \(p=f^n p\).

The main question
is whether, for irreducible toral automorphisms, this necessary condition is
also sufficient. The irreducibility assumption is essential: in the reducible case, de la
Llave constructed examples showing that periodic data alone do not imply
$C^1$ rigidity. The periodic-data rigidity problem for Anosov automorphisms goes back to the
foundational work of de la Llave-Marco-Moriy\'on in the 1980s. In the
irreducible case it has led to the following conjecture.
\begin{conjecture}\label{con:1}
Let \(f:\mathbb T^N\to\mathbb T^N\) be a \(C^{s}\), $s>1$ Anosov diffeomorphism with
linearization \(A\). Suppose that \(A\) is irreducible. Then:
\begin{enumerate}
\item \emph{Local rigidity.} If $f$ is a $C^1$-small perturbation of $A$, and if \(f\) and \(A\) have
the same periodic data, then \(f\) is
\(C^{1+\text{H\"older}}\) conjugate to \(A\).

\smallskip

  \item \emph{Global rigidity.} If $s$ is sufficiently large and if \(f\) and \(A\) have
the same periodic data, then the topological conjugacy between \(f\) and \(A\)
is  \(C^{1+\text{H\"older}}\).

\smallskip
  \item\label{for:3} \emph{$C^\infty$ bootstrapping.} If $f\in C^\infty(\TT^N)$ and the conjugacy is $C^1$, then the conjugacy is in fact $C^\infty$.
\end{enumerate}
\end{conjecture}
 The $C^\infty$ bootstrapping statement, item~\ref{for:3}, was recently proved in
\cite{wang2}. In this paper we prove the periodic-data rigidity statements.
Our first main theorem is the following global rigidity result.

\begin{theorem}\label{th:6}
 Let $f$ be a $C^2$ Anosov diffeomorphism on $\mathbb{T}^N$, and let $A$ be its
linearization. Suppose that $A$ is an irreducible automorphism of
$\mathbb{T}^N$. If $f$ and $A$ have the same periodic data, then $f$ is
$C^{1+\text{H\"older}}$ conjugate to $A$.
\end{theorem}
This theorem is global: no smallness assumption is imposed on $f$. The $C^2$
regularity assumption is used to guarantee the existence of the dominated
splittings needed in the proof. In the local setting, this regularity
assumption can be reduced.
\begin{theorem}\label{th:5}
 Let $A$ be an irreducible Anosov automorphism of $\mathbb T^N$, and let $f$
be a $C^1$-small $C^{1+\text{H\"older}}$ perturbation of $A$. If $f$ and $A$ have the same periodic data, then $f$ is
$C^{1+\text{H\"older}}$ conjugate to $A$.
\end{theorem}
Together with the $C^\infty$ bootstrapping theorem of \cite{wang2}, these results
settle Conjecture \ref{con:1} in the stated regularity classes.

\subsection{History and motivation}

In dimension two, the \(C^\infty\) local periodic-data rigidity problem for
Anosov automorphisms was solved by de la Llave--Marco--Moriy\'on in their
pioneering work \cite{L0,L1,LM}, building on \cite{LMM}. In higher dimensions the situation is
more subtle: reducibility creates genuine obstructions that are invisible to
periodic data.

Recall that a toral automorphism is irreducible if its characteristic
polynomial is irreducible over \(\mathbb Q\), equivalently, if it has no
nontrivial rational invariant subspace. De la Llave constructed examples of
reducible Anosov automorphisms and \(C^\infty\) perturbations with the same
periodic data that are only H\"older conjugate. Moreover, for every
\(k\in\mathbb N\) and every \(N\geq4\), there are reducible Anosov
automorphisms on \(\TT^N\) with \(C^\infty\) perturbations having the same
periodic data but whose conjugacies are only \(C^k\), and not \(C^{k+1}\)
\cite{L1}. These counterexamples show that irreducibility is an essential
hypothesis and led to Conjecture~\ref{con:1}.

The regularity problem naturally splits into two parts. The first asks whether
the periodic-data obstruction implies \(C^1\), or more precisely
\(C^{1+\text{H\"older}}\), regularity of the conjugacy. The second asks whether
a \(C^1\) conjugacy between \(C^\infty\) systems automatically bootstraps to a
\(C^\infty\) conjugacy. The \(C^\infty\) bootstrapping problem is the third
part of Conjecture~\ref{con:1}. A positive answer on \(\TT^3\) was given in
\cite{G17}. The result of \cite{wang1} gives such a bootstrapping theorem when
\(f\) is sufficiently \(C^l\)-close to \(A\), for very large \(l\), under a
weak irreducibility assumption on \(A\). Recently, a complete answer was given
in \cite{wang2} under a much weaker irreducibility assumption. The present
paper addresses the periodic-data rigidity parts of Conjecture~\ref{con:1}.

There has been substantial partial progress toward the local rigidity part of
Conjecture~\ref{con:1} for large classes of irreducible automorphisms \(A\);
see \cite{KS03,L3,L2,G08,KS09,GKS11,KS13,S15}. More recently, some of these
local results were globalized in \cite{WA}. These results typically require
additional spectral or geometric assumptions, often ensuring that the relevant
Lyapunov blocks admit invariant conformal structures or are uniformly
quasiconformal; for instance, such assumptions are often satisfied when the
relevant Lyapunov blocks are one- or two-dimensional.

\subsection{Comparison with previous methods}

The present results remove all of these additional assumptions: no smallness
assumption, no simplicity assumption on the spectrum, no restriction on
Lyapunov multiplicities, and no full invariant conformal structure are
imposed. The
only algebraic hypothesis is irreducibility of \(A\).

We next explain why the previously available methods do not suffice in this
generality.

\subsubsection{Geometric obstruction: full conformality}

A common strategy, which we call \emph{the geometric method}, is first to prove
\(C^{1+\text{H\"older}}\) regularity of the conjugacy along invariant Lyapunov
foliations and then to use Journ\'e's lemma to obtain global regularity.
Conformal or uniformly quasiconformal structures on the Lyapunov subspaces
play a crucial role in this approach, because they allow one to control the
growth rates of \(Df^n\) along the corresponding Lyapunov foliations. This
control leads to \(C^{1+\text{H\"older}}\) regularity of the conjugacy along those
foliations. To the best of the author's knowledge, previously known approaches to this
periodic-data rigidity problem rely, at the decisive first step, on such a
geometric mechanism.


Consequently, this approach is especially effective when the Lyapunov blocks
are one- or two-dimensional, or more generally when they carry invariant
conformal structures. In the general irreducible case, however, Lyapunov blocks
may have higher dimension and need not be conformal. To ensure the existence
of such conformal structures, one has to impose restrictive spectral
assumptions, for instance that the automorphism has no more than three
eigenvalues with the same modulus. Little was known without such geometric or
spectral assumptions.

\subsubsection{Analytic methods obstruction:  twisted negative-time series}\label{sec:32}

The main analytic ingredient in the proof concerns twisted cohomological
equations. After passing from tangent-bundle cocycles to the manifold-level
conjugacy equation, the component \(H_i=p_i\circ H\) of the Franks--Manning
conjugacy satisfies an equation of the form
\begin{align*}
 H_i=p_i+h_i,
        \qquad
        A_i h_i-h_i\circ f=R_i
\end{align*}
where \(A_i=A|_{E_i}\) and \(R_i\) is the \(E_i\)-component of the nonlinear
error. Thus the regularity of the conjugacy is reduced, in part, to the
regularity of solutions of a twisted cohomological equation.

For an unstable Lyapunov block, this equation gives the positive-time
representation
\[
        h_i=
        \sum_{m=0}^{\infty} A_i^{-(m+1)}R_i\circ f^m .
\]
This representation is well adapted to differentiating \(h_i\) along the
stable foliations of \(f\), because the positive iterates of \(f\) contract
stable directions. To obtain differentiability along the unstable foliations,
one would like to use instead the formal negative-time expression
\[
        h_i^-:=
        -\sum_{m=-1}^{-\infty} A_i^{-(m+1)}R_i\circ f^m .
\]
The desired strategy is first to prove that \(h_i^-\) is a well-defined
distribution and has the required distributional derivatives along unstable
directions, and then to prove that
\[
        h_i^-=h_i
        \qquad\text{as distributions}.
\]
Once this is achieved, differentiability in both stable and unstable
directions can be combined by an elliptic-regularity-type theorem for
distributions along foliations. This step generally needs
the error term \(R_i\) to have sufficiently high regularity. For this reason, existing analytic approaches usually require
\(f\in C^\infty(\TT^N)\), or at least sufficiently high regularity.

This type of analytic strategy is common in the study of \(C^\infty\) rigidity
for higher-rank actions \cite{FKS},\cite{KS94}, but its application to rank-one systems is not
standard in the literature. In higher-rank settings, one often has the freedom
to choose a group element for which the twist in the cohomological equation is
neutral, or nearly neutral, relative to the relevant directions. In that case,
convergence of the negative-time distribution can often be obtained from decay
of correlations.

For rank-one actions, this flexibility is absent. The twist \(A_i\) is tied to
a chosen Lyapunov block and in general it cannot be replaced by
a more favorable, nearly neutral twist. In
attempting to prove convergence of \(h_i^-\), one must balance the growth of
\[
        \bigl\|A_i^{-(m+1)}\bigr\|
        \qquad (m\to -\infty)
\]
against the decay of correlations for \(f\). For a nonlinear \(f\), the
available decay estimates are obtained through the H\"older conjugacy with
\(A\), and the resulting rate generally does not dominate the growth coming
from the twist. Even for the algebraic automorphism \(A\), standard exponential mixing
arguments compensate for the twist only when the observables have sufficiently
high regularity. The required order depends on \(A\), and in general
\(C^2\) regularity is far from sufficient.   Consequently, it is difficult in general to prove that
\(h_i^-\) is even a distribution by this method.

There is a further obstruction. Even if \(h_i^-\) is known to be a
distribution, proving its differentiability along unstable foliations by this
analytic approach remains problematic. One difficulty is that the
elliptic-regularity-type argument requires control of distributional
derivatives of high order. In the present setting, however, \(f\) is only
assumed to be \(C^2\), so the error term \(R_i\) has limited regularity. Even in the \(C^\infty\) setting, higher regularity of \(f\) alone does not
remove the obstruction caused by the twist. Existing techniques for overcoming this twist require an a priori
\(C^{1+\text{H\"older}}\) regularity assumption on the conjugacy \(H\). This
assumption is used to transfer growth estimates from the linear iterates
\(A^n\) to the nonlinear iterates \(Df^n\); see \cite{wang2}. However,
obtaining such growth estimates for \(Df^n\) is precisely one of the core
difficulties in proving \(C^1\) rigidity.

Thus the usual analytic approach does not provide a viable starting point for
rank-one rigidity in this setting. New ingredients, substantially different
from the standard higher-rank method, are needed to handle the negative-time
series and to prove the desired rigidity.

\subsubsection{The linear cocycle obstruction: non-abelian Liv\v{s}ic failure}
The main geometric ingredient in the proof concerns establishing linear
cocycle rigidity. The cocycle-theoretic input is closely related to
Liv\v{s}ic theory for linear cocycles. We recall this background in order to
explain why the argument cannot rely on a standard cocycle-rigidity theorem.

Let \(g:\mathcal X\to\mathcal X\) be a hyperbolic system and let \(G\) be a
topological group. Given two H\"older cocycles
\(\mathcal F_1,\mathcal F_2:\mathcal X\to G\), one asks whether they are
cohomologous, namely whether there exists a map
\(\mathcal C:\mathcal X\to G\) such that
\[
        \mathcal C(gx)\cdot \mathcal F_1(x)
        =
        \mathcal F_2(x)\cdot \mathcal C(x),
        \qquad x\in\mathcal X .
\]
A necessary condition is conjugacy of periodic data: for every periodic point
\(p=g^n p\), there should exist \(\mathcal C(p)\in G\) such that
\[
        \mathcal C(p)\cdot \mathcal F_1^n(p)
        =
        \mathcal F_2^n(p)\cdot \mathcal C(p),
\]
where
\[
        \mathcal F_i^n(p)=
        \mathcal F_i(g^{n-1}p)\cdot
        \mathcal F_i(g^{n-2}p)\cdots
        \mathcal F_i(p),
        \qquad i=1,2 .
\]
For abelian groups, the classical Liv\v{s}ic theorem gives a positive answer:
the periodic obstruction is sufficient for cohomology
\cite{livsic,livsic1}. The case in which one cocycle is trivial is also well
understood in several important settings, including compact groups
\cite{Parry} and some non-abelian groups \cite{K11,AA}.

The general non-abelian problem is much more delicate. It cannot be reduced to
the case of a trivial cocycle, and Liv\v{s}ic theory fails in this generality
even under boundedness assumptions on the conjugacies over periodic orbits
\cite{S13}. Moreover, even for cocycles that are close to one another,
conjugate periodic data need not imply continuous cohomology. Indeed, examples
of \(SL(3,\mathbb R)\)-valued cocycles with conjugate periodic data which are
not continuously cohomologous were constructed in \cite{GKS11}.

This obstruction is directly relevant to the present paper. In our setting,
one would like to compare the nonlinear cocycle \(Df|_{\mathcal E_i}\) with
the linear cocycle \(A_i\) over a Lyapunov block \(\mathcal E_i\). The
preceding discussion shows that conjugate periodic data are not sufficient, in
general, to obtain full cohomology of such cocycles. Thus full cocycle
rigidity on the whole Lyapunov block cannot be used as a black box.

The first new ingredient of the paper is therefore a partial cocycle rigidity
theorem. Although the full cocycle \(Df|_{\mathcal E_i}\) need not be
cohomologous to \(A_i\), we show that the periodic-data assumption implies
rigidity on a canonical conformal layer inside this block.

At first sight, obtaining such partial cocycle rigidity is not easier than
obtaining full cocycle rigidity. First, the conformal layer arises from an
abstract construction. Hence it is not a priori clear that it is close to any
\(A\)-invariant subspace, even when \(f\) is \(C^1\)-close to \(A\). This makes
it unclear where such a conformal subbundle should be sent by a transfer map,
and therefore makes it difficult even to formulate the relevant linear cocycle
equation. Second, as explained above, standard Liv\v{s}ic theory does not
extend to general linear cocycles. Thus it is not reasonable to expect to solve
a cohomology equation without first identifying an appropriate solvability
condition.

Consequently, a new scheme, substantially different from the traditional
cocycle-rigidity approach, is needed. The partial rigidity result obtained here
is the starting point of the partial-to-global mechanism used in the current
paper.

\subsection{The new partial-to-global mechanism}
We introduce
a different mechanism for proving periodic-data rigidity, which combines geometric and analytic methods. The phrase
``partial-to-global'' refers to two distinct transitions. The first transition
is from partial linear cocycle rigidity to partial differentiability of the
conjugacy. More precisely, we prove cocycle rigidity for \(Df\) only on a
canonical conformal subbundle, and then show that this subbundle integrates to
a foliation along which \(H\) is \(C^{1+\text{H\"older}}\). This is the main
geometric input of the paper. Up to this stage, the proof uses only the
conformality of the subbundle, the periodic-data assumption, and the
corresponding conformal structure for \(A\); irreducibility of \(A\) is not
used.

The second transition occurs only after passing from tangent-bundle cocycles
to the manifold-level cohomological equation for the component
\(H_i=p_i\circ H\) of the conjugacy. At this level, irreducibility of \(A\)
enters. The first transition gives partial \(C^{1+\text{H\"older}}\)
regularity of \(H^{-1}\) along a linear foliation tangent to an
\(A\)-invariant subspace \(W\subset E_i\). Irreducibility of \(A\) implies
that \(W\) is Diophantine. This Diophantine property controls the small
divisors that arise when one integrates by parts against Fourier modes. Since
the small-divisor loss is only polynomial, the contraction of \(A^{-n}\) along
the same foliation gives the exponential gain needed to control the formal
negative-time series.
This allows us to identify
\[
        h_i^-=h_i
        \qquad\text{as distributions}.
\]
Differentiating this identity along an arbitrary H\"older vector field
\(\mathcal V\) tangent to \(\mathcal E_i\), we obtain
\[
        D_{\mathcal V}h_i=D_{\mathcal V}h_i^-
\]
as distributions. Thus the \(\mathcal E_i\)-directional distributional derivatives of \(h_i\)
have both positive- and negative-time representations, not merely the
representation coming from the initial conformal subbundle.

The remaining analytic step is to upgrade this distributional information to
H\"older regularity. Using the triangular reduction of
\(Df|_{\mathcal E_i}\), the positive-time representation gives translation
estimates in stable directions, while the negative-time representation gives
the corresponding estimates in unstable directions. Since the two
distributional derivatives above are equal, these one-sided estimates combine
to give the two-sided estimates required by the distribution-to-H\"older
criterion. Hence \(D_{\mathcal V}h_i\) is represented by a H\"older map. Since
\(\mathcal V\) is arbitrary in \(\mathcal E_i\), it follows that \(h_i\), and
therefore the corresponding component \(H_i\), is
\(C^{1+\text{H\"older}}\) along the whole Lyapunov foliation
\(\mathcal W_i^f\). This is the analytic promotion from the initial conformal
layer to the entire Lyapunov block.

\section{Proof strategy}

We now describe the main ideas in the proof of Theorems~\ref{th:6} and
\ref{th:5}. We first reduce the problem to leafwise regularity. Let
\[
        E_A^u=E_{i_0}\oplus E_{i_0+1}\oplus\cdots\oplus E_\ell
\]
be the unstable subspace of \(A\), decomposed according to the moduli of its
eigenvalues, and let
\[
        \mathcal E_f^u=
        \mathcal E_{i_0}\oplus \mathcal E_{i_0+1}
        \oplus\cdots\oplus \mathcal E_\ell
\]
be the corresponding dominated splitting for the unstable subspace of \(f\).

The key step in the proof is the following. Suppose that
\[
        H(\mathcal W_i^f)=\mathcal W_i^A
        \qquad\text{for some } i_0\leq i\leq\ell .
\]
Then \(H\) is \(C^{1+\text{H\"older}}\) along \(\mathcal W_i^f\). Since
\(H_i=p_i\circ H=p_i+h_i\), this amounts to proving that the solution \(h_i\)
of the twisted cohomological equation
\[
        A_i h_i-h_i\circ f=R_i
\]
is \(C^{1+\text{H\"older}}\) along \(\mathcal W_i^f\). As explained in
Section~\ref{sec:32}, the proof has two main parts.

\smallskip

\noindent\(\hypertarget{o.1}{(\mathcal P_1)}\)
We prove that \(h_i^-\) is a well-defined distribution and that
\[
        h_i^-=h_i
        \qquad\text{as distributions}.
\]

\smallskip

\noindent\(\hypertarget{o.2}{(\mathcal P_2)}\)
For every \(\alpha\)-H\"older vector field \(\mathcal V\) taking values in
\(\mathcal E_i\), we prove that the distributional derivative
\[
        D_{\mathcal V}h_i=D_{\mathcal V}h_i^-
\]
is represented by a H\"older map. Therefore \(h_i\) is
\(C^{1+\text{H\"older}}\) along \(\mathcal W_i^f\).

The proof has three main stages. First, we obtain partial cocycle rigidity on
a canonical conformal layer inside a Lyapunov block. Second, we convert this
partial cocycle rigidity into partial differentiability of the conjugacy and
use irreducibility to control the negative-time distributional series. Third,
we use a distribution-to-H\"older argument and an induction over the canonical
flag to promote this partial regularity to full
\(C^{1+\text{H\"older}}\) regularity along the whole Lyapunov block. Repeating
this over all Lyapunov blocks and applying Journ\'e's lemma gives the global
result.

\subsection{Proof of \(\mathcal P_1\)} In this part, we explain how \hyperlink{o.1}{\((\mathcal P_1)\)} is proved.

\subsubsection{Establishing partial linear cocycle rigidity inside a subbundle of \(\mathcal E_i\)}

This part summarizes the contents of Sections~\ref{sec:33} and~\ref{sec:7}.

Although the full cocycle \(Df|_{\mathcal E_i}\) need not be conformal and, in
general, cannot be cohomologous to the linear cocycle \(A_i\), there is a
canonical invariant flag
\[
        0=\mathcal F_{i,0}\subset \mathcal F_{i,1}\subset\cdots
        \subset \mathcal F_{i,j_i}=\mathcal E_i
\]
such that each quotient \(\mathcal F_{i,j+1}/\mathcal F_{i,j}\) carries an
invariant conformal structure; see \eqref{for:33} of
Section~\ref{sec:12}. The first step is to prove cocycle rigidity on the first
conformal layer. More precisely, for a suitable \(A\)-invariant subspace
\(V_{i,1}\subset E_i\), we construct a H\"older bundle isomorphism $\mathcal C_{i,1}:\mathcal F_{i,1}\longrightarrow V_{i,1}$
satisfying
\begin{align}\label{for:106}
 \mathcal C_{i,1}(fx)\circ Df_x|_{\mathcal F_{i,1}(x)}
        =
        A_i|_{V_{i,1}}\circ \mathcal C_{i,1}(x).
\end{align}
The method is different from the standard cocycle-rigidity approach. Instead
of prescribing \(V_{i,1}\) in advance and trying to solve a linear cocycle
cohomology equation with values in \(GL(V_{i,1})\), we write
\[
        \mathcal C_{i,1}(x)
        =
        p_i|_{\mathcal F_{i,1}(x)}+q(x)
\]
and reduce the problem to the twisted cohomological equation
\[
        A_iq(x)-q(fx)\circ Df_x|_{\mathcal F_{i,1}(x)}
        =
        (\mathcal R_i)_{\mathcal F_{i,1}(x)} .
\]
This formulation avoids the need to know the target subspace \(V_{i,1}\) in
advance. Once \(q\) is constructed, the subspace \(V_{i,1}\) is obtained as the
image of \(\mathcal C_{i,1}\).

The construction of \(q\) is explicit. The periodic-data assumption allows us
to start from a solution at a fixed point. Stable and unstable holonomies then
extend this solution along the corresponding invariant manifolds, and
conjugacy of periodic data implies that the stable and unstable extensions
agree on homoclinic points. This gives a global H\"older solution of the
partial cocycle equation; see the proof of Proposition~\ref{po:5}.

We then apply a similar construction to each conformal quotient
\(\mathcal F_{i,j+1}/\mathcal F_{i,j}\). This yields a bundle isomorphism
\[
        \mathcal C_i(x):\mathcal E_i(x)\to E_i
\]
under which \(Df|_{\mathcal E_i}\) takes a block upper triangular form
\(\widetilde A_i\) at every \(x\in\mathbb T^N\), with diagonal blocks
coinciding with the corresponding diagonal blocks of \(A_i\); see
Theorem~\ref{th:9}.

\subsubsection{Establishing partial regularity of \(H\)}

This part summarizes the contents of Sections~\ref{sec:34} and~\ref{sec:35}.
In this step, we convert partial cocycle rigidity into partial regularity of the conjugacy. More precisely, we show that \(\mathcal F_{i,1}\) is uniquely integrable to a foliation with uniformly \(C^{1+\text{H\"older}}\) leaves, denoted by \(\mathcal W_{\mathcal F_{i,1}}\), and that \(H\) is \(C^{1+\text{H\"older}}\) along \(\mathcal W_{\mathcal F_{i,1}}\). Moreover, if \(W\subset E_i\) is the corresponding \(A\)-invariant subspace and \(W^L\) denotes the linear foliation tangent to \(W\),
\[
        H\bigl(\mathcal W_{\mathcal F_{i,1}}\bigr)=W^L .
\]
The precise statement is given in Theorem~\ref{th:1}.

Since the integrability of \(\mathcal F_{i,1}\) is not known a priori, we
cannot begin with leafwise differentiability. Instead, we introduce the notion
of curve differentiability; see Section~\ref{sec:36}. We prove that, for any
\(C^{1+\text{H\"older}}\) curve \(\gamma(t)\) satisfying
\(\gamma'(t)\in\mathcal F_{i,1}\), there exists \(\delta>0\) such that, for
every \(C^1\) function \(\omega\) compactly supported in
\((-\delta,\delta)\),
\begin{align}\label{for:103}
 \int_{-\delta}^{\delta} H\circ\gamma(t)\omega'(t)\,dt
 =
 -\int_{-\delta}^{\delta}
 \mathcal D_i(\gamma(t))(\gamma'(t))\omega(t)\,dt .
\end{align}
Here \(\mathcal D_i\) is \(\mathcal C_{i,1}\), up to a constant linear
isomorphism.

The main difficulty in proving \eqref{for:103} is that the relevant estimates
come from the partial cocycle rigidity relation for \(\mathcal C_{i,1}\), while
the expression above is supported on a one-dimensional curve. Formally, one
would like to differentiate the positive-time series for \(H_i=p_i+h_i\) along
\(\gamma\). Since
\(\gamma'(t)\in\mathcal F_{i,1}(\gamma(t))\), the partial cocycle rigidity
relation \eqref{for:106} should control the resulting derivatives. However, the series is
difficult to estimate directly on the curve. The key idea is therefore to
approximate the curve distribution by normalized averages over small plaques
in a foliation box. This converts the curve-level problem into an averaged
problem on \(\mathbb T^N\), where the previously established estimates can be
applied.

Identity \eqref{for:103} implies that \(DH|_{\mathcal F_{i,1}}\) is given,
up to a constant linear isomorphism, by \(\mathcal C_{i,1}\). It follows that
\(\mathcal F_{i,1}\) is uniquely integrable and that \(H\) is
\(C^{1+\text{H\"older}}\) along the resulting foliation.

Up to this point, the argument works at the tangent-bundle level and does not
use irreducibility of \(A\). The subspace \(W\) obtained here will be used in
the next step: irreducibility of \(A\) implies that \(W\) has the Diophantine
property needed to control the negative-time distribution.

\subsubsection{Identifying \(h_i^-\) with \(h_i\)}

This part summarizes the contents of Section~\ref{sec:37}. The negative-time
series \(h_i^-\) is handled using the partial \(C^{1+\alpha}\) regularity of
\(H^{-1}\) along the linear foliation tangent to the subspace \(W\subset E_i\)
obtained above. The Diophantine property of \(W\) then allows us to integrate
by parts against Fourier modes and to control the small divisors that appear.
This shows that the negative endpoint in the telescoping argument tends to
zero. Consequently,
\[
        h_i^-\circ H^{-1}=h_i\circ H^{-1} \qquad\text{as distributions}.
\]
Changing variables back through \(H\), we obtain
\[
        h_i^-=h_i
        \qquad\text{as distributions}.
\]
This is one place where irreducibility of \(A\) is used in an essential way: it guarantees the Diophantine property for the invariant subspaces that arise in the proof. A new feature of the argument is that the decay needed to control the formal negative-time expression is not obtained from standard exponential mixing estimates for toral automorphisms. Instead, partial \(C^{1+\alpha}\) regularity along a Diophantine linear foliation is combined with a Fourier integration-by-parts argument. The Diophantine property controls the small divisors with only polynomial loss, while the contraction of \(A^{-n}\) along the same foliation provides the exponential gain needed to realize the formal negative-time expression as a distribution and identify it with \(h_i\circ H^{-1}\).

\subsection{Proof of $\mathcal{P}_2$} Section~\ref{sec:38} proves \(\mathcal P_2\). We take distributional
derivatives along an \(\alpha\)-H\"older vector field \(\mathcal V\) taking
values in \(\mathcal E_i\). The equality \(h_i^-=h_i\) implies
\[
        D_{\mathcal V}h_i
        =
        D_{\mathcal V}h_i^-
\]
as distributions. Hence the derivative of \(h_i\) along the full Lyapunov
block \(\mathcal E_i\) has two distributional representations: a positive-time
one and a negative-time one. More precisely,
\[
        \lim_{m\to\infty}
        A_i^{-m}p_i\circ Df_z^m(\mathcal V_z)
        =
        p_i(\mathcal V_z)+D_{\mathcal V}h_i
\]
and
\[
        \lim_{m\to-\infty}
        A_i^{-m}p_i\circ Df_z^m(\mathcal V_z)
        =
        p_i(\mathcal V_z)+D_{\mathcal V}h_i^-
\]
in the sense of distributions.

The triangular reduction of \(Df|_{\mathcal E_i}\) is then used inductively to
control these two limits. The positive-time expression gives stable-direction
translation estimates, while the negative-time expression gives
unstable-direction translation estimates. After conjugating by \(H\), these
become translation estimates along the linear stable and unstable foliations
of \(A\). Since the positive- and negative-time distributions are equal, these
one-sided estimates combine to give the two-sided estimates required by the
distribution-to-H\"older criterion. Therefore \(D_{\mathcal V}h_i\) is represented by a H\"older map. Since
\(\mathcal V\) was arbitrary in \(\mathcal E_i\), this gives
\(C^{1+\text{H\"older}}\) regularity of \(H\) along \(\mathcal W_i^f\).

\subsection{Scope of the method}

The method developed in this paper may be viewed as an important step toward the broader
rank-one rigidity program, including both periodic-data rigidity and Lyapunov
rigidity on general nilmanifolds. A notable feature of the argument is that it does not rely on a full
conformal structure. Instead, it combines partial cocycle rigidity with an
analytic promotion mechanism for twisted cohomological equations. This provides
a new perspective on rigidity problems in which conformality or uniform
quasiconformality is not available.

In the nilmanifold setting, however, new difficulties arise. Existing rigidity
results often require stronger spectral assumptions, such as simplicity of the
spectrum \cite{Witt}. One reason is that, when the spectrum is not simple, the
intermediate foliations used in the toral case may no longer be available in a
natural form. Since these intermediate foliations are the starting point of the
present argument, this creates a genuine obstruction to a direct extension of
the method.

The curve-differentiability approach introduced here may provide a way to
address this difficulty. It allows one to extract differentiability information
from distributional identities without assuming, at the outset, the existence
of all the intermediate smooth foliations that appear in the toral proof. This
suggests that the partial-to-global mechanism developed in this paper could be
useful in future work on rigidity problems beyond tori.

The same mechanism may also be relevant to rigidity problems for higher-rank
actions. In higher rank, several known rigidity results rely on semisimplicity
assumptions; see, for example, \cite{GKS23,KS97}. The strategy developed here,
which combines partial cocycle rigidity with an analytic promotion argument
for twisted cohomological equations, suggests a possible route toward treating
more general nonsemisimple actions.

\smallskip

\noindent{\bf Acknowledgements.}
The author is deeply grateful to A. Gogolev for numerous clarifying
conversations. The author is also grateful to B. Kalinin and V. Sadovskaya for
referring her to their paper on partial conformal structures. The author thanks
J. DeWitt, F. Yang, and H. Hu for many helpful conversations.





\section{Preliminaries}\label{sec:5}
We work in one of the following two settings:
\begin{enumerate}
  \item \emph{Global setting.} Let $f$ be a $C^2$ Anosov diffeomorphism on
  $\mathbb T^N$, and let $A\in GL(N,\mathbb Z)$ be its linearization.

\smallskip
  \item \emph{Local setting.} Let $A\in GL(N,\mathbb Z)$ be an Anosov
  automorphism of $\mathbb T^N$, and let $f$ be a
  $C^{1+\text{H\"older}}$ diffeomorphism on $\mathbb T^N$ which is a
  $C^1$-small perturbation of $A$.
\end{enumerate}
In either case, we assume that $A$ is irreducible and that $f$ and $A$ have
the same periodic data. Let
\[
\rho_1 < \cdots <\rho_{i_0-1}<1<\rho_{i_0}<\cdots \rho_\ell
\]
be the distinct moduli of the eigenvalues of \(A\) and  let
\begin{align}\label{for:2}
 \mathbb{R}^N = E_1 \oplus \cdots \oplus E_\ell
\end{align}
be the corresponding \(A\)-invariant splitting.

Then there is a bi $\eta$-H\"older conjugacy $H$  between $f$ and $A$ \cite{Franks}, i.e.,
\begin{align}\label{for:140}
 A\circ H= H \circ f.
\end{align}
By \cite{W}, any two such conjugacies differ by an affine automorphism of
\(\mathbb T^N\) commuting with \(A\). Consequently, the regularity of one such
conjugacy implies the same regularity for all of them.

In what follows, $C$ will denote any constant that depends only
  on the given  action  $f$ and the manifolds $\TT^N$.   $C_{x,y,z,\cdots}$ will denote any constant that in addition to the
above depends also on parameters $x, y, z,\cdots$. We point out that the constant $C$ has been changing throughout the proof.

\subsection{Same periodic data}\label{sec:12}

We list several consequences of the assumption that $f$ and $A$ have the same
periodic data. We use the same notation in both the global and local
settings. Items~\ref{for:26}--\ref{for:86} are consequences of \cite{WA} in
the global setting and of \cite{GKS11} in the local setting. The flag statement
in item~\ref{for:33} follows from Theorems~3.9 and~3.10 of \cite{KS13}.

\begin{enumerate}
  \item\label{for:26} \emph{Dominated splitting.} There exists $\alpha>0$ such
that $f$ admits a $Df$-invariant splitting
\[
        T\mathbb T^N=\mathcal E_1\oplus\cdots\oplus\mathcal E_\ell
\]
into $\alpha$-H\"older continuous subbundles. Moreover, for every sufficiently
small $\epsilon>0$, there is a H\"older metric on each $\mathcal E_i$ such
that, for every $v\in\mathcal E_i$,
\begin{align}\label{for:129}
 \rho_ie^{-\epsilon} \|v\|
\leq
\|Df v\|
\leq
\rho_ie^{\epsilon} \|v\|.
\end{align}

\item\label{for:23} \emph{Unique integrability of the bundles $\mathcal E_i$.} Each $\mathcal E_i$
is uniquely integrable to a foliation with uniformly $C^{1+\alpha}$ leaves,
which we denote by $\mathcal{W}^f_i$. The weak flag of bundles
\begin{align*}
 \mathcal E_{i, j}
=
\mathcal E_{i}\oplus\cdots\oplus\mathcal E_j,\quad i\geq i_0
\end{align*}
is uniquely integrable to a weak foliation with uniformly
$C^{1+\alpha}$ leaves, which we denote by
$\mathcal{W}^f_{i,j}$.

\smallskip
  \item\label{for:87} \emph{Intertwining of foliations.} For each $i\geq i_0$, the conjugacy $H$ maps the foliation
  $\mathcal W^f_{i_0,i}$ to the linear foliation $\mathcal W^A_{i_0,i}$ tangent to
  $E_{i_0}\oplus \cdots \oplus E_i$.

  \smallskip

  \item \label{for:86} \emph{Intertwining of foliations under extra assumption.} Fix \(i_0\leq i<j\leq \ell\). Suppose that:
  \begin{enumerate}
  \item $A$ is irreducible.

    \item $H(\mathcal W^f_{i,j})=\mathcal W_{i,j}^A$, where $\mathcal W_{i,j}^A$ denotes the linear foliation tangent to  $E_{i}\oplus \cdots \oplus E_j$,
        \smallskip

            \item $H(\mathcal{W}^f_i) = \mathcal{W}^A_i$ and $H$ is a $C^{1+\text{H\"older}}$ diffeomorphism along $\mathcal{W}^f_{i}$, where $\mathcal W_{i}^A$ denotes the linear foliation tangent to  $E_{i}$.
  \end{enumerate}
  Then $H(\mathcal W^f_{i+1,j})=\mathcal W_{i+1,j}^A$.

\smallskip
  \item\label{for:33} \emph{Flags of $\alpha$-H\"older  $Df$-invariant sub-bundles}. Suppose $\mathcal{S}$ is a $Df$-invariant subbundle.
For every $f$-periodic point $p$ denote by $\mu_p$ the invariant measure
 on its orbit. Set
\begin{align*}
\lambda_{+,\mathcal{S}}(Df,\mu_p)&=\lim_{n\to\infty}\frac{\log\norm{Df^n|_{\mathcal{S}_p}}}{n}
\qquad\text{and}\\
\lambda_{-,\mathcal{S}}(Df,\mu_p)&=\lim_{n\to\infty}\frac{\log\norm{Df^{-n}|_{\mathcal{S}_p}}^{-1}}{n}.
\end{align*}
Since $f$ and $A$ have the same
periodic data, $Df$ has \emph{constant periodic data} with exponent $\log \rho_i$ along $\mathcal E_i$, i.e.,
\begin{align*}
 \lambda_{+,\mathcal E_i}(\rho_i^{-1}df,\mu_p)=\lambda_{-,\mathcal E_i}(\rho_i^{-1}Df,\mu_p)=0
\end{align*}
for every periodic point $p$. Then it follows from Theorems~3.9 and~3.10 of \cite{KS13} that the following hold:
  there exists a flag of $\alpha$-\text{H\"older} $Df$-invariant
  subbundles
  \begin{equation*}
  \{0\}
  =
  \mathcal F_{i,0}
  \subset
  \mathcal F_{i,1}
  \subset
  \cdots
  \subset
  \mathcal F_{i,j_i}
  =
  \mathcal E_i
  \end{equation*}
  and $\alpha$-\text{H\"older} Riemannian metrics on the quotient bundles
  \[
  \mathcal F_{i,m}/\mathcal F_{i,m-1},
  \qquad m=1,\ldots,j_i,
  \] such that, for some positive $\alpha$-H\"older function
$\phi_i : \TT^N \to \RR$ the quotient-cocycles induced by the cocycle
$\phi_i Df$ on
  $\mathcal F_{i,m}/\mathcal F_{i,m-1}$ are isometries. Moreover, for any $x\in \TT^N$
\begin{align*}
 \norm{Df^n|_{\mathcal F_{i,m}}}\leq C\rho_i^n|n|^{m-1},\qquad \forall\,1\leq m\leq j_i,\,\,\forall\,n\in\ZZ\backslash\{0\}.
\end{align*}

\end{enumerate}

\subsection{Stable/unstable foliations and closing lemma}\label{sec:8} Since $A$ is diagonalizable over $\CC$ and all eigenvalues of $A$ on $E_i$ have modulus $\rho_i$, we have
\begin{align}\label{for:93}
 \norm{A^n|_{E_i}}\leq C\rho_i^n, \qquad \forall\,n\in\ZZ.
\end{align}
Define
\[
 E^{s,A}=\oplus_{\rho_i<1}E_i,\qquad
E^{u,A}=\oplus_{\rho_i>1} E_i.
\]
Denote the corresponding linear foliations by $W^{s,A}$ and  $W^{u,A}$ respectively, which are called stable and unstable foliations. Set
\begin{align*}
 \nu_0=\max \{\rho_{i_0-1},\rho_{i_0}^{-1}\}^{\frac{1}{2}},\qquad \nu_0<\nu<1.
\end{align*}
Then for $n\geq0$ we have
\begin{align}\label{for:77}
 \norm{A^n|_{E^{s,A}}}\leq C\nu_0^{n},\qquad \norm{A^{-n}|_{E^{u,A}}}\leq C\nu_0^{n}.
\end{align}
Similarly, define
\[
\mathcal E^{s,f}=\oplus_{\rho_i<1}\mathcal E_i,\qquad
\mathcal E^{u,f}=\oplus_{\rho_i>1}\mathcal E_i.
\]
Denote the corresponding foliations by $\mathcal{W}^{s,f}$ and  $\mathcal{W}^{u,f}$ respectively, which are called stable and unstable foliations.  The foliations
$\mathcal W^{s,f}$ and $\mathcal W^{u,f}$ have uniformly $C^{1+\alpha}$ leaves.

By \eqref{for:129}, for any \(v\in\mathcal E^{s,f}(x)\) and \(n>0\), or any
\(v\in\mathcal E^{u,f}(x)\) and \(n<0\), we have
\begin{align*}
 \norm{Df^n_x(v)}\leq C\nu^{|n|}\norm{v}.
\end{align*}
We denote by $\mathcal{W}^{s,f}_{loc}(x)$ (resp. $\mathcal{W}^{u,f}_{loc}(x)$) a sufficiently
small ball around $x$ in the leaf $\mathcal{W}^{s,f}(x)$ (resp. $\mathcal{W}^{u,f}(x)$).
   For any $y\in \mathcal{W}^{s,f}_{loc}(x)$ (resp. $y\in \mathcal{W}^{u,f}_{loc}(x)$) and any $n\geq0$ (resp. $n\leq0$) we have
\begin{align}\label{for:130}
 &d(f^nx,\,f^ny)\leq C\nu^nd(x,y)\qquad \big(\text{resp. } d(f^nx,\,f^ny)\leq C\nu^{-n}d(x,y)\big).
\end{align}
We have a local product structure: there is $\delta>0$ such that for any $x,\,y\in \TT^N$ with $d(x,y)\leq \delta$, there is $z\in \mathcal{W}^{s,f}_{loc}(x)\bigcap \mathcal{W}^{u,f}_{loc}(y)$ satisfying
\begin{align}\label{for:44}
 d(x,z)+d(y,z)<Cd(x,y).
\end{align}
The next result is a standard fact for Anosov diffeomorphisms \cite{KH}.
\begin{theorem}\label{th:2} (Closing Lemma)  There exists constant $\delta>0$  such that for any $x\in\TT^N$ and
$n\in\NN$, if $d(x,f^nx)<\delta$, then there exists a periodic point $p\in \TT^N$ with $f^np=p$ such that
\begin{align*}
  d(f^ix,f^ip)\leq Cd(x,f^nx)\nu^{\min\{i,n-i\}},\qquad \forall\,0\leq i\leq n.
\end{align*}

\end{theorem}

\subsection{Invariant measure}\label{sec:13}  Let \(\mathfrak m\) denote
    Lebesgue measure on \(\mathbb T^N\). Let $\mu:=(H^{-1})_*\mathfrak m$.
As $H$ is a topological conjugacy between $f$ and $A$, the measure
$\mu$ is $f$-invariant.

  In fact, $\mu$ is the
Bowen-Margulis measure of maximal entropy for $f$, since $\mathfrak m$ is the
Bowen-Margulis measure of maximal entropy for $A$. Indeed, denoting
topological entropy by $h_{\mathrm{top}}$ and metric entropy with respect
to $\mu$ by $h_{\mu}$, we have
\[
h_{\mu}(f)
=
h_{\mathfrak m}(A)
=
h_{\mathrm{top}}(A)
=
h_{\mathrm{top}}(f).
\]
In particular, $\mu$ is ergodic, has full support, and has
local product structure.

 \eqref{for:129} of Section \ref{sec:12} implies that the Lyapunov exponents
$l_i^{f,\mu}$ of $\mu$ for
$f$ coincide with the Lyapunov exponents $l_i^A$ of $A$.
Therefore, the sum of the positive Lyapunov exponents, counted with
multiplicity, is equal to the entropy:
\[
h_{\mu}(f)
=
h_{\mathfrak{m}}(A)
=
\sum_{l_i^A>0} l_i^A
=
\sum_{l_i^{f,\mu}>0} l_i^{f,\mu}.
\]
Thus equality holds in Ruelle's inequality, i.e. \(\mu\) satisfies Pesin's
entropy formula. It follows that
$\mu$ has absolutely continuous conditional measures on the
unstable foliation of $f$ \cite{L}. Applying the same argument
to $f^{-1}$, we obtain absolutely continuous conditional measures on the stable
foliation of $f$. Since $\mu$ has local product structure, we conclude that
$\mu$ itself is absolutely continuous with respect to $\mathfrak m$.

Moreover, the density $\kappa(x):=\frac{d\mu}{d\mathfrak{m}}(x)$
is $C^{\alpha}$ and positive. Indeed, $\kappa$
satisfies the measurable coboundary equation
$\frac{\kappa(fx)}{\kappa(x)}=|\det Df(x)|^{-1}.$
By measurable Liv\v{s}ic regularity, $\kappa$ has a $C^{\alpha}$ positive
representative. Consequently, $\mu$ is equivalent to $\mathfrak m$.

\subsection{Decomposed twisted cohomological equation}\label{sec:14}

After changing coordinates, we may assume that \(0\) is a common fixed point of
both \(f\) and \(A\). We choose the unique conjugacy \(H\) in the homotopy class
of the identity satisfying \(H(0)=0\).

Lift \(f\) and \(H\) to maps on \(\mathbb R^N\), denoted by \(\bar f\) and
\(\bar H\), such that
\[
\bar H=\operatorname{Id}+\bar h,
\qquad
\bar f=A+\bar R,
\]
where \(\bar h,\bar R:\mathbb R^N\to\mathbb R^N\) are
\(\mathbb Z^N\)-periodic functions satisfying $\bar h(0)=0$ and $\bar R(0)=0$.

The lifted conjugacy equation $A\circ \bar H=\bar H\circ \bar f$ gives
\[
A\bar h-\bar h\circ \bar f=\bar R.
\]
This equation projects to the following equation on \(\mathbb T^N\):
\begin{equation}\label{for:78}
Ah-h\circ f=R,
\end{equation}
where \(h,R:\mathbb T^N\to\mathbb R^N\) satisfy $h(0)=0$ and $R(0)=0$.  Moreover, \(R\) is \(C^{1+\alpha}\), while \(h\) is $C^\beta$ continuous.

Let \(p_i:\mathbb R^N\to E_i\) denote the projection associated to the
splitting \eqref{for:2}.
Set
\[
h_i=p_i h,
\qquad
R_i=p_i R,
\qquad
A_i=A|_{E_i}.
\]
Since the splitting is \(A\)-invariant, equation \eqref{for:78} decomposes into
the twisted cohomological equations
\begin{equation}\label{for:156}
A_i h_i-h_i\circ f=R_i,
\qquad 1\leq i\leq \ell.
\end{equation}
If $i\geq i_0$, we have
\begin{align}\label{for:92}
 h_i=\sum_{m=0}^{\infty} A_i^{-(m+1)}R_i\circ f^m.
\end{align}
\section{Rigidity for normalized cocycles}\label{sec:33}
In this section, we prove an abstract cocycle rigidity criterion for a
$Df$-invariant H\"older subbundle.  We recall that $0$ is a fixed point for $f$; see Section \ref{sec:14}.
\begin{proposition}\label{po:5} Suppose $\mathcal{L}$ is a $Df$-invariant $\alpha$-H\"older subbundle over $\TT^N$ and $E$ is a subspace of $\RR^{N}$ with $\dim E=\dim \mathcal{L}$.  Let $\mathcal A:E\to E$ be a linear automorphism. Suppose that there exist $\rho>0$ and a
Riemannian metric $\|\cdot\|$ on $\mathcal L$ such that the following
conditions hold.
\begin{enumerate}
\item\label{for:124} $\mathcal A$ is diagonalizable over $\CC$ and all eigenvalues of $\mathcal A$ have modulus $\rho$;

\smallskip
  \item\label{for:126} At the fixed point $0$ of $f$, we have
  \begin{align*}
   P_{E}Df_0|_{\mathcal{L}_0}=\mathcal{A}P_E|_{\mathcal{L}_0},
  \end{align*}
  where $P_E$ denotes the projection from $\RR^N$ to $E$. Moreover, $P_E:\mathcal{L}_0\to E$ is an isomorphism.


  \item\label{for:123} For every $n\in\mathbb Z$ and every $x\in\mathbb T^N$, we have
  \[
  \bigl\|\rho^{-n}Df_x^n|_{\mathcal L_x}\bigr\|\leq C.
  \]
  \item\label{for:88} There exists a sequence $k_n\to +\infty$ such that:
\begin{enumerate}
  \item\label{for:36} $\rho^{-k_n}\mathcal{A}^{k_n}\to I_{id}|_{E}$ and $\rho^{k_n}\mathcal{A}^{-k_n}\to I_{id}|_{E}$ as $n\to +\infty$;

  \smallskip
  \item\label{for:28} for any $\epsilon>0$ there is $l_\epsilon\in\NN$ such that: for any $k_n$ with $n\geq l_\epsilon$ and any $x\in\TT^N$ with $f^{2k_n}x=x$ we have
  \begin{align*}
   \Big\|\rho^{-k_n}Df^{k_n}_{f^{k_n}x}|_{\mathcal{L}_{f^{k_n}(x)}}-\rho^{k_n}Df^{-k_n}_{f^{k_n}x}|_{\mathcal{L}_{f^{k_n}(x)}}\Big\|\leq \epsilon.
  \end{align*}
  \end{enumerate}

\end{enumerate}
Then the linear cocycle equation
\begin{align}\label{for:9}
 \mathcal{K}(fx)\circ Df_x|_{\mathcal{L}_x}=\mathcal{A}\circ \mathcal{K}(x),\qquad \forall\,x\in \TT^N
\end{align}
has an $\alpha$-H\"older solution $\mathcal{K}$ such that $\mathcal{K}(x):\mathcal{L}_x\to E$ is an isomorphism for any $x\in \TT^N$.  Moreover,
\begin{align*}
 \norm{(\mathcal{K}(x))^{-1}u}\leq C\norm{u},\qquad \forall\,x\in \TT^N,\,\,\forall\,u\in E.
\end{align*}
In particular, $\mathcal K$ is an $\alpha$-H\"older bundle isomorphism with
$\alpha$-H\"older inverse.
\end{proposition}

\subsection{Role of Proposition~\ref{po:5}} Proposition~\ref{po:5} will be used to prove Theorem~\ref{th:9}. More
precisely, Proposition~\ref{po:5} produces H\"older conjugacies for the
quotient cocycles. These quotient conjugacies are then assembled to obtain
the block upper-triangular reduction of $Df$ over the Lyapunov block
$\mathcal E_i$.

The hypotheses are formulated in terms of
uniform boundedness of the normalized cocycle and a periodic approximation
condition. In the proof of Theorem~\ref{th:9}, these hypotheses will be
verified for the quotient cocycles $\mathcal F_{i,j}/\mathcal F_{i,j-1}$,
which carry invariant conformal structures (see \eqref{for:33} of  Section~\ref{sec:12}).

\subsection{Proof strategy}\label{sec:2}

We begin by summarizing the main steps in the proof. We first reduce the linear cocycle equation \eqref{for:63} to the
twisted cohomological equation \eqref{for:49}. More precisely, we look for a
solution of \eqref{for:63} in the form
\[
        \mathcal K_x=P_E|_{\mathcal L_x}+q(x),
\]
which leads to the equation
\[
        \mathcal A\circ q(x)-q(fx)\circ Df_x|_{\mathcal L_x}
        =
        \mathbf r(x).
\]
At the fixed point \(0\), condition~\eqref{for:126} gives
\(\mathbf r(0)=0\). Hence the finite positive- and negative-time sums used to
define \(q^+\) and \(q^-\) vanish at \(0\); see Remark~\ref{re:3}.

We then construct two one-sided solutions. The positive-time series
defines a solution $q^+$ on the stable leaf $\mathcal W^{s,f}(0)$, while the
negative-time series defines a solution $q^-$ on the unstable leaf
$\mathcal W^{u,f}(0)$; see Section \ref{sec:10}. We also prove that $q^+$ is
H\"older along $\mathcal W^{s,f}(0)$ and that $q^-$ is H\"older along
$\mathcal W^{u,f}(0)$; see Section \ref{sec:4}.

 The next step is to prove that the two one-sided solutions agree on the
homoclinic set
\[
        \mathcal S_0
        =
        \mathcal W^{s,f}(0)\cap \mathcal W^{u,f}(0).
\]
This is where the periodic-data assumption is used, through the closing
argument and condition~\eqref{for:88}; see Section \ref{sec:21}. Once
$q^+=q^-$ on $\mathcal S_0$, we obtain a well-defined H\"older section
$q$ on $\mathcal S_0$; see Section \ref{sec:22}.

 Since $\mathcal S_0$ is dense in $\mathbb T^N$, the section $q$ extends
uniquely to a H\"older section
\[
        \widetilde q\in C^\alpha(\mathbb T^N,\operatorname{Hom}(\mathcal L,E)).
\]
We then show that $\widetilde q$ solves the twisted cohomological equation on
all of $\mathbb T^N$; see Section \ref{sec:23}.

\smallskip

 Finally, setting
\[
        \mathcal K(x)=P_E|_{\mathcal L_x}+\widetilde q(x),
\]
we obtain a H\"older solution of the original linear cocycle equation. The
last step is to prove that $\mathcal K(x):\mathcal L_x\to E$ is an
isomorphism for every $x\in\mathbb T^N$, with uniformly bounded inverse; see Section \ref{sec:24}.

\subsection{Notation and basic facts} We list notation and basic facts that will be used in this section.

\subsubsection{Local charts for $\mathcal{L}$}\label{sec:6}   Since $\mathcal L$ is an $\alpha$-H\"older subbundle of
$T\TT^N\cong \TT^N\times\RR^N$, we may choose a finite open cover
$\{U_i\}_{i\in I}$ of $\TT^N$ and, on each $U_i$, an
$\alpha$-H\"older local trivialization $\Theta_i(x):\RR^{\text{rank}\mathcal L}\to \mathcal L_x$. For $x,y\in U_i$, define
\[
\mathcal I_{x,y}
:=
\Theta_i(y)\circ \Theta_i(x)^{-1}
:\mathcal L_x\to \mathcal L_y .
\]
Then
\begin{align}\label{for:39}
  \mathcal{I}_{x,y}=\mathcal{I}_{y,x}^{-1}, \quad\text{ and }\quad \norm{\mathcal{I}_{x,y}u-u}\leq Cd(x,y)^\alpha\norm{u}, \quad x,y\in U_i,\quad u\in \mathcal L_x
\end{align}
where both fibers are viewed as subspaces of $\RR^N$.

After decreasing $\alpha>0$ if necessary, we may assume that both
$\mathcal L$ and $Df$ are $\alpha$-H\"older. Since $\mathcal L$ is
$Df$-invariant, the restricted cocycle $Df|_{\mathcal L}$ is
$\alpha$-H\"older with respect to the identifications $\mathcal I_{x,y}$. More precisely,
whenever $x,y$ lie in a common chart and $fx,fy$ lie in a common chart, we have
\[
\left\|
Df_x(u)
-
\mathcal I_{fy,fx}
\circ Df_y
\circ \mathcal I_{x,y}(u)
\right\|
\leq
C d(x,y)^\alpha \|u\|,
\qquad u\in\mathcal L_x.
\]
For simplicity, we write this as
\[
\left\|
Df_x|_{\mathcal L_x}
-
Df_y|_{\mathcal L_y}
\right\|
\leq
C d(x,y)^\alpha,
\]
where the two maps are compared using the identifications
$\mathcal I_{x,y}$ and $\mathcal I_{fy,fx}$.

\subsubsection{Fiber bunched property of $\mathcal{L}$}\label{for:85}
Let $\xi:\TT^N\to GL(N,\RR)$ be a map. The $GL(N,\mathbb R)$-valued
cocycle over $f$ generated by $\xi$ is the map
\[
        \mathcal B:\mathbb T^N\times\mathbb Z\to GL(N,\mathbb R)
\]
defined by $\mathcal B(x,0)=\operatorname{Id}$ and, for $n\in\mathbb N$,
\[
        \mathcal B(x,n)
        =
        \xi(f^{n-1}x)\circ\cdots\circ \xi(x),
        \qquad
        \mathcal B(x,-n)
        =
        \mathcal B(f^{-n}x,n)^{-1}.
\]
We say that a $\alpha$-H\"older cocycle $\mathcal{B}$ over $f$ is \emph{fiber bunched} (see \cite{S15}) if
there exists $0<\theta < 1$ such that for all $x \in \TT^N$ and $n \in \mathbb{N}$,
\begin{equation*}
    \|\mathcal{B}^n_x\| \cdot \|(\mathcal{B}^n_x)^{-1}\| \cdot \nu^{n\alpha} < C\theta^n
    \quad \text{and} \quad
    \|\mathcal{B}^{-n}_x\| \cdot \|(\mathcal{B}^{-n}_x)^{-1}\| \cdot \nu^{n\alpha} < C\theta^n,
\end{equation*}
where $\nu\in(0,1)$ is the contraction constant fixed in Section~\ref{sec:8}.

It follows from condition \eqref{for:123} of Proposition \ref{po:5} that for any $n\geq 0$
\begin{align*}
 \norm{Df^n|_{\mathcal{L}}}\cdot \norm{Df^{-n}|_{\mathcal{L}}}\leq C.
\end{align*}
Consequently, both $Df|_{\mathcal L}$ and $Df^{-1}|_{\mathcal L}$ are
fiber-bunched. This allows us to compare iterates of the cocycle along orbit
segments which remain exponentially close.
\begin{proposition}\label{po:1} \cite[proposition 4.2]{KS13}
There is $\delta>0$ such that for any $0<\gamma<\delta$, any $m\in\NN$ and any $x,\,y\in \TT^N$, if
  \begin{align*}
   d(f^ix,f^iy)\leq C\gamma\nu^i,\qquad 0\leq i\leq m,
  \end{align*}
then we have
\begin{align*}
\big\| (Df^{m}_y)^{-1}\circ \mathcal{I}_{f^mx,f^my}\circ Df^m_x(u)-\mathcal{I}_{x,y}(u)\big\|\leq C\gamma^\alpha\norm{u}
\end{align*}
for any $u\in \mathcal{L}_x$.
\end{proposition}
\begin{corollary}\label{cor:1} There is $\delta>0$ such that for any $0<\gamma<\delta$, any $m\in\NN$ and any $x,\,y\in \TT^N$,  if
\begin{align*}
d(f^jx,f^jy)\leq C\gamma\nu^{m-j},\qquad \forall\,\,0\leq j\leq m,
\end{align*}
then we have
\begin{align*}
\Big\| (Df^{-m}_{f^m(y)})^{-1}\circ \mathcal{I}_{x,y}\circ Df^{-m}_{f^m(x)}(u)-\mathcal{I}_{f^m(x),f^m(y)}(u)\Big\|\leq C\gamma^\alpha\norm{u}
\end{align*}
for any $u\in \mathcal{L}_{f^{m}(x)}$.

\end{corollary}
\begin{proof}
Let $X=f^m x$ and $Y=f^m y$. By assumption, for $0\leq j\leq m$,
\begin{align*}
  d\big(f^{-j}X,f^{-j}Y\big)=d(f^{m-j}x,f^{m-j}y)\leq C\gamma\nu^{m-(m-j)}=C\gamma\nu^{j},\quad 0\leq j\leq m,
\end{align*}
Thus the pair $X,Y$ satisfies the hypothesis of Proposition~\ref{po:1} for
the inverse map $f^{-1}$. Applying Proposition~\ref{po:1} to the
fiber-bunched cocycle $Df^{-1}|_{\mathcal L}$, with $x,y$ replaced by
$X,Y$, gives the desired estimate.
\end{proof}

The following two lemmas are useful reformulations of
Proposition~\ref{po:1} and Corollary~\ref{cor:1} for the normalized cocycles
$\rho^{-m}Df^m|_{\mathcal L}$ and $\rho^mDf^{-m}|_{\mathcal L}$.

\begin{lemma}\label{le:5} There is $\delta>0$ such that for any $0<\gamma<\delta$, any $m\in\NN$ and any $x,\,y\in \TT^N$, if
  \begin{align*}
   d(f^ix,f^iy)\leq C\gamma\nu^i,\qquad 0\leq i\leq m,
  \end{align*}
then we have
\begin{align}\label{for:148}
\Big\|\mathcal{I}_{f^m(x),f^m(y)}\circ (\rho^{-m}Df^m_x)(u)-(\rho^{-m}Df^{m}_y)\circ\mathcal{I}_{x,y}(u)\Big\|\leq C\gamma^\alpha\norm{u}
\end{align}
for any $u\in \mathcal{L}_x$; and
\begin{align}\label{for:151}
\Big\|(\rho^{m}Df^{-m}_{f^m(y)})\circ\mathcal{I}_{f^m(x),f^m(y)}(u)-\mathcal{I}_{x,y}\circ (\rho^{m}Df^{-m}_{f^m(x)})(u)\Big\|
\leq C\gamma^\alpha\norm{u}
\end{align}
for any $u\in \mathcal{L}_{f^m(x)}$.
\end{lemma}
\begin{proof}
\eqref{for:148}: For any $u\in \mathcal{L}_x$ we have
\begin{align*}
 &\Big\|\mathcal{I}_{f^m(x),f^m(y)}\circ (\rho^{-m}Df^m_x)(u)-(\rho^{-m}Df^{m}_y)\circ\mathcal{I}_{x,y}(u)\Big\|\\
 &= \Big\| (\rho^{-m}Df^{m}_y)\circ \big((Df^{m}_y)^{-1}\circ \mathcal{I}_{f^m(x),f^m(y)}\circ Df^m_x(u)-\mathcal{I}_{x,y}(u)\big)\Big\|\\
 &\leq \big\|\rho^{-m}Df^{m}_y|_{\mathcal{L}_y}\big\|\cdot \Big\|(Df^{m}_y)^{-1}\circ \mathcal{I}_{f^m(x),f^m(y)}\circ Df^m_x(u)-\mathcal{I}_{x,y}(u)\Big\|\\
 & \overset{\text{(1)}}{\leq} C\cdot C\gamma^\alpha\norm{u}=C_{1}\gamma^\alpha\norm{u}.
\end{align*}
Here in $(1)$ we use condition \eqref{for:123} of Proposition \ref{po:5} and Proposition \ref{po:1}. Thus we get \eqref{for:148}.

\smallskip

\eqref{for:151}: Let
\begin{align}\label{for:19}
 \mathcal{Y}=\mathcal{I}_{f^m(x),f^m(y)}\circ (\rho^{-m}Df^m_x)|_{\mathcal{L}_x}-(\rho^{-m}Df^{m}_y)\circ\mathcal{I}_{x,y}.
\end{align}
It follows from \eqref{for:148} that
\begin{align*}
 \norm{\mathcal{Y}}\leq C\gamma^\alpha.
\end{align*}
Then we have
\begin{align*}
 &\Big\|(\rho^{m}Df^{-m}_{f^m(y)})\circ\mathcal{I}_{f^m(x),f^m(y)}(u)-\mathcal{I}_{x,y}\circ (\rho^{m}Df^{-m}_{f^m(x)})(u)\Big\|\\
 &=\Big\|(\rho^{m}Df^{-m}_{f^m(y)})\circ \mathcal{Y}\circ (\rho^{m}Df^{-m}_{f^m(x)})(u)\Big\|\notag\\
 &\leq \big\|\rho^{m}Df^{-m}_{f^m(y)}|_{\mathcal{L}_{f^m(y)}} \big\|\cdot \big\|\mathcal{Y}\big\|\cdot \big\|(\rho^{m}Df^{-m}_{f^m(x)})(u)\big\|\notag\\
 &\overset{\text{(1)}}{\leq} C\cdot C\gamma^\alpha \cdot C\norm{u}.
\end{align*}
Here in $(1)$ we use condition \eqref{for:123} of Proposition \ref{po:5} and \eqref{for:19}. Thus we get \eqref{for:151}.

\end{proof}

Using Corollary \ref{cor:1} and condition \eqref{for:123} of Proposition \ref{po:5}, the following lemma  follows in the same way.

\begin{lemma}\label{le:1} There is $\delta>0$ such that for any $0<\gamma<\delta$, any $m\in\NN$ and any $x,\,y\in \TT^N$,  if
\begin{align*}
d(f^jx,f^jy)\leq C\gamma\nu^{m-j},\qquad \forall\,\,0\leq j\leq m,
\end{align*}
then we have
\begin{align}\label{for:149}
&\Big\|\mathcal{I}_{x,y}\circ (\rho^{m}Df^{-m}_{f^m(x)})(u)-(\rho^{m}Df^{-m}_{f^m(y)})\circ \mathcal{I}_{f^m(x),f^m(y)}(u)\Big\|\leq C\gamma^\alpha\norm{u}
\end{align}
for any $u\in \mathcal{L}_{f^m(x)}$; and
\begin{align}\label{for:152}
&\Big\|(\rho^{-m}Df^{m}_{y})\circ\mathcal{I}_{x,y}(u)-\mathcal{I}_{f^m(x),f^m(y)}\circ (\rho^{-m}Df^{m}_{x})(u)\Big\|\leq C\gamma^\alpha\norm{u}
\end{align}
for any $u\in \mathcal{L}_x$.
\end{lemma}
\subsection{Proof of Proposition \ref{po:5}}

\subsubsection{Step 1:  Reduction to a twisted cohomological equation}\label{sec:9}  We look for a solution $\mathcal K_x:\mathcal L_x\to E$ of
\eqref{for:9} in the form
\[
\mathcal K_x=P_E|_{\mathcal L_x}+q(x),
\]
where $q(x)\in \text{Hom}(\mathcal L_x,E)$. Let
\begin{align*}
 \textbf{r}(x)
        =
        P_E\circ Df_x|_{\mathcal L_x}
        -
        \mathcal A\circ P_E|_{\mathcal L_x},\qquad x\in \TT^N.
\end{align*}
Then the linear cocycle equation \eqref{for:9} is equivalent to the twisted cohomological equation:
\begin{align}\label{for:49}
 \mathcal{A}\circ q(x)-q(fx)\circ Df_x|_{\mathcal{L}_x}=\textbf{r}(x), \qquad x\in \TT^N.
\end{align}
Condition \eqref{for:126} of Proposition \ref{po:5} gives that
\begin{align}\label{for:188}
 \textbf{r}(0)=0.
\end{align}
Let
\begin{align*}
J(x)=-P_{E}|_{\mathcal{L}_x}.
\end{align*}
Since $\mathcal{K}(x)\equiv0$ is a trivial solution of equation \eqref{for:9}, $J$ solves
\begin{align}\label{for:16}
 \mathcal{A}\circ J(x)-J(fx)\circ Df|_{\mathcal{L}_x}=\textbf{r}(x), \qquad x\in \TT^N.
\end{align}
\subsubsection{Step 2: Preliminaries } In this step, we define series of linear maps and list some facts that will be frequently used later.
For any $m\geq0$ (resp. $n\leq-1$) define
\begin{align}
 &q^{[m]}(x)=\sum_{j=0}^{m} \mathcal{A}^{-(j+1)}\circ \textbf{r}(f^jx)\circ Df^j|_{\mathcal{L}_x} \label{for:193}\\
 \big(\text{resp. }& q^{[n]}(x)=-\sum_{j=n}^{-1} \mathcal{A}^{-(j+1)}\circ \textbf{r}(f^jx)\circ Df^j|_{\mathcal{L}_x}\big)\label{for:194}
\end{align}
Iterating \eqref{for:16},  for any $m\geq0$ we have
\begin{align}\label{for:18}
 q^{[m]}(x)=\mathcal{A}^{-(m+1)}\circ P_{E}\circ Df^{m+1}|_{\mathcal{L}_x}+J(x),\quad \forall\,x\in\TT^N.
\end{align}
Consequently, for any $m\geq0$  we have
\begin{align}\label{for:99}
 \big\|q^{[m]}(x)\big\|&\leq \big\| \mathcal{A}^{-(m+1)}\circ P_{E}\circ Df^{m+1}|_{\mathcal{L}_x}\big\|+\norm{J(x)}\notag\\
 &\leq \big\| (\rho^{m+1}\mathcal{A}^{-(m+1)})\big\|\cdot \big\| (\rho^{-(m+1)}Df^{m+1})|_{\mathcal{L}_x}\big\|+C\notag\\
 &\overset{\text{(1)}}{\leq} C\cdot C+C
\end{align}
for any $x\in \TT^N$.  Here in $(1)$ we use conditions \eqref{for:124} and \eqref{for:123} of Proposition \ref{po:5}.

Similarly, iterating \eqref{for:16} backwards for any $n\leq-1$ we obtain
\begin{align}\label{for:89}
 q^{[n]}(x)=\mathcal{A}^{-n}\circ P_{E}\circ Df^{n}|_{\mathcal{L}_x}+J(x),\quad \forall\,x\in\TT^N.
\end{align}
Similar to \eqref{for:99}, for any $n\leq -1$ we have
\begin{align}\label{for:31}
 \big\|q^{[n]}(x)\big\|\leq C, \qquad \forall\,x\in\TT^N.
\end{align}
From \eqref{for:18} and \eqref{for:89} for any $m\geq0$ and $n\leq-1$ we have
\begin{align}\label{for:12}
 &\mathcal{A}^{-(m+1)}\circ P_{E}\circ Df^{m+1}|_{\mathcal{L}_x}-\mathcal{A}^{-n}\circ P_{E}\circ Df^n|_{\mathcal{L}_x}\notag\\
 &=\sum_{j=n}^{m} \mathcal{A}^{-(j+1)}\circ \textbf{r}(f^jx)\circ Df^j|_{\mathcal{L}_x},
\end{align}
for any $x\in \TT^N$.

Since $0$ is a fixed point of $f$ and $\textbf{r}(0)=0$ (see \eqref{for:188}), we have $q^{[k]}(0)=0$ for any $k\in\ZZ$. From \eqref{for:18} and \eqref{for:89}  we see  that
\begin{align}\label{for:50}
 \mathcal{A}^{-k}\circ P_{E}\circ Df^{k}|_{\mathcal{L}_{0}}+J(0)=0,\qquad \forall \,k\in\ZZ.
\end{align}
\begin{remark} Both definition \eqref{for:193} (resp. \eqref{for:194}) and the equivalent expression \eqref{for:18} (resp. \eqref{for:89}), will be used below.
\end{remark}
\subsubsection{Step 3: Construction of the solutions $q^+$ and $q^-$}\label{sec:10}

In this step, we show that, for every $x\in \mathcal W^{s,f}(0)$, the limit
\begin{equation}\label{for:27}
q^+(x):=\lim_{n\to+\infty} q^{[n]}(x)\qquad\text{exists.}
\end{equation}
Similarly, for every $x\in \mathcal W^{u,f}(0)$, the limit
\begin{equation}\label{for:40}
q^-(x):=\lim_{n\to-\infty} q^{[n]}(x)\qquad \text{exists.}
\end{equation}
Moreover,
\begin{equation}\label{for:98}
\|q^+(x)\|\leq C,
\qquad x\in \mathcal W^{s,f}(0),
\end{equation}
and
\begin{equation}\label{for:30}
\|q^-(x)\|\leq C,
\qquad x\in \mathcal W^{u,f}(0).
\end{equation}
 \begin{remark}\label{re:3}
 Since $q^{[k]}(0)=0$ for every $k\in\ZZ$, the limits defined in
\eqref{for:27} and \eqref{for:40} satisfy $q^+(0)=0$ and $q^-(0)=0$.

 Since $\textbf{r}(0)=0$, this value indeed satisfies the twisted equation
\eqref{for:49} at the fixed point $0$. The purpose of this step is to extend
this solution to the whole stable leaf $\mathcal W^{s,f}(0)$, respectively to
the whole unstable leaf $\mathcal W^{u,f}(0)$.
 \end{remark}
We first prove \eqref{for:27}. Fix $x\in \mathcal W^{s,f}(0)$. First, we show that for  any unit vector $v\in \mathcal{L}_x$ and any $m_1> m_2\geq1$
\begin{align}\label{for:1}
 \Big\|\sum_{j=m_2}^{m_1-1} \mathcal{A}^{-(j+1)}\circ \textbf{r}(f^jx)\circ Df^j_x(v)\Big\|\leq C\big\|q^{[m_1-m_2-1]}(f^{m_2}(x))\big\|.
\end{align}
To see this we rewrite
\begin{align}\label{for:20}
 &\sum_{j=m_2}^{m_1-1} \mathcal{A}^{-(j+1)}\circ \textbf{r}(f^jx)\circ Df^j_x(v)\notag\\
 &=(\rho^{m_2}\mathcal{A}^{-m_2})\circ \big(\sum_{j=m_2}^{m_1-1}  \mathcal{A}^{-(j-m_2+1)}\circ \textbf{r}(f^jx)\circ Df^{j-m_2}_{f^{m_2}(x)}\big)\circ\big(\rho^{-m_2}Df_x^{m_2}(v)\big)\notag\\
 &\overset{\text{(1)}}{=}(\rho^{m_2}\mathcal{A}^{-m_2})\circ \big(q^{[m_1-m_2-1]}(f^{m_2}(x))\big)\circ\big(\rho^{-m_2}Df_x^{m_2}(v)\big).
\end{align}
Here in $(1)$ we recall \eqref{for:193} and note that
\begin{align*}
 &\sum_{j=m_2}^{m_1-1}  \mathcal{A}^{-(j-m_2+1)}\circ \textbf{r}(f^jx)\circ Df^{j-m_2}_{f^{m_2}(x)}|_{\mathcal{L}_{f^{m_2}(x)}}\\
 &=\sum_{j=0}^{m_1-m_2-1}  \mathcal{A}^{-(j+1)}\circ \textbf{r}\big(f^{j}(f^{m_2}x)\big)\circ Df^{j}_{f^{m_2}x}|_{\mathcal{L}_{f^{m_2}(x)}}\\
 &=q^{[m_1-m_2-1]}(f^{m_2}(x)).
\end{align*}
Then \eqref{for:1} follows from \eqref{for:20} and  conditions \eqref{for:124} and \eqref{for:123} of Proposition \ref{po:5}.

By \eqref{for:193}, from \eqref{for:1} we see that to prove \eqref{for:27}, it suffices to show that
\begin{align*}
 \big\|q^{[m_1-m_2-1]}(f^{m_2}(x))\big\|
\end{align*}
is arbitrarily small whenever $m_2$ is sufficiently large, uniformly in
$m_1>m_2$.

By \eqref{for:18} we have
\begin{align}\label{for:182}
&
\left\|q^{[m_1-m_2-1]}(f^{m_2}x)\right\| \notag\\
&=
\left\|
\mathcal A^{-(m_1-m_2)}
\circ
P_E
\circ
Df_{f^{m_2}x}^{m_1-m_2}|_{\mathcal L_{f^{m_2}x}}
+
J(f^{m_2}x)
\right\| \notag\\
&\overset{\text{(1)}}{=}
\left\|
\left(
\mathcal A^{-(m_1-m_2)}
\circ
P_E
\circ
Df_{f^{m_2}x}^{m_1-m_2}|_{\mathcal L_{f^{m_2}x}}
+
J(f^{m_2}x)
\right)
\right. \notag\\
&\qquad\left.
-
\left(
\mathcal A^{-(m_1-m_2)}
\circ
P_E
\circ
Df_0^{m_1-m_2}|_{\mathcal L_0}
+
J(0)
\right)
\circ
\mathcal I_{f^{m_2}x,0}
\right\| \notag\\
&\overset{\text{(2)}}{\leq}
\|\mathcal Y\|
+
\left\|
J(f^{m_2}x)
-
J(0)\circ \mathcal I_{f^{m_2}x,0}
\right\|,
\end{align}
Here in $(1)$ we use \eqref{for:50}; in $(2)$ we set
\begin{align*}
 \mathcal Y
:={}&
\mathcal A^{-(m_1-m_2)}
\circ
P_E
\circ
Df_{f^{m_2}x}^{m_1-m_2}|_{\mathcal L_{f^{m_2}x}}-
\mathcal A^{-(m_1-m_2)}
\circ
P_E
\circ
Df_0^{m_1-m_2}
\circ
\mathcal I_{f^{m_2}x,0}.
\end{align*}
We first estimate the second term in \eqref{for:182}.  We recall local identifications between fibers in Section \ref{sec:6}.  Suppose $x\in \mathcal{W}^{s,f}(0)$.  For any $\epsilon>0$ there is $l=l(x,\epsilon)\in\NN$ such that
\begin{align}\label{for:21}
 d(f^{p+n}(x),0)<\epsilon\nu^n,\qquad p\geq l,\quad n\geq0,
\end{align}
(see \eqref{for:130} of Section \ref{sec:8}).  Suppose $m_1>m_2\geq l$. We have
\begin{align}\label{for:90}
 &\big\|J(f^{m_2}(x))-J(0)\circ (\mathcal{I}_{f^{m_2}(x),0}) \big\|\overset{\text{(1)}}{\leq} Cd(f^{m_2}(x),0)^\alpha\overset{\text{(2)}}{<} C\epsilon^\alpha.
\end{align}
Here in $(1)$ we use the fact that $\mathcal{L}$ is $\alpha$-H\"{o}lder and thus $J$ is $\alpha$-H\"{o}lder along $\mathcal{L}$, see Section \ref{sec:6}; in $(2)$ we use \eqref{for:21} by letting $p=m_2$ and $n=0$.

Next, we estimate $\big\|\mathcal{Y}\big\|$. Let $s=m_1-m_2$. We rewrite
$\mathcal Y$ as
\[
\mathcal Y
=
(\rho^s\mathcal A^{-s})
\circ
P_E
\circ
(\mathcal Y_1+\mathcal Y_2),
\]
where
\[
\mathcal Y_1
=
(I_{id}-\mathcal I_{f^{m_1}x,0})
\circ
\left(
\rho^{-s}Df_{f^{m_2}x}^{s}|_{\mathcal L_{f^{m_2}x}}
\right)
\]
and
\[
\mathcal Y_2
=
\mathcal I_{f^{m_1}x,0}
\circ
\left(
\rho^{-s}Df_{f^{m_2}x}^{s}|_{\mathcal L_{f^{m_2}x}}
\right)
-
\left(
\rho^{-s}Df_0^{s}
\right)
\circ
\mathcal I_{f^{m_2}x,0}.
\]
From \eqref{for:21} we have
\begin{align*}
 d\big(f^{n}(f^{m_2}x),f^n(0)\big)=d(f^{m_2+n}(x),0)<\epsilon\nu^n,\qquad \forall\,n\geq0.
\end{align*}
Next, we estimate $\|\mathcal{Y}_2\|$. Taking $m_2$ sufficiently large, we may apply \eqref{for:148} of Lemma \ref{le:5} with
$f^{m_2}x$ and $0$ to estimate $\mathcal{Y}_2$.  Thus we have:
  \begin{align*}
&\| \mathcal{Y}_2\|\leq C\epsilon^\alpha.
\end{align*}
Now we estimate $\|\mathcal{Y}_1\|$.
\begin{align*}
 \| \mathcal{Y}_1\|\leq \norm{I_{id}-\mathcal I_{f^{m_1}x,0}}\cdot \|\rho^{-s}Df_{f^{m_2}x}^{s}\|\overset{\text{(1)}}{\leq} Cd(f^{m_1}x,0)^\alpha \cdot C\overset{\text{(2)}}{\leq} C_1\epsilon^\alpha.
\end{align*}
Here in $(1)$ we use \eqref{for:39} and    condition \eqref{for:123}  of Proposition \ref{po:5}; in $(2)$ we use \eqref{for:21} by letting $p=m_1$ and $n=0$.

As a consequence, we have
\begin{align}\label{for:22}
\big\|\mathcal{Y}\big\|&\leq \|\rho^s\mathcal A^{-s}\|
\|P_E\|
(
\|\mathcal Y_1\|+\|\mathcal Y_2\|
)\overset{\text{(1)}}{\leq}C(\norm{\mathcal{Y}_1}+\norm{\mathcal{Y}_2})\leq C_1\epsilon^\alpha.
\end{align}
Here in $(1)$ we use condition \eqref{for:124}  of Proposition \ref{po:5}.

From \eqref{for:182}, \eqref{for:90} and \eqref{for:22},  we have
\begin{align}\label{for:25}
  &\big\|q^{[m_1-m_2-1]}(f^{m_2}(x))\big\|\leq C\epsilon^\alpha.
 \end{align}
 This together with \eqref{for:1} gives
 \begin{align*}
  &\Big\|\sum_{j=m_2}^{m_1-1} \mathcal{A}^{-(j+1)}|_{E}\circ \textbf{r}(f^jx)\circ Df^j_x|_{\mathcal{L}_x}\Big\|\leq C\epsilon^\alpha.
 \end{align*}
 This implies \eqref{for:27}.

 The proof of \eqref{for:40} is similar. Finally, \eqref{for:98} follows
directly from \eqref{for:27} and \eqref{for:99}, while \eqref{for:30} follows
from \eqref{for:40} and \eqref{for:31}.

 \subsubsection{Step 4:  H\"older regularity of $q^+(x)$ and $q^-(x)$}\label{sec:4}
 In this step, we show that there exists $\delta_0>0$ such that, for any
$x,y\in\mathcal W^{s,f}(0)$ with $d(x,y)<\delta_0$, we have
\begin{align}\label{for:47}
\|q^+(x)-q^+(y)\circ\mathcal I_{x,y}\|
\leq
C d(x,y)^\alpha.
\end{align}
Similarly, for any $x,y\in\mathcal W^{u,f}(0)$ with $d(x,y)<\delta_0$, we have
\begin{align}\label{for:46}
\|q^-(x)-q^-(y)\circ\mathcal I_{x,y}\|
\leq
C d(x,y)^\alpha.
\end{align}
 We first prove \eqref{for:47}. Let $x,y\in\mathcal W^{s,f}(0)$ with
$d(x,y)$ sufficiently small.
For any $m>0$ we have
\begin{align}\label{for:43}
 &\big\|q^{[m-1]}(x)-q^{[m-1]}(y)\circ \mathcal{I}_{x,y}\big\|\notag\\
 &\overset{\text{(1)}}{= }\Big\|(\rho^{m}\mathcal{A}^{-m})\circ P_{E}\circ (\rho^{-m}Df^{m}_x)|_{\mathcal{L}_{x}}+J(x)\notag\\
 &-\big((\rho^{m}\mathcal{A}^{-m})\circ P_{E}\circ(\rho^{-m}Df^{m}_y)\circ \mathcal{I}_{x,y}+J(y)\circ \mathcal{I}_{x,y}\big)\Big\|\notag\\
 &\leq \|\rho^{m}\mathcal{A}^{-m}\|\cdot\big\|\mathcal{Y}\big\|+\big\|J(x)-J(y)\circ \mathcal{I}_{x,y} \big\|\notag\\
 &\overset{\text{(2)}}{\leq } C\big\|\mathcal{Y}\big\|+Cd(x,y)^\alpha,
\end{align}
where
\begin{align*}
\mathcal{Y}=P_{E}\circ(\rho^{-m}Df^{m}_x)|_{\mathcal{L}_{x}}-P_{E}\circ(\rho^{-m}Df^{m}_y)\circ \mathcal{I}_{x,y}.
\end{align*}
Here in $(1)$ we use \eqref{for:18}; in $(2)$ we use condition \eqref{for:124}  of Proposition \ref{po:5} and
the $\alpha$-\text{H\"older} of $J$ along $\mathcal{L}$.

Next, we estimate $\big\|\mathcal{Y}\big\|$. We rewrite
\begin{align*}
 \mathcal{Y}&=P_{E}\circ(\rho^{-m}Df^{m}_x)|_{\mathcal{L}_{x}}-P_{E}\circ\mathcal{I}_{f^{m}(x),f^{m}(y),\mathcal{L}}\circ (\rho^{-m}Df^{m}_x)|_{\mathcal{L}_{x}}\\
 &+P_{E}\circ\mathcal{I}_{f^{m}(x),f^{m}(y),\mathcal{L}}\circ (\rho^{-m}Df^{m}_x)|_{\mathcal{L}_{x}}-P_{E}\circ(\rho^{-m}Df^{m}_y)\circ \mathcal{I}_{x,y}.
\end{align*}
For the first part, we have
\begin{align*}
 &\big\|P_{E}\circ(\rho^{-m}Df^{m}_x)|_{\mathcal{L}_{x}}-P_{E}\circ\mathcal{I}_{f^{m}(x),f^{m}(y)}\circ (\rho^{-m}Df^{m}_x)|_{\mathcal{L}_{x}}\big\|\\
 &\leq C\big\|\mathcal{I}_{f^{m}(x),f^{m}(y)}-I_{id}\big\|\cdot\big\| (\rho^{-m}Df^{m})|_{\mathcal{L}_{x}}\big\|\\
 &\overset{\text{(1)}}{\leq } Cd(f^{m}(x),f^{m}(y))^\alpha\cdot C\\
 &\overset{\text{(2)}}{\leq }C_1\nu^{\alpha m}d(x,y)^\alpha.
\end{align*}
Here in $(1)$ we use \eqref{for:39} and condition \eqref{for:123}  of Proposition \ref{po:5}; in $(2)$ we use the contraction along stable leaves, see \eqref{for:130} of Section \ref{sec:8}.

Now we estimate the second part.   By the contraction along stable leaves, for
$d(x,y)$ sufficiently small, apply \eqref{for:148} of Lemma \ref{le:5} we have
 \begin{align*}
 \big\| \mathcal{I}_{f^{m}(x),f^{m}(y)}\circ (\rho^{-m}Df^{m}_x)|_{\mathcal{L}_{x}}-(\rho^{-m}Df^{m}_y)\circ \mathcal{I}_{x,y} \big\| \leq Cd(x,y)^\alpha.
\end{align*}
Hence, it follows that
\begin{align*}
 \big\|P_{E}\circ\mathcal{I}_{f^{m}(x),f^{m}(y)}\circ (\rho^{-m}Df^{m}_x)|_{\mathcal{L}_{x}}-P_{E}\circ(\rho^{-m}Df^{m}_y)\circ \mathcal{I}_{x,y}\big\|\leq Cd(x,y)^\alpha.
\end{align*}
Hence, we have
\begin{align}\label{for:97}
 \big\|\mathcal{Y}\big\|\leq C_1\nu^{\alpha m}d(x,y)^\alpha+Cd(x,y)^\alpha \leq C_{2}d(x,y)^\alpha.
\end{align}
\eqref{for:97} and \eqref{for:43} give
\begin{align*}
 \big\|q^{[m-1]}(x)-q^{[m-1]}(y)\circ \mathcal{I}_{x,y}\big\|\leq Cd(x,y)^\alpha
\end{align*}
 for any $m>0$. Let $m\to\infty$.  It follows from \eqref{for:27} that
\begin{align*}
 \big\|q^+(x)-q^+(y)\circ \mathcal{I}_{x,y}\big\|\leq Cd(x,y)^\alpha.
\end{align*}
Hence we get \eqref{for:47}.

We now prove the unstable estimate \eqref{for:46}. Let
$x,y\in\mathcal W^{u,f}(0)$ with $d(x,y)$ sufficiently small.  By using \eqref{for:89}, similar to \eqref{for:43}, for any $m>0$ we have
\begin{align}\label{for:161}
 &\big\|q^{[-m]}(x)-q^{[-m]}(y)\circ \mathcal{I}_{x,y}\big\|\notag\\
 &=\Big\|(\rho^{-m}\mathcal{A}^{m})\circ P_{E}\circ (\rho^{m}Df^{-m}_x)|_{\mathcal{L}_{x}}+J(x)\notag\\
 &-\big((\rho^{-m}\mathcal{A}^{m})\circ P_{E}\circ(\rho^{m}Df^{-m}_y)\circ \mathcal{I}_{x,y}+J(y)\circ \mathcal{I}_{x,y}\big)\Big\|\notag\\
&\leq C\big\|\mathcal{Y}_1\big\|+Cd(x,y)^\alpha,
\end{align}
 where
\begin{align*}
\mathcal{Y}_1=P_{E}\circ(\rho^{m}Df^{-m}_x)|_{\mathcal{L}_{x}}
 -P_{E}\circ(\rho^{m}Df^{-m}_y)\circ \mathcal{I}_{x,y}.
\end{align*}
Since $x,y$ lie on the same unstable leaf, backward iterates contract. Thus,
for $d(x,y)$ sufficiently small, we have
\begin{align*}
 d\big(f^j(f^{-m}x),&\,f^j(f^{-m}y)\big)=d(f^{j-m}x,\,f^{j-m}y)\leq C\nu^{m-j}d(x,y).
\end{align*}
Applying  \eqref{for:149} of Lemma \ref{le:1} to the pair
$f^{-m}x,f^{-m}y$, we have
\begin{align}\label{for:160}
\Big\|\mathcal{I}_{f^{-m}(x),f^{-m}(y)}\circ (\rho^{m}Df^{-m}_{x})|_{\mathcal{L}_{x}}-(\rho^{m}Df^{-m}_{y})\circ \mathcal{I}_{x,y}\Big\|\leq Cd(x,y)^\alpha.
\end{align}
By using \eqref{for:160}, similar to \eqref{for:97}, we can show that
\begin{align*}
 \big\|\mathcal{Y}_1\big\|\leq Cd(x,y)^\alpha.
\end{align*}
This, together with \eqref{for:161}, gives
\begin{align*}
 \big\|q^{[-m]}(x)-q^{[-m]}(y)\circ \mathcal{I}_{x,y}\big\|\leq Cd(x,y)^\alpha
\end{align*}
 for any $m>0$. Let $m\to\infty$.  It follows from \eqref{for:40} that
 \begin{align*}
 \big\|q^-(x)-q^-(y)\circ \mathcal{I}_{x,y}\big\|\leq Cd(x,y)^\alpha.
\end{align*}
Hence we get \eqref{for:46}. Taking $\delta_0$ to be the minimum of the local
constants required above completes the proof.

\subsubsection{Step 5:  Coincidence of  $q^+$ and $q^-$ on  $\mathcal{W}^{s,f}(0)\bigcap \mathcal{W}^{u,f}(0)$ }\label{sec:21} Let
\begin{align*}
 \mathcal{S}_{0}=\mathcal{W}^{s,f}(0)\bigcap \mathcal{W}^{u,f}(0).
\end{align*}
 In this step, we show that
\begin{align}\label{for:45}
 q^+(x)=q^-(x),\qquad \forall \, x\in \mathcal{S}_{0}.
\end{align}

\smallskip
\noindent Recall the sequence $k_n\to+\infty$ from condition \eqref{for:88} of Proposition \ref{po:5}. We first prove
that, for every $x\in\mathcal S_0$,
\begin{align}\label{for:84}
 \lim_{n\to\infty}\big(q^{[k_n-1]}(x)-q^{[-k_n]}(x)\big)=0.
\end{align}
Since $x\in\mathcal S_0$, the limits defining $q^+(x)$ and $q^-(x)$ both
exist. Hence \eqref{for:45} follows directly from \eqref{for:84}.

Fix $x\in\mathcal S_0$. The key step in proving \eqref{for:84} is the following: for any $\epsilon>0$, there is $\tau_{x,\epsilon}\in \NN$ such that for all  $k_n\geq \tau_{x,\epsilon}$:
\begin{align}\label{for:32}
 \big\|P_{E}\circ(\rho^{-k_n}Df^{k_n})|_{\mathcal{L}_{x}}-P_{E}\circ(\rho^{k_n}Df^{-k_n})|_{\mathcal{L}_{x}}\big\|\leq C\epsilon^\alpha.
\end{align}
 As a direct consequence of \eqref{for:32} we have
\begin{align}\label{for:42}
 \lim_{n\to \infty}\big\|P_{E}\circ(\rho^{-k_n}Df^{k_n})|_{\mathcal{L}_{x}}-P_{E}\circ(\rho^{k_n}Df^{-k_n})|_{\mathcal{L}_{x}}\big\|=0.
\end{align}
Before we proceed to the proof of \eqref{for:32}, we show how it implies \eqref{for:84}.

For any $x\in\TT^N$, we have
\begin{align}\label{for:41}
 q^{[k_n-1]}(x)-q^{[-k_n]}(x)
&\overset{\text{(1)}}{= }
(\rho^{k_n}\mathcal A^{-k_n})
\circ P_E\circ(\rho^{-k_n}Df_x^{k_n})|_{\mathcal L_x}\notag \\
&\quad -
(\rho^{-k_n}\mathcal A^{k_n})
\circ P_E\circ(\rho^{k_n}Df_x^{-k_n})|_{\mathcal L_x} \notag\\
&\overset{\text{(2)}}{= }
\mathcal Y_1+\mathcal Y_2+\mathcal Y_3,
\end{align}
Here in $(1)$ we use \eqref{for:18} and \eqref{for:89}; in $(2)$
\begin{align*}
 \mathcal{Y}_1&=(\rho^{k_n}\mathcal{A}^{-k_n})\circ P_{E}\circ (\rho^{-k_n}Df^{k_n})|_{\mathcal{L}_{x}}-P_{E}\circ(\rho^{-k_n}Df^{k_n})|_{\mathcal{L}_{x}},\\
 &=(\rho^{k_n}\mathcal{A}^{-k_n}-I_{id}|_{E})\circ P_{E}\circ(\rho^{-k_n}Df^{k_n})|_{\mathcal{L}_{x}}\\
 \mathcal{Y}_2&=P_{E}\circ(\rho^{-k_n}Df^{k_n})|_{\mathcal{L}_{x}}-P_{E}\circ(\rho^{k_n}Df^{-k_n})|_{\mathcal{L}_{x}},\\
 \mathcal{Y}_3&=P_{E}\circ(\rho^{k_n}Df^{-k_n})|_{\mathcal{L}_{x}}-(\rho^{-k_n}\mathcal{A}^{k_n})\circ P_{E}\circ (\rho^{k_n}Df^{-k_n})|_{\mathcal{L}_{x}}\\
 &=(I_{id}|_{E}-\rho^{-k_n}\mathcal{A}^{k_n})\circ P_{E}\circ(\rho^{k_n}Df^{-k_n})|_{\mathcal{L}_{x}}.
\end{align*}
By using condition \eqref{for:123} of Proposition \ref{po:5} we have
\begin{align*}
\norm{\mathcal{Y}_1}\leq C\big\|\rho^{k_n}\mathcal{A}^{-k_n}-I_{id}|_{E}\big\|,\quad \norm{\mathcal{Y}_3}\leq C\big\|I_{id}|_{E}-\rho^{-k_n}\mathcal{A}^{k_n}\big\|.
\end{align*}
By condition \eqref{for:36} of Proposition \ref{po:5}, these two quantities tend to zero. Together with \eqref{for:42}, this proves \eqref{for:84}.

It remains to prove \eqref{for:32}. Since $x\in \mathcal{S}_{0}$, for any $\epsilon>0$ there is $\tau_{x,\epsilon}\in \NN$ such that
\begin{align*}
 d(f^{m}(x),f^{-m}(x))\leq \epsilon,\qquad \forall\,m\geq \tau_{x,\epsilon}.
\end{align*}
Assume $\epsilon>0$ is small enough so that the Anosov Closing lemma (see Theorem \ref{th:2}) applies.
For any $k_n\geq \tau_{x,\epsilon}$, since
\begin{align*}
 d(f^{k_n}(x),f^{-k_n}(x))\leq \epsilon,
\end{align*}
there is $y\in\TT^N$ satisfying $y=f^{2k_n}(y)$ and
\begin{align}\label{for:29}
 d(f^{-k_n+j}(x),\,f^j(y))&=d\big(f^{j}(f^{-k_n}x),f^j(y)\big)\leq C\nu^{\min\{j,2k_n-j\}}d(f^{k_n}(x),f^{-k_n}(x))\notag\\
 &\leq C\nu^{\min\{j,2k_n-j\}}\epsilon,
\end{align}
for any $0\leq j\leq 2k_n$.

Increasing $\tau_{x,\epsilon}$ if necessary, we may assume
$k_n\geq l_\epsilon$ (see condition \eqref{for:28} of Proposition \ref{po:5}). Applying \eqref{for:28} to the periodic point $y$, we obtain
\begin{align}\label{for:35}
 \Big\|(\rho^{-k_n}Df^{k_n})|_{\mathcal{L}_{f^{k_n}(y)}}-(\rho^{k_n}Df^{-k_n})|_{\mathcal{L}_{f^{k_n}(y)}}\Big\|\leq \epsilon.
\end{align}
Moreover, taking $j=0$ and $j=2k_n$ in \eqref{for:29}, we get
\begin{align}\label{for:154}
 d(f^{-k_n}(x),\,y)\leq C\epsilon\quad\text{and}\quad d(f^{k_n}(x),\,y)\leq C\epsilon.
\end{align}
We now estimate
\begin{align*}
P_{E}\circ(\rho^{-k_n}Df^{k_n})|_{\mathcal{L}_{x}}-P_{E}\circ(\rho^{k_n}Df^{-k_n})|_{\mathcal{L}_{x}}=\mathcal{X}_1+\mathcal{Y}+\mathcal{X}_2,
\end{align*}
where
\begin{gather*}
\mathcal{X}_1=P_{E}\circ(I_{id}-\mathcal{I}_{f^{k_n}(x),y})\circ (\rho^{-k_n}Df^{k_n}_x)|_{\mathcal{L}_{x}},\\
\mathcal{Y}=P_{E}\circ\mathcal{I}_{f^{k_n}(x),y}\circ(\rho^{-k_n}Df^{k_n}_x)|_{\mathcal{L}_x}
  -P_{E}\circ\mathcal{I}_{f^{-k_n}(x),y}\circ(\rho^{k_n}Df^{-k_n}_x)|_{\mathcal{L}_x},\\
  \mathcal{X}_2=P_{E}\circ(\mathcal{I}_{f^{-k_n}(x),y}-I_{id})\circ (\rho^{k_n}Df^{-k_n}_x)|_{\mathcal{L}_x}.
\end{gather*}
By using condition \eqref{for:123} of Proposition, we have
\begin{align}\label{for:157}
 \max\{\norm{\mathcal{X}_1},\,\norm{\mathcal{X}_2}\}&\leq C\max\{\norm{I_{id}-\mathcal{I}_{f^{k_n}(x),y}},\,\norm{\mathcal{I}_{f^{-k_n}(x),y}-I_{id}}\}\notag\\
 & \overset{\text{(1)}}{\leq }C_1d(f^{k_n}(x),y)^\alpha+C_1d(f^{-k_n}(x),y)^\alpha\overset{\text{(2)}}{\leq }C_{2}\epsilon^\alpha.
\end{align}
 Here in $(1)$ we use \eqref{for:39}; in $(2)$ we use \eqref{for:154}.

It remains to estimate $\mathcal Y$. Taking $j=k_n$ in \eqref{for:29}, we get
\begin{align*}
 d(x,\,f^{k_n}(y))\leq C\nu^{k_n}\epsilon.
\end{align*}
Thus, the closeness  of $x$ and $f^{k_n}(y)$ allows us to rewrite $\mathcal{Y}$ as
\begin{align*}
 \mathcal{Y}=P_{E}\mathcal{Y}_1+P_{E}\mathcal{Y}_2+P_{E}\mathcal{Y}_3,
\end{align*}
where
\begin{align*}
\mathcal{Y}_1&= \mathcal{I}_{f^{k_n}(x),y}\circ (\rho^{-k_n}Df^{k_n}_x)|_{\mathcal{L}_x}- (\rho^{-k_n}Df^{k_n}_{f^{k_n}(y)})\circ\mathcal{I}_{x,f^{k_n}(y)},\\
\mathcal{Y}_2&=(\rho^{-k_n}Df^{k_n}_{f^{k_n}(y)})\circ\mathcal{I}_{x,f^{k_n}(y)}-(\rho^{k_n}Df^{-k_n}_{f^{k_n}(y)})\circ\mathcal{I}_{x,f^{k_n}(y)},\\
\mathcal{Y}_3&=(\rho^{k_n}Df^{-k_n}_{f^{k_n}(y)})\circ \mathcal{I}_{x,f^{k_n}(y)}-\mathcal{I}_{f^{-k_n}(x),y}\circ (\rho^{k_n}Df^{-k_n}_x)|_{\mathcal{L}_x}.
\end{align*}
First, by \eqref{for:35} and the uniform boundedness of the identifications:
\begin{align*}
 \big\|\mathcal{Y}_2\big\|&\leq \Big\|(\rho^{-k_n})Df^{k_n}_{f^{k_n}(y)}|_{\mathcal{L}_{f^{k_n}(y)}}
 -(\rho^{k_n}Df^{-k_n}_{f^{k_n}(y)})|_{\mathcal{L}_{f^{k_n}(y)}}\Big\|\cdot\big\|\mathcal{I}_{x,f^{k_n}(y)}\big\|\leq C\epsilon.
\end{align*}
Next,  \eqref{for:29} shows that for any $0\leq j\leq k_n$, we have
\begin{align*}
 d\big(f^{j}(x)&,f^{j}(f^{k_n}y)\big)=d\big(f^{-k_n+(j+k_n)}(x),f^{j+k_n}(y)\big)\notag\\
 &\leq C\nu^{\min\{j+k_n,2k_n-(j+k_n)\}}\epsilon=C\nu^{k_n-j}\epsilon.
\end{align*}
Thus the pair $x,f^{k_n}y$ satisfies the hypothesis of Lemma~\ref{le:1}.
Using \eqref{for:152}, we obtain
\begin{align*}
 \big\| \mathcal{Y}_1 \big\|\leq C(C\epsilon)^\alpha.
\end{align*}
Similarly, \eqref{for:29} shows that for any $0\leq j\leq k_n$, we have
\begin{align*}
d(f^{j}(f^{-k_n}(x)),f^j(y))=d(f^{-k_n+j}(x),f^j(y))&\leq C\nu^{j}\epsilon.
\end{align*}
Thus the pair $f^{-k_n}x,y$ satisfies the hypothesis of Lemma~\ref{le:5}.
Using \eqref{for:151}, we obtain
\begin{align*}
 \big\| \mathcal{Y}_3 \big\|\leq C(C\epsilon)^\alpha.
\end{align*}
Hence, for \(\epsilon>0\) sufficiently small,
\begin{align*}
 \big\|\mathcal{Y}\big\|\leq C(C\epsilon)^\alpha+C\epsilon+C(C\epsilon)^\alpha\leq C_1\epsilon^\alpha.
\end{align*}
This, together with \eqref{for:157}, gives
\begin{align*}
\big\|P_E\circ (\rho^{-k_n}Df^{k_n})|_{\mathcal{L}_{x}}-P_E\circ (\rho^{k_n}Df^{-k_n})|_{\mathcal{L}_{x}}\big\|\leq C\epsilon^\alpha.
\end{align*}
This implies \eqref{for:32}.

\subsubsection{Step 6:  H\"older regularity of $q(x)$ on  $\mathcal{S}_{0}$ }\label{sec:22}  Let
\begin{align}\label{for:24}
 q(x)=q^+(x)=q^-(x),\qquad \forall \, x\in \mathcal{S}_{0}.
\end{align}
In this step, we show that:
\begin{enumerate}
  \item\label{for:192} on $\mathcal{S}_{0}$ we have $\norm{q(x)}\leq C$;

  \item\label{for:38} there is $\delta_{0,1}>0$ such that for any $x,\,y\in \mathcal{S}_{0}$, if $d(x,y)<\delta_{0,1}$, then
\begin{align*}
\big\|q(x)-q(y)\circ \mathcal{I}_{x,y}\big\|\leq Cd(x,y)^\alpha.
\end{align*}
\end{enumerate}
The bound in \eqref{for:192} follows directly from \eqref{for:98}. We now
prove \eqref{for:38}. Choose $\delta_{0,1}>0$ sufficiently small so
that the following two conditions hold. First, the local product structure
applies at scale $\delta_{0,1}$ (see \eqref{for:44} of Section \ref{sec:8}). Second, whenever
$d(x,y)<\delta_{0,1}$ and
\begin{align*}
 z\in \mathcal W^{s,f}_{\mathrm{loc}}(x)
        \cap
        \mathcal W^{u,f}_{\mathrm{loc}}(y)
\end{align*}
is the local product point, the three points $x,y,z$ lie in a common
trivializing chart for $\mathcal L$. In particular,
\begin{align}\label{for:58}
 \mathcal I_{x,z}
        =
        \mathcal I_{y,z}\circ\mathcal I_{x,y}.
\end{align}
Let $x,y\in\mathcal S_0$ with $d(x,y)<\delta_{0,1}$, and let
\begin{align*}
   z\in \mathcal W^{s,f}_{\mathrm{loc}}(x)
        \cap
        \mathcal W^{u,f}_{\mathrm{loc}}(y)
\end{align*}
be the local product point. Since $x\in\mathcal W^{s,f}(0)$ and
$z\in\mathcal W^{s,f}_{\mathrm{loc}}(x)$, we have
$z\in\mathcal W^{s,f}(0)$. Similarly, since
$y\in\mathcal W^{u,f}(0)$ and
$z\in\mathcal W^{u,f}_{\mathrm{loc}}(y)$, we have
$z\in\mathcal W^{u,f}(0)$. Hence $z\in\mathcal S_0$.

By the local product estimates,
\[
        d(x,z)+d(y,z)\leq C d(x,y).
\]
Using the stable H\"older regularity of $q^+$ and the unstable H\"older
regularity of $q^-$ from Step~4, we obtain
\begin{align*}
\|q(x)&-q(y)\circ\mathcal I_{x,y}\|
\leq
\|q(x)-q(z)\circ\mathcal I_{x,z}\|+
\|q(z)\circ\mathcal I_{x,z}
      -q(y)\circ\mathcal I_{x,y}\|\\
&\overset{\text{(a)}}{= } \|q(x)-q(z)\circ\mathcal I_{x,z}\|+\|(q(z)\circ\mathcal I_{y,z}-q(y))\circ\mathcal I_{x,y}\|\\
&\leq \|q(x)-q(z)\circ\mathcal I_{x,z}\|+C\|q(z)\circ\mathcal I_{y,z}-q(y)\|\\
&\leq
C d(x,z)^\alpha+C_1 d(z,y)^\alpha\\
&\leq
C_2 d(x,y)^\alpha .
\end{align*}
Here in $(a)$ we use \eqref{for:58}. This proves \eqref{for:38}.
\begin{remark}\label{re:1} Since $A$ is transitive,  $f$ is also transitive. Hence $\mathcal{S}_{0}$ is dense in $\TT^N$, see \cite{B}. Then \eqref{for:38}, together with the usual extension
argument in local trivializations implies that $q$ extends uniquely to an
$\alpha$-\text{H\"older} section
\[
\widetilde q\in C^\alpha(\TT^N,\operatorname{Hom}(\mathcal L,E)).
\]
Since $0\in\mathcal{S}_{0}$, Remark \ref{re:3} gives $\tilde{q}(0)=0$.

\end{remark}
\subsubsection{Step 7:  Solvability of \eqref{for:49}}\label{sec:23} In this step,  we show that $\tilde{q}$ solves \eqref{for:49}, i.e.,
\begin{align}\label{for:52}
 \mathcal{A}\circ \tilde{q}(x)-\tilde{q}(fx)\circ Df|_{\mathcal{L}_x}=\textbf{r}(x), \qquad \forall\,x\in \TT^N.
\end{align}
\smallskip
\noindent As $\tilde{q}$ is H\"{o}lder on $\TT^N$ and $\mathcal{S}_{0}$ is dense in $\TT^N$, to show that \eqref{for:52} holds, it suffices to show that
$\tilde{q}$ solves \eqref{for:49} on $\mathcal{S}_{0}$, i.e.,
\begin{align}\label{for:53}
 \mathcal{A}\circ \tilde{q}(x)-\tilde{q}(fx)\circ Df|_{\mathcal{L}_x}=\textbf{r}(x), \qquad \forall\,x\in \mathcal{S}_{0}.
\end{align}
Fix $x\in\mathcal S_0$. Since $\mathcal S_0$ is $f$-invariant, we have
$fx\in\mathcal S_0$. Hence
\[
\widetilde q(x)=q(x),
\qquad
\widetilde q(fx)=q(fx).
\]
Using the definition of $q$ on $\mathcal S_0$ and the convergence
$q^{[n]}\to q^+$ on stable leaves, we have
\begin{align}\label{for:54}
 &\mathcal{A}\circ q(x)-q(fx)\circ Df|_{\mathcal{L}_x}\notag\\
 &\overset{\text{(1)}}{=}\lim_n\big(\mathcal{A}\circ q^{[k_n-1]}(x)-q^{[k_n-1]}(fx)\circ Df|_{\mathcal{L}_x}\big)\notag\\
 &\overset{\text{(2)}}{=}\lim_n\big(\sum_{j=0}^{k_n-1}
\mathcal A^{-j}
\circ
\mathbf r(f^j x)
\circ
Df_x^j|_{\mathcal L_x}-\sum_{j=0}^{k_n-1}
\mathcal A^{-(j+1)}
\circ
\mathbf r(f^{j+1}x)
\circ
Df_x^{j+1}|_{\mathcal L_x}\big)\notag\\
 &=\textbf{r}(x)-\lim_n\mathcal{A}^{-k_n}\circ \textbf{r}(f^{k_n}x)\circ Df^{k_n}|_{\mathcal{L}_{x}}.
\end{align}
Here in $(1)$ we use \eqref{for:27}; in $(2)$ we recall definition \eqref{for:193} of $q^{[m]}$ and the resulting telescoping identity.

Next, we show that
\begin{align}\label{for:51}
 \lim_n\mathcal{A}^{-k_n}\circ \textbf{r}(f^{k_n}x)\circ Df^{k_n}|_{\mathcal{L}_{x}}=0,\qquad \forall\,x\in \mathcal{S}_{0}.
\end{align}
In fact, we have
\begin{align*}
 &\big\|\mathcal{A}^{-k_n}\circ \textbf{r}(f^{k_n}x)\circ Df^{k_n}|_{\mathcal{L}_{x}}\big\|\\
 &=\big\|(\rho^{k_n}\mathcal{A}^{-k_n})\circ \textbf{r}(f^{k_n}x)\circ (\rho^{-k_n}Df^{k_n})|_{\mathcal{L}_{x}} \big\|\\
 &\leq\big\|(\rho^{k_n}\mathcal{A}^{-k_n})\big\|\cdot \big\|\textbf{r}(f^{k_n}x)\big\|\cdot\big\|(\rho^{-k_n}Df^{k_n})|_{\mathcal{L}_{x}}\big\|\\
 &\overset{\text{(1)}}{\leq}C\cdot\big\|\textbf{r}(f^{k_n}x)\big\|\overset{\text{(2)}}{\rightarrow} C\cdot \norm{\textbf{r}(0)}=0,\qquad\text{as }n\to \infty.
\end{align*}
Here in $(1)$ we use conditions \eqref{for:124} and \eqref{for:123} of Proposition \ref{po:5}; in $(2)$ we use $f^{k_n}x\to 0$ as $n\to \infty$ and
recall \eqref{for:188}.

Hence we get \eqref{for:51}. This, together with \eqref{for:54},  implies \eqref{for:53}.
Hence  we proved \eqref{for:52}.

\subsubsection{Step 8: Invertibility  of $\mathcal{K}(x)$}\label{sec:24} In this step, we show that
\begin{align}\label{for:55}
 \norm{\mathcal{K}(x)u}\geq C^{-1}\norm{u},\qquad \forall\,x\in \TT^N,\,\,\forall\,u\in \mathcal{L}_{x}.
\end{align}
Let
\begin{align*}
  \mathcal{K}(x)=P_E|_{\mathcal L_x}+\tilde{q}(x) \qquad \forall\,x\in \TT^N.
\end{align*}
From discussion in Section \ref{sec:9}, we see that $\mathcal{K}$ solves \eqref{for:9}. Since $\tilde{q}(0)=0$ (see Remark \ref{re:1}),
$\mathcal{K}(0):\,\mathcal{L}_{0}\to E$ is injective (see \eqref{for:126} of Proposition \ref{po:5}). By continuity of $\mathcal{K}$ and $\mathcal{L}$, there is a neighbourhood $U$ of $0$ such that
\begin{align}\label{for:37}
 \norm{\mathcal{K}(x)u}\geq C^{-1}\norm{u},\qquad \forall\,x\in U,\quad \forall\,u\in \mathcal{L}_{x}.
\end{align}
We now propagate this estimate along forward iterates.  From \eqref{for:9}, we have
\[
\mathcal K_{f^n x}\circ Df_x^n|_{\mathcal{L}_x}
=
\mathcal A^n\circ \mathcal K_x.
\]
Equivalently,
\[
\mathcal K_{f^n x}
=
(\rho^{-n}\mathcal A^n)
\circ
\mathcal K_x
\circ
(\rho^n Df_{f^n x}^{-n})|_{\mathcal{L}_{f^nx}}.
\]
For any $x\in U$, $u\in \mathcal{L}_{f^n(x)}$ and $n\geq0$ we have
\begin{align}\label{for:56}
  \norm{\mathcal{K}(f^nx)(u)}&=\Big\|(\rho^{-n}\mathcal{A}^n)\circ \mathcal{K}(x)\circ (\rho^{n}Df^{-n})|_{\mathcal{L}_{f^n(x)}}(u)\Big\|\notag\\
  &\overset{\text{(1)}}{\geq}C^{-1}\big\|\mathcal{K}(x)\circ (\rho^{n}Df^{-n})|_{\mathcal{L}_{f^n(x)}}(u)\big\|\notag\\
  &\overset{\text{(2)}}{\geq} C^{-1}C^{-1}\norm{(\rho^{n}Df^{-n})|_{\mathcal{L}_{f^n(x)}}(u)}\notag\\
  &\overset{\text{(3)}}{\geq}C^{-1}C^{-1}\cdot C^{-1}\norm{u}.
\end{align}
Here in $(1)$ we use the uniform invertibility of
\(\rho^{-n}\mathcal A^n\), which follows from condition~\eqref{for:124} of Proposition \ref{po:5};
 in $(2)$ we use \eqref{for:37}; in $(3)$  we use that
$(\rho^n Df_{f^n x}^{-n})^{-1}=\rho^{-n}Df_x^n$ and condition~\eqref{for:123} of Proposition \ref{po:5}.

\eqref{for:56} shows that \eqref{for:55} holds on $\bigcup_{n\geq0}f^n(U)$, which is a dense set by transitivity of $f$. By continuity of $\mathcal{K}$,
 \eqref{for:55} holds on $\TT^N$.

 Since $\dim\mathcal L_x=\dim E$, estimate~\eqref{for:55} implies that
$\mathcal K(x):\mathcal L_x\to E$ is an isomorphism for every $x$. Moreover,
\eqref{for:55} gives the uniform inverse bound. Since
$\widetilde q$ and $\mathcal L$ are $\alpha$-H\"older, the bundle map
$\mathcal K$ is $\alpha$-H\"older; the uniform inverse bound then implies that
$\mathcal K^{-1}$ is also $\alpha$-H\"older. This completes the proof of
Proposition~\ref{po:5}.

\section{Reduction of the derivative cocycle }\label{sec:7}
\subsection{Notation} We recall the following notations.
\begin{enumerate}
  \item The numbers $\rho_i$, the Lyapunov blocks $\mathcal E_i$, and the
  linear subspaces $E_i$ are defined in Section~\ref{sec:12}.

  \item The maps $A_i$ and $p_i$ are defined in Section~\ref{sec:14}.

  \item The subbundles $\mathcal F_{i,j}$ are defined in item~\ref{for:33}
  of Section~\ref{sec:12}.
\end{enumerate}
In this section, we prove a cocycle-reduction result over the Lyapunov blocks
$\mathcal E_i$. This reduction will play a crucial role in the proofs of
Theorems~\ref{th:6} and~\ref{th:5}.

Since $A$ is irreducible, $A$ is diagonalizable over $\mathbb C$. Thus, for each
$i$, after a fixed real linear change of coordinates on $E_i$, the map
$A_i$ has a real block-diagonal form.
\begin{theorem}\label{th:9} For every \(1\leq i\leq \ell\), there exists a flag of \(A_i\)-invariant
subspaces
\begin{align}\label{for:63}
\{0\}
=
V_{i,0}
\subset
V_{i,1}
\subset
\cdots
\subset
V_{i,j_i}
=
E_i
\end{align}
and an $\alpha$-H\"older bundle isomorphism $\mathcal C_i:\mathcal E_i\to E_i$, whose inverse is also $\alpha$-H\"older, such that, for every \(x\in\TT^N\), the map
\[
\mathcal C_i(x):\mathcal E_i(x)\to E_i
\]
is a linear isomorphism satisfying the following properties.
\begin{enumerate}
\item For every \(0\leq j\leq j_i-1\), there is an $A_i$-invariant subspace
$W_{i,j+1}\subset E_i$ such that
\begin{align*}
 V_{i,j+1}=V_{i,j}\oplus W_{i,j+1}.
\end{align*}

  \item For every \(0\leq j\leq j_i\),
  \[
  \mathcal C_i(x)(\mathcal F_{i,j}(x))=V_{i,j}.
  \]
  We denote the restriction $\mathcal C_{i,j}(x):=\mathcal C_i(x)|_{\mathcal F_{i,j}(x)}$.

  \item\label{for:91}  There exists an $\alpha$-H\"older map
\[
        \widetilde A_i:\mathbb T^N\to \operatorname{End}(E_i),
\]
such that, for every $x\in\mathbb T^N$,
\[
        \mathcal C_i(fx)\circ Df_x|_{\mathcal E_i(x)}
        =
        \widetilde A_i(x)\circ \mathcal C_i(x).
\]
With respect to the decomposition
\[
        E_i=W_{i,1}\oplus\cdots\oplus W_{i,j_i},
\]
the linear map $\widetilde A_i(x)$ is block upper triangular for every
$x\in\mathbb T^N$, and its diagonal blocks coincide with the corresponding
diagonal blocks of $A_i$.

  \end{enumerate}

In particular, for the first subbundle in the flag, we have
\begin{align}\label{for:70}
\mathcal C_{i,1}(fx)\circ Df_x|_{\mathcal F_{i,1}(x)}
=
A_i|_{V_{i,1}}\circ \mathcal C_{i,1}(x),
\qquad x\in\TT^N.
\end{align}
\end{theorem}

\begin{remark}
From \eqref{for:70}, it is natural to think of $\mathcal C_{i,1}$ as
$DH$ along $\mathcal F_{i,1}$. However, $\mathcal F_{i,1}$ is constructed
abstractly and is not necessarily integrable. Even if $\mathcal F_{i,1}$ is
integrable, the differentiability of $H$ along $\mathcal F_{i,1}$ is not
straightforward.
\end{remark}

\subsection{Role of Theorem~\ref{th:9}}

Theorem~\ref{th:9} will be used twice later, in the proofs of
Theorems~\ref{th:6} and~\ref{th:5}.

The first use is in the proof of Proposition~\ref{po:2}, where we obtain
differentiability of $H$ along curves tangent to $\mathcal F_{i,1}$. This
eventually leads to the integrability of $\mathcal F_{i,1}$ and to the
differentiability of $H$ along $\mathcal F_{i,1}$; see
Theorem~\ref{th:1}. In this step, the conjugacy relation
\eqref{for:70} plays a crucial role.

The second use is in the proof of Theorem~\ref{th:4}, where we obtain
differentiability of $H$ along the whole block $\mathcal E_i$. In that step,
the full block upper-triangular reduction in \eqref{for:91} plays a
crucial role.

\subsection{Proof strategy}
 Fix $1\leq i\leq\ell$. The flag
\begin{align*}
  \{0\}
        =
        \mathcal F_{i,0}
        \subset
        \mathcal F_{i,1}
        \subset
        \cdots
        \subset
        \mathcal F_{i,j_i}
        =
        \mathcal E_i
\end{align*}
reduces the problem to the quotient cocycles on $\mathcal B_{i,j}
        =
        \mathcal F_{i,j}/\mathcal F_{i,j-1}$.
Each quotient cocycle carries an invariant conformal structure, and the
periodic-data estimates verified below allow us to apply
Proposition~\ref{po:5}. This gives a H\"older conjugacy on each quotient
bundle. We then assemble these quotient conjugacies, using a H\"older
splitting subordinate to the flag, to obtain the block upper-triangular
reduction of $Df$ over $\mathcal E_i$.

\subsection{Notations and basic facts} We list notations and basic facts that will be used in this section. Recall notations in \eqref{for:33} of Section~\ref{sec:12}.

\subsubsection{Periodic data}

 \begin{lemma}\label{ob:1}
 Fix $1\leq i\leq \ell$. For every $1\leq j\leq j_i$, every periodic point $p$ of $f$ with period $n$, there exist an $A_i^n$-invariant subspace
$s_{i,j,p}\subset E_i$ and a linear isomorphism
\[
\mathcal T_{i,j,p}:
(\mathcal F_{i,j})_p/(\mathcal F_{i,j-1})_p
\to
s_{i,j,p}
\]
such that
\[
\mathcal T_{i,j,p}\circ \overline{Df_p^n}
=
A_i^n|_{s_{i,j,p}}\circ \mathcal T_{i,j,p},
\]
where $\overline{Df_p^n}$ denotes the map induced by $Df_p^n$ on the quotient
$(\mathcal F_{i,j})_p/(\mathcal F_{i,j-1})_p$. Moreover,
\[
\max\left\{
\|\mathcal T_{i,j,p}\|,
\|\mathcal T_{i,j,p}^{-1}\|
\right\}
\leq C.
\]
\end{lemma}
\begin{proof}
By the construction of the flag $\mathcal F_{i,j}$, the quotient cocycle
induced by $\phi_i Df$ on $\mathcal F_{i,j}/\mathcal F_{i,j-1}$
is an isometry with respect to a H\"older Riemannian metric (see \eqref{for:33} of Section \ref{sec:12}). Hence, for every
periodic point $p$ of period $n$, the map
\[
\phi_i(f^{n-1}p)\cdots \phi_i(p)\,
\overline{Df_p^n}
\]
is an isometry on $(\mathcal F_{i,j})_p/(\mathcal F_{i,j-1})_p$.

Since all eigenvalues of $\overline{Df_p^n}$ have modulus $\rho_i^n$, we have
\begin{align}\label{for:14}
\phi_i(f^{n-1}p)\cdots \phi_i(p)=\rho_i^{-n}.
\end{align}
Thus $\rho_i^{-n}\overline{Df_p^n}$ is an isometry on $(\mathcal F_{i,j})_p/(\mathcal F_{i,j-1})_p$.

On the other hand, since $A_i$ is diagonalizable over $\mathbb C$ and all
of its eigenvalues have modulus $\rho_i$, we may choose an inner product on
$E_i$ with respect to which $\rho_i^{-1}A_i$ is an isometry.

By the periodic-data assumption, $Df_p^n$ is conjugate to $A^n$. Since
$\mathcal F_{i,j}(p)$ and $\mathcal F_{i,j-1}(p)$ are $Df_p^n$-invariant, the
eigenvalues of the quotient map $\overline{Df_p^n}$ are among the eigenvalues
of $Df_p^n|_{\mathcal E_i(p)}$, counted with multiplicity. Hence they agree
with a subcollection of the eigenvalues of $A_i^n$, counted with
multiplicity. Let $s_{i,j,p}\subset E_i$ be the real $A_i^n$-invariant
subspace corresponding to this subcollection. Then
$A_i^n|_{s_{i,j,p}}$ and $\overline{Df_p^n}$ have the same eigenvalues,
counted with multiplicity.

Therefore the normalized maps
\[
        \rho_i^{-n}\overline{Df_p^n}
        \qquad\text{and}\qquad
        \rho_i^{-n}A_i^n|_{s_{i,j,p}}
\]
are isometries with the same complex eigenvalues, counted with multiplicity.
Real isometries are orthogonally conjugate precisely when they have the same
complex eigenvalues with multiplicity. Hence there exists an isometry
\[
        \mathcal T_{i,j,p}:
        (\mathcal F_{i,j})_p/(\mathcal F_{i,j-1})_p
        \to
        s_{i,j,p}
\]
such that
\[
        \mathcal T_{i,j,p}\circ \rho_i^{-n}\overline{Df_p^n}
        =
        \rho_i^{-n}A_i^n|_{s_{i,j,p}}\circ \mathcal T_{i,j,p}.
\]
Equivalently,
\[
        \mathcal T_{i,j,p}\circ \overline{Df_p^n}
        =
        A_i^n|_{s_{i,j,p}}\circ \mathcal T_{i,j,p}.
\]
Finally, the H\"older metrics on the quotient bundles are uniformly equivalent
to the background metric, and the metric on $E_i$ is fixed. Hence the
isometries $\mathcal T_{i,j,p}$ and their inverses have uniformly bounded
operator norms. Therefore
\[
\max\left\{
\|\mathcal T_{i,j,p}\|,
\|\mathcal T_{i,j,p}^{-1}\|
\right\}
\leq C.
\]

\end{proof}
\begin{corollary}\label{cor:3} Fix \(1\le i\le \ell\) and \(1\le j\le j_i\). For every \(n\in\ZZ\), we have
\begin{align*}
 \Big\|
\rho_i^{-n}\overline{Df_x^n}
\Big\|
\leq C,
\qquad x\in\TT^N,
\end{align*}
where \(\overline{Df_x^n}\) denotes the cocycle induced by \(Df_x^n\) on the
quotient bundle $\mathcal F_{i,j}/\mathcal F_{i,j-1}$.
\end{corollary}
\begin{proof}
Let
\[
\Phi_i(x):=\ln(\rho_i\phi_i(x)).
\]
\eqref{for:14} in Lemma \ref{ob:1} shows that, for every periodic point \(p\) of
period \(m\), we have
\[
\Pi_{k=0}^{m-1}\phi_i(f^k p)=\rho_i^{-m}.
\]
Hence
\[
\sum_{k=0}^{m-1}\Phi_i(f^k p)
=
\ln\left(\Pi_{k=0}^{m-1}\rho_i\phi_i(f^k p)\right)
=
0.
\]
By the Liv\v{s}ic theorem, there exists a H\"older function
\(\lambda:\TT^N\to\RR\) such that
\[
\lambda(fx)-\lambda(x)=\Phi_i(x).
\]
Therefore, for every \(n\geq 1\),
\[
\Pi_{k=0}^{n-1}\phi_i(f^k x)
=
\rho_i^{-n}e^{\lambda(f^n x)-\lambda(x)}.
\]
Since the quotient cocycle induced by \(\phi_i Df\) on $\mathcal F_{i,j}/\mathcal F_{i,j-1}$
is an isometry with respect to a H\"older Riemannian metric, and since this metric is uniformly equivalent to
the background metric, we have
\begin{align*}
\Big\|
\left(\Pi_{k=0}^{n-1}\phi_i(f^k x)\right)
\overline{Df_x^n}
\Big\|
\leq C
\end{align*}
Then we have
\[
\left\|
\rho_i^{-n}\overline{Df_x^n}
\right\|=e^{\lambda(x)-\lambda(f^n x)}\Big\|
\left(\Pi_{k=0}^{n-1}\phi_i(f^k x)\right)
\overline{Df_x^n}
\Big\|
\leq
Ce^{\lambda(x)-\lambda(f^n x)}
\leq C_1.
\]
The estimate for $n<0$ follows by applying the same argument to the inverse
quotient cocycle. The case $n=0$ is immediate. This completes the proof.

\end{proof}

\subsubsection{Asymptotic behavior} \begin{lemma}\label{for:8} Fix $1\leq i\leq \ell$ and $1\leq j\leq j_i$. Denote by $\overline{Df_x}$
induced cocycle by $Df$ on the quotient bundle $\mathcal B:=\mathcal F_{i,j}/\mathcal F_{i,j-1}$.
Then there exists a sequence $k_n\to+\infty$ such that:
\begin{enumerate}
  \item\label{for:5} we have
  \begin{align*}
   \rho_i^{-k_n}A_i^{k_n}\to I_{id}|_{E_i}
  \quad\text{and}\quad
  \rho_i^{k_n}A_i^{-k_n}\to I_{id}|_{E_i}\quad n\to +\infty;
   \end{align*}

  \item\label{for:13} for any $\epsilon>0$ there is $l_\epsilon\in\NN$ such that: for any $k_n$ with $n\geq l_\epsilon$ and any $x\in\TT^N$ with $f^{2k_n}x=x$ we have
  \begin{align*}
   \Big\|\rho_i^{-k_n}\overline{Df^{k_n}_{f^{k_n}(x)}}-\rho_i^{k_n}\overline{Df^{-k_n}_{f^{k_n}(x)}}\Big\|\leq \epsilon.
  \end{align*}
  \end{enumerate}

\end{lemma}

\begin{proof}
Since $A_i$ is diagonalizable over $\CC$ and all of its eigenvalues have
modulus $\rho_i$, we may write its eigenvalues as
\[
\rho_i e^{\sqrt{-1}a_1},\ldots,\rho_i e^{\sqrt{-1}a_m}.
\]
Choose a sequence $\ell_n\to+\infty$ with $\ell_{n+1}-\ell_n\to+\infty$ such
that
\begin{align*}
 e^{\sqrt{-1}a_r\ell_n}\qquad\text{converges for every } 1\leq r\leq m.
\end{align*}
Let $k_n=\ell_{n+1}-\ell_{n}$.  This proves \eqref{for:5}.

\eqref{for:13}: Suppose that $f^{2k_n}x=x$. Let $y=f^{k_n}x$. Then $f^{2k_n}y=y$. By Lemma \ref{ob:1}, there
exist an $A_i^{2k_n}$-invariant subspace $s_{i,j,y}\subset E_i$ and a linear isomorphism
\[
T_{i,j,y}:\mathcal B_y\to s_{i,j,y}
\]
such that
\[
T_{i,j,y}\circ \overline{Df^{2k_n}_y}
=
A^{2k_n}|_{s_{i,j,y}}\circ T_{i,j,y}.
\]
Hence
\begin{align}\label{for:187}
\left(\rho_i^{-2k_n}\overline{Df_y^{2k_n}}-I_{id}|_{\mathcal{B}_y}\right)
=
T_{i,j,y}^{-1}
\left(
\rho_i^{-2k_n}A^{2k_n}|_{s_{i,j,y}}-I_{id}
\right)
T_{i,j,y}.
\end{align}
By \eqref{for:5}, for every $\epsilon>0$ there exists $l_\epsilon\in\NN$ such
that
\begin{align*}
 \big\|
\rho_i^{-2k_n}A_i^{2k_n}
-
I_{id}|_{E_i}
\big\|
<\epsilon^2,
\qquad n\geq l_\epsilon.
\end{align*}
By using the uniform bounds on $T_{i,j,y}$ and $T_{i,j,y}^{-1}$ from
Lemma~\ref{ob:1}, \eqref{for:187} gives
\begin{align*}
  \big\| \rho_i^{-2k_n}\overline{Df_y^{2k_n}}-I_{id}|_{\mathcal{B}_y}\big\|&\leq C\epsilon^2, \qquad n\geq l_\epsilon.
  \end{align*}
We note that
\[
        \rho_i^{-k_n}\overline{Df_y^{k_n}}
        -
        \rho_i^{k_n}\overline{Df_y^{-k_n}}
        =
        \rho_i^{k_n}\overline{Df_y^{-k_n}}
        \left(
        \rho_i^{-2k_n}\overline{Df_y^{2k_n}}
        -
        Id|_{\mathcal B_y}
        \right).
\]
 This, together with Corollary~\ref{cor:3}, gives
\[
\begin{aligned}
        \left\|
        \rho_i^{-k_n}\overline{Df_y^{k_n}}
        -
        \rho_i^{k_n}\overline{Df_y^{-k_n}}
        \right\|
        &\leq
        \left\|
        \rho_i^{k_n}\overline{Df_y^{-k_n}}
        \right\|
        \left\|
        \rho_i^{-2k_n}\overline{Df_y^{2k_n}}
        -
        \operatorname{Id}_{\mathcal B_y}
        \right\|  \\
        &\leq
        C\cdot C\epsilon^2\leq\epsilon.
\end{aligned}
\]
Since $y=f^{k_n}x$, this is exactly \eqref{for:13}. The proof is complete.

 \end{proof}

\subsection{Proof of Theorem \ref{th:9}} Fix $1\leq i\leq\ell$. Recall that $\mathcal E_i=\mathcal F_{i,j_i}$ and that we have the flag
\[
\{0\}
=
\mathcal F_{i,0}
\subset
\mathcal F_{i,1}
\subset
\cdots
\subset
\mathcal F_{i,j_i}
=
\mathcal E_i,
\]
see \eqref{for:33} of  Section~\ref{sec:12}. For each $1\leq j\leq j_i$, set
\[
\mathcal B_{i,j}
:=
\mathcal F_{i,j}/\mathcal F_{i,j-1},
\]
and denote by
\[
\overline{Df_x}:(\mathcal B_{i,j})_x\to(\mathcal B_{i,j})_{fx}
\]
the quotient cocycle induced by $Df_x$.

Since $Df_0$ is conjugate to $A$, and since $Df_0$ preserves the flag at
$0$, the quotient map
\[
\overline{Df_0}:(\mathcal B_{i,j})_0\to(\mathcal B_{i,j})_0
\]
is diagonalizable over $\mathbb C$, and all of its eigenvalues have modulus
$\rho_i$. After reordering and grouping the real diagonal blocks of
$A_i:=A|_{E_i}$ if necessary, we may choose an $A_i$-invariant block
$W_{i,j}\subset E_i$ and a linear isomorphism
\[
P_{i,j}:(\mathcal B_{i,j})_0\to W_{i,j}
\]
such that, writing $A_{i,j}:=A_i|_{W_{i,j}}$, we have
\begin{align}\label{for:11}
P_{i,j}\circ \overline{Df_0}
=
A_{i,j}\circ P_{i,j}.
\end{align}
We now apply Proposition~\ref{po:5} to the quotient cocycle
$\overline{Df_x}$ on $\mathcal B_{i,j}$, with target space $(\mathcal B_{i,j})_0$ and fixed
linear map $\overline{Df_0}$. More precisely, we use Proposition~\ref{po:5} in its vector-bundle form, with
\[
        \mathcal L=\mathcal B_{i,j},\qquad
        E=(\mathcal B_{i,j})_0,\qquad
        \mathcal A=\overline{Df_0}.
\]
We verify the hypotheses of Proposition~\ref{po:5}.
\begin{enumerate}
\item By \eqref{for:11}, the map $\overline{Df_0}|_E$ is conjugate to
$A_{i,j}$. Hence $\overline{Df_0}$ is diagonalizable over $\mathbb C$, and
all of its eigenvalues have modulus $\rho_i$.

  \item At the fixed point $0$, the fixed-point compatibility condition is
tautological:
\[
\overline{Df_0}=\mathcal{A}\qquad \text{on }(\mathcal B_{i,j})_0.
\]


  \item By Corollary \ref{cor:3}, for every $n\in\ZZ$, we have
  \[
  \bigl\|\rho_i^{-n}\overline{Df_x^n}\bigr\|\leq C,
  \qquad x\in \TT^N.
  \]

  \item Let the sequence $k_n$ be as in Lemma \ref{for:8}. By \eqref{for:11} and \eqref{for:5} of Lemma \ref{for:8}, we have
  \begin{align*}
   \rho_i^{-k_n}(\overline{Df_0})^{k_n}\to I_{id}|_{(\mathcal B_{i,j})_0},
\qquad
\rho_i^{k_n}(\overline{Df_0})^{-k_n}\to I_{id}|_{(\mathcal B_{i,j})_0}.
  \end{align*}
  Moreover, by \eqref{for:13} of Lemma \ref{for:8}, for any $\epsilon>0$ there is $l_\epsilon\in\NN$ such that: for any $k_n$ with $n\geq l_\epsilon$ and any $x\in\TT^N$ with $f^{2k_n}x=x$ we have
  \begin{align*}
   \Big\|(\rho_i^{-k_n}\overline{Df^{k_n})}|_{\mathcal{B}_{f^{k_n}(x)}}-(\rho_i^{k_n}\overline{Df^{-k_n})}|_{\mathcal{B}_{f^{k_n}(x)}}\Big\|\leq \epsilon.
  \end{align*}

\end{enumerate}
Hence all hypotheses of Proposition~\ref{po:5} are satisfied. Then for each $1\leq j\leq j_i$, there exists an
$\alpha$-\text{H\"older} bundle map
\[
\mathcal K_{i,j}(x):(\mathcal B_{i,j})_x\to (\mathcal B_{i,j})_0
\]
such that each $\mathcal K_{i,j}(x)$ is a linear isomorphism and
\[
\mathcal K_{i,j}(fx)\circ \overline{Df_x}
=
\overline{Df_0}\circ \mathcal K_{i,j}(x),
\qquad x\in\mathbb T^N.
\]
Moreover,
\[
\|(\mathcal K_{i,j}(x))^{-1}u\|
\leq
C\|u\|,
\qquad x\in\mathbb T^N,\quad u\in (\mathcal B_{i,j})_0.
\]
Define
\[
\widehat{\mathcal C}_{i,j}(x)
:=
P_{i,j}\circ \mathcal K_{i,j}(x).
\]
Then
\[
\widehat{\mathcal C}_{i,j}(x):(\mathcal B_{i,j})_x\to W_{i,j}
\]
is an \(\alpha\)-\text{H\"older} bundle isomorphism satisfying
\begin{equation}\label{for:quotient-conjugacy}
\widehat{\mathcal C}_{i,j}(fx)\circ \overline{Df_x}
=
A_{i,j}\circ \widehat{\mathcal C}_{i,j}(x),
\qquad x\in\TT^N.
\end{equation}
We now assemble these quotient conjugacies. Choose an
\(\alpha\)-\text{H\"older} splitting subordinate to the flag:
\[
\mathcal E_i
=
\mathcal G_{i,1}\oplus\cdots\oplus\mathcal G_{i,j_i},
\]
where
\[
\mathcal F_{i,j}
=
\mathcal G_{i,1}\oplus\cdots\oplus\mathcal G_{i,j}.
\]
Let
\[
\pi_{i,j}:\mathcal F_{i,j}\to
\mathcal F_{i,j}/\mathcal F_{i,j-1}
=
\mathcal B_{i,j}
\]
be the quotient projection. Define
\[
V_{i,j}
:=
W_{i,1}\oplus\cdots\oplus W_{i,j},
\qquad 0\leq j\leq j_i,
\]
with \(V_{i,0}=\{0\}\). Then
\[
\{0\}=V_{i,0}\subset V_{i,1}\subset\cdots\subset V_{i,j_i}=E_i
\]
is a flag of \(A_i\)-invariant, and hence \(A\)-invariant, subspaces.

For
\[
v=v_1+\cdots+v_{j_i},
\qquad v_j\in\mathcal G_{i,j}(x),
\]
define
\begin{align*}
 \mathcal C_i(x)v
:=
\sum_{j=1}^{j_i}
\widehat{\mathcal C}_{i,j}(x)\bigl(\pi_{i,j}(v_j)\bigr).
\end{align*}
Then
\[
\mathcal C_i(x):\mathcal E_i(x)\to E_i
\]
is a linear isomorphism for every \(x\in\TT^N\). Moreover, for every
\(0\leq j\leq j_i\),
\[
\mathcal C_i(x)(\mathcal F_{i,j}(x))=V_{i,j}.
\]
Thus, if we denote
\[
\mathcal C_{i,j}(x):=\mathcal C_i(x)|_{\mathcal F_{i,j}(x)},
\]
then
\[
\mathcal C_{i,j}(x):\mathcal F_{i,j}(x)\to V_{i,j}
\]
is a bundle isomorphism.

Since \(Df\) preserves the flag
\[
\mathcal F_{i,0}\subset\mathcal F_{i,1}\subset\cdots\subset\mathcal F_{i,j_i},
\]
the map
\[
\mathcal C_i(fx)\circ Df_x|_{\mathcal E_i(x)}\circ \mathcal C_i(x)^{-1}
\]
is block upper triangular with respect to the decomposition
\[
E_i=W_{i,1}\oplus\cdots\oplus W_{i,j_i}.
\]
Its \(j\)-th diagonal block is \(A_{i,j}\), by
\eqref{for:quotient-conjugacy}. Therefore, defining
\[
\widetilde A_i(x)
:=
\mathcal C_i(fx)\circ Df_x|_{\mathcal E_i(x)}\circ \mathcal C_i(x)^{-1},
\]
we obtain a block upper-triangular linear map
\[
\widetilde A_i(x):E_i\to E_i
\]
whose diagonal blocks coincide with the corresponding diagonal blocks of
\(A_i\), and
\[
\mathcal C_i(fx)\circ Df_x|_{\mathcal E_i(x)}
=
\widetilde A_i(x)\circ \mathcal C_i(x),
\qquad x\in\TT^N.
\]

It remains to check the H\"older regularity. Since each
\(\mathcal K_{i,j}\) is \(\alpha\)-H\"older and its inverse is
uniformly bounded, the identity
\[
\mathcal K_{i,j}(x)^{-1}-\mathcal K_{i,j}(y)^{-1}
=
\mathcal K_{i,j}(x)^{-1}
\bigl(\mathcal K_{i,j}(y)-\mathcal K_{i,j}(x)\bigr)
\mathcal K_{i,j}(y)^{-1}
\]
shows that \(\mathcal K_{i,j}^{-1}\) is also \(\alpha\)-H\"older.
Thus each \(\widehat{\mathcal C}_{i,j}\) is a bi-\(\alpha\)-H\"older
bundle isomorphism. Since the splitting
\[
\mathcal E_i
=
\mathcal G_{i,1}\oplus\cdots\oplus\mathcal G_{i,j_i}
\]
and the quotient projections \(\pi_{i,j}\) are \(\alpha\)-H\"older,
the assembled bundle map \(\mathcal C_i\) is bi-\(\alpha\)-H\"older.

Finally, for \(j=1\), we have
\[
\mathcal B_{i,1}
=
\mathcal F_{i,1}/\mathcal F_{i,0}
=
\mathcal F_{i,1}.
\]
Hence
\[
\mathcal C_{i,1}(fx)\circ Df_x|_{\mathcal F_{i,1}(x)}
=
A_i|_{V_{i,1}}\circ \mathcal C_{i,1}(x),
\qquad x\in\TT^N.
\]
This proves the stated assertion for the first subbundle in the flag and
completes the proof of Theorem~\ref{th:9}.

\section{Curve differentiability of $H$}\label{sec:34}
\subsection{Notation}\label{sec:39} We recall the following notation.
\begin{enumerate}
  \item $\mathfrak m$ and $\mu$ are defined in Section~\ref{sec:12}.
  \item The numbers $\rho_i$, the Lyapunov blocks $\mathcal E_i$, and the
  linear subspaces $E_i$ are defined in Section~\ref{sec:12}.

  \item The maps $A_i$ and $p_i$ are defined in Section~\ref{sec:14}.
  \item The foliations $\mathcal W_i^f$ and $\mathcal W_i^A$ are defined in
  \eqref{for:23} of Section~\ref{sec:12}.
  \item The subbundle $\mathcal F_{i,1}$ is defined in \eqref{for:33} of
  Section~\ref{sec:12}.
  \item The map \(\mathcal C_{i,1}\) and the subspace \(V_{i,1}\) are defined in
  Theorem~\ref{th:9}.

  \item\label{for:7} Set \(H_i=p_i\circ H\). After choosing a lift to the universal cover,
  we write
  \[
        H_i=p_i+h_i .
  \]
  Any two choices of lift differ by a constant vector. Hence this choice does
  not affect the differentiability of \(H_i\) or \(H\) in this section.
\end{enumerate}
We fix $i_0\leq i\leq\ell$. Throughout this section, we write $\mathcal F:=\mathcal F_{i,1}$ for simplicity.
\subsection{Curve differentiability along $\mathcal F$}\label{sec:36} We introduce the notion of curve differentiability of $H$ along
$\mathcal F$. Since $H$ is torus-valued, derivatives and integrals of
$H\circ\gamma$ will always be understood after choosing a local lift to
$\mathbb R^N$. This convention is independent of the choice of lift.

\begin{enumerate}
  \item A curve $\gamma:(-1,1)\to\TT^N$ is called a \emph{regular curve along
  $\mathcal F$} if $\gamma$ is a $C^{1+\alpha}$ curve and
$\gamma'(t)\in \mathcal{F}_{\gamma(t)}$ for any $t\in (-1,1)$.  We say that $\gamma$ passes through $x$ if
  $\gamma(0)=x$.

\smallskip
  \item $x\in\TT^N$. Two regular curves $\gamma_1,\gamma_2$ along $\mathcal F$
  passing through $x$ are called \emph{equivalent} if $\gamma_1'(0)=\gamma_2'(0)$.
We denote the equivalence class of $\gamma$ by $[\gamma]_x$.

\smallskip
  \item  Let $x\in\TT^N$ and $u\in\mathcal F_x$. We say that $H$ is
  \emph{curve differentiable along $u$} if there exists a regular curve
  $\gamma$ along $\mathcal F$ such that
  \[
  \gamma(0)=x,
  \qquad
  \gamma'(0)=u,
  \]
  and the derivative
  \[
  \left.\frac{d}{dt}\right|_{t=0} H\circ\gamma(t)
  \]
  exists and is independent of the choice of such \(\gamma\). In this case we define
  \[
  dH_x^{\mathfrak c}(u)
  :=
  \left.\frac{d}{dt}\right|_{t=0} H\circ\gamma(t).
  \]

\end{enumerate}

\subsection{Main result}
The role of Proposition~\ref{po:2} is to provide the first step in the proof of Theorem~\ref{th:1}. More precisely, Proposition~\ref{po:2} proves curve differentiability of \(H\) along \(\mathcal F\), which will later be used to establish the integrability of \(\mathcal F\)
and the differentiability of \(H\) along its integral leaves.
\begin{proposition}\label{po:2}
Suppose that \(H(\mathcal W_i^f)=\mathcal W_i^A\) for some
\(i_0\leq i\leq \ell\). Then, for every \(x\in\TT^N\) and every
\(u\in\mathcal F_x\), we have
\[
dH_x^{\mathfrak c}(u)=\mathcal D_i(x)(u),
\]
where
\[
\mathcal D_i(x):=B_i\circ \mathcal C_{i,1}(x):\mathcal F_x\to E_i.
\]
Here  \(B_i:V_{i,1}\to E_i\) is a fixed linear map.
\end{proposition}

\begin{remark}
From the conjugacy relation \eqref{for:70} in Theorem~\ref{th:9}, it is
natural to think that \(\mathcal C_{i,1}\) represents
\(DH_i|_{\mathcal F}\). However, the solution in \eqref{for:70} is not
canonically normalized. For instance, one may post-compose
\(\mathcal C_{i,1}\) with a constant linear map commuting with \(A_i\) and
obtain another solution of the same conjugacy relation. Proposition~\ref{po:2}
identifies the actual curve derivative of \(H_i\) along \(\mathcal F\): up to
the fixed linear map \(B_i\), it is precisely \(\mathcal C_{i,1}\).
\end{remark}

A key step in the proof of Proposition \ref{po:2} is the following result:
\begin{lemma}\label{th:7} Suppose $H(\mathcal{W}^f_i) = \mathcal{W}^A_i$ for some $i_0\leq i\leq \ell$. Then for any regular curve $\gamma$ along $\mathcal{F}$, there is $\delta>0$ such that for any $C^1$ function $\omega$ compactly supported on $(-\delta,\delta)$, we have
\begin{align*}
 \int_{-\delta}^\delta H\circ \gamma(t)\omega'(t)dt=-\int_{-\delta}^\delta \mathcal D_i(\gamma(t))(\gamma'(t))\omega(t)dt.
\end{align*}
\end{lemma}
The integrals in Lemma~\ref{th:7} are understood componentwise. More precisely,
if \(g\) is vector-valued and \(\varphi\) is scalar-valued, then \(g\varphi\)
denotes their componentwise product, and the integral of \(g\varphi\) is taken
componentwise.

\begin{remark}\label{re:4}
Under the assumption
$H(\mathcal W_i^f)=\mathcal W_i^A$, the components $H_j$, $j\neq i$, are
constant along the leaves of $\mathcal W_i^f$. Hence, for $j\neq i$,
\begin{align*}
 \int_{-\delta}^{\delta}
        H_j\circ\gamma(t)\,\omega'(t)\,dt
        =
        H_j(\gamma(0))
        \int_{-\delta}^{\delta}\omega'(t)\,dt
        =
        0,
\end{align*}
because $\omega\in C_c^1(-\delta,\delta)$. Hence,
\begin{align*}
 \int_{-\delta}^\delta H\circ \gamma(t)\omega'(t)dt=\int_{-\delta}^\delta H_i\circ \gamma(t)\omega'(t)dt.
\end{align*}
\end{remark}

\subsection{Proof of Proposition~\ref{po:2}}

In this subsection, we prove Proposition~\ref{po:2} assuming
Lemma~\ref{th:7}. If \(u=0\), the conclusion is immediate. Hence assume that
\(u\neq0\).

Let \(\gamma\) be any regular curve along \(\mathcal F\) such that
\[
        \gamma(0)=x,
        \qquad
        \gamma'(0)=u.
\]
By Lemma~\ref{th:7}, after shrinking the interval if necessary, for every
\(\omega\in C_c^1((-\delta,\delta))\) we have
\[
\int_{-\delta}^{\delta}
(H\circ\gamma)(t)\omega'(t)\,dt
=
-\int_{-\delta}^{\delta}
\mathcal D_i(\gamma(t))(\gamma'(t))\omega(t)\,dt.
\]
The integrals are understood after choosing a local lift of \(H\circ\gamma\)
to \(\mathbb R^N\). Thus the lifted curve \(H\circ\gamma\) has weak derivative
\[
        t\mapsto \mathcal D_i(\gamma(t))(\gamma'(t)).
\]
Since \(\mathcal D_i\) is H\"older and \(\gamma\) is \(C^{1+\alpha}\), this
weak derivative is continuous. Therefore \(H\circ\gamma\) is \(C^1\), and
\[
        \frac{d}{dt}(H\circ\gamma)(t)
        =
        \mathcal D_i(\gamma(t))(\gamma'(t)).
\]
Evaluating at \(t=0\), we obtain
\[
        \left.\frac{d}{dt}\right|_{t=0}H\circ\gamma(t)
        =
        \mathcal D_i(x)(u).
\]
The right-hand side depends only on \(x\) and \(u\), not on the chosen regular
curve \(\gamma\). Hence \(H\) is curve differentiable along \(u\), and
\[
        dH_x^{\mathfrak c}(u)=\mathcal D_i(x)(u).
\]
This proves Proposition~\ref{po:2}.

It remains to prove Lemma~\ref{th:7}; this will occupy the rest of the section.

\subsection{Proof strategy for Lemma~\ref{th:7}}

We explain the strategy of the proof of Lemma~\ref{th:7}. Since the
integrability of $\mathcal F$ is not known a priori, we cannot begin with
leafwise differentiability. Instead, we prove differentiability along curves.
Fix a regular curve $\gamma:(-\delta,\delta)\to \TT^N$ tangent to
$\mathcal F$. The goal is to prove that, for every
$\omega\in C_c^1((-\delta,\delta))$,
\[
        \int_{-\delta}^{\delta}
        (H_i\circ\gamma)(t)\omega'(t)\,dt
        =
        -
        \int_{-\delta}^{\delta}
        \mathcal D_i(\gamma(t))(\gamma'(t))\omega(t)\,dt .
\]

Formally, using the expansion of $H_i$ from \eqref{for:92}, one would like to
write
\begin{align*}
&\int_{-\delta}^{\delta}
        (H_i\circ\gamma)(t)\omega'(t)\,dt\\
&=
\int_{-\delta}^{\delta}
        p_i(\gamma(t))\omega'(t)\,dt
+
\sum_{m=0}^{\infty}
\int_{-\delta}^{\delta}
        A_i^{-(m+1)}(R_i\circ f^m)(\gamma(t))\omega'(t)\,dt  \\
&=
-\int_{-\delta}^{\delta}
        p_i(\gamma'(t))\omega(t)\,dt
-
\sum_{m=0}^{\infty}
\int_{-\delta}^{\delta}
        A_i^{-(m+1)}
        D(R_i\circ f^m)_{\gamma(t)}(\gamma'(t))
        \omega(t)\,dt .
\end{align*}
Since $\gamma'(t)\in\mathcal F_{\gamma(t)}$, it is natural to use the
conjugacy relation \eqref{for:70} from Theorem~\ref{th:9} to replace the
growth of $Df^m|_{\mathcal F}$ by the linear growth of $A_i^m$. However, the
resulting series is difficult to control directly on the one-dimensional curve.
The main idea is therefore to approximate the curve distribution by normalized
averages over small plaques in a foliation box. This converts the curve-level
problem into an averaged problem on $\TT^N$, where the estimates developed
earlier can be applied.

More precisely, we choose a foliation chart $\Phi_k$ adapted to $\gamma$, so
that
\[
        \Phi_k(t,0,0)=\gamma(t).
\]
Here $t$ is the distinguished direction along the curve, while the remaining
leaf variables and transverse variables are denoted by $w$ and $y$. Let
\[
        \mathcal V_{\Phi_k(t,w,y)}
        :=
        D\Phi_k(t,w,y)(\partial_t).
\]
We then define a normalized averaged expression
$\mathcal N_{\mathcal V,\epsilon}$ (see \eqref{for:15}) over a tube $O_\epsilon$ around $\gamma$.
Here $\epsilon$ measures the size of the tube in the $w$- and $y$-directions,
and \(c_\epsilon\) is the inverse of the corresponding
normalizing volume. The factor
\(\phi_1^\epsilon(w)\phi_2^\epsilon(y)\) in the definition of
\(\mathcal N_{\mathcal V,\epsilon}\) localizes the average near the central
curve. As \(\epsilon\to0\), the normalized tube averages converge to the
corresponding integral along \(\gamma\). In particular, one obtains
\[
        \int_{-\delta}^{\delta}
        H_i(\gamma(t))\omega'(t)\,dt
        =
        \lim_{\epsilon\to0}\mathcal N_{\mathcal V,\epsilon}
        -
        \int_{-\delta}^{\delta}
        p_i(\gamma'(t))\omega(t)\,dt,\qquad \text{see Section }\ref{sec:30}.
\]
However, although \(\mathcal V\) agrees with the curve direction on the central
curve
\begin{align*}
 \Phi_k(t,0,0)=\gamma(t),
\end{align*}
 away from the curve it need not take values
in \(\mathcal F\). Hence \eqref{for:70} from Theorem~\ref{th:9} cannot be
applied directly. We therefore introduce a vector field \(\mathcal U\) which
takes values in \(\mathcal F\) and agrees with \(\mathcal V\) on the central
curve. The difference between \(\mathcal V\) and \(\mathcal U\) on the
\(\epsilon\)-neighborhood is \(O(\epsilon^\alpha)\). Combining this with the
averaging estimate of Lemma~\ref{le:6}, we obtain
\[
        \lim_{\epsilon\to0}\mathcal N_{\mathcal V,\epsilon}
        =
        \lim_{\epsilon\to0}\mathcal N_{\mathcal U,\epsilon}.
\]
Hence we may replace \(\mathcal V\) by the more convenient field
\(\mathcal U\); see Section \ref{sec:31}.

Finally, we compute the limit involving \(\mathcal N_{\mathcal U,\epsilon}\).
After applying the reduction and the approximate-identity argument, the finite
sums split into two terms; see \eqref{for:141}. The first term converges to
\[
        \int_{-\delta}^{\delta}
        p_i(\gamma'(t))\omega(t)\,dt,
        \qquad\text{see }\eqref{for:143}.
\]
The second term is the only place where the mixing property of \(A\) enters.
More precisely, after rewriting the foliation-box integral with respect to
Lebesgue measure on \(\mathbb T^N\), we conjugate by \(H\). The iterates
\(f^{k_m}\) then become the linear iterates \(A^{k_m}\). The mixing of \(A\)
allows us to replace the oscillating factor by its space average; see
\eqref{for:100}. This produces the averaged linear map \(B_i\), and hence the
operator
\[
        \mathcal D_i(x)=B_i\circ\mathcal C_{i,1}(x).
\]
Consequently,
\[
        \lim_{\epsilon\to0}\mathcal N_{\mathcal U,\epsilon}
        =
        \int_{-\delta}^{\delta}
        \mathcal T_i(\gamma(t))(\gamma'(t))\omega(t)\,dt ,
\]
where
\[
        \mathcal T_i(x)
        =
        p_i|_{\mathcal F_x}
        -
        \mathcal D_i(x).
\]
Combining this with
\[
        \lim_{\epsilon\to0}\mathcal N_{\mathcal V,\epsilon}
        =
        \lim_{\epsilon\to0}\mathcal N_{\mathcal U,\epsilon},
\]
we obtain
\[
\begin{aligned}
\int_{-\delta}^{\delta}
        H_i(\gamma(t))\omega'(t)\,dt
&=
\int_{-\delta}^{\delta}
        \mathcal T_i(\gamma(t))(\gamma'(t))\omega(t)\,dt
-
\int_{-\delta}^{\delta}
        p_i(\gamma'(t))\omega(t)\,dt \\
&=
-
\int_{-\delta}^{\delta}
        \mathcal D_i(\gamma(t))(\gamma'(t))\omega(t)\,dt .
\end{aligned}
\]
By Remark~\ref{re:4}, the same identity holds with \(H\) in place of \(H_i\).
This proves Lemma~\ref{th:7}.

\subsection{Notation and basic facts} We collect the notation and basic facts that will be used in the proof.

\subsubsection{Absolute continuity of $\mathcal W_i^f$}\label{sec:15}

We explain why the conditional measures of Lebesgue measure along
$\mathcal W_i^f$ are absolutely continuous, under the assumption $H(\mathcal W_i^f)=\mathcal W_i^A$. The argument is similar to those in \cite{KS20} and \cite{Yang}.

Let $\xi$ be a measurable partition subordinate to $\mathcal W_i^f$, and let $\xi^A:=H\xi$
be the corresponding measurable partition subordinate to the linear foliation
$\mathcal W_i^A$. Recall that $\mu=(H^{-1})_*\mathfrak m$ is $f$-invariant.
We have
\[
H_\mu(f^{-1}\xi\mid \xi)
=
H_{\mathfrak m}(A^{-1}\xi^A\mid \xi^A).
\]
For the linear foliation $\mathcal W_i^A$, the conditional measures of
$\mathfrak m$ are Lebesgue measures on the affine leaves. Hence Ledrappier's
entropy formula gives
\[
H_{\mathfrak m}(A^{-1}\xi^A\mid \xi^A)
=
\log |\det(A|_{E_i})|.
\]
On the other hand, since the Lyapunov exponents of $f$ along
$\mathcal W_i^f$ agree with the corresponding Lyapunov exponents of $A$, we
have
\[
\log |\det(A|_{E_i})|
=
\int \log \operatorname{Jac}(f|_{\mathcal W_i^f})(x)\,d\mu(x).
\]
Therefore,
\[
H_\mu(f^{-1}\xi\mid \xi)
=
\int \log \operatorname{Jac}(f|_{\mathcal W_i^f})(x)\,d\mu(x).
\]
By Ledrappier's criterion \cite{L}, the conditional measures of $\mu$ along the leaves
of $\mathcal W_i^f$ are absolutely continuous with respect to leafwise
Lebesgue measure. Since $\mu$ is equivalent to Lebesgue measure
$\mathfrak m$ (see Section \ref{sec:13}), the same absolute-continuity statement holds for the
conditional measures of $\mathfrak m$ along $\mathcal W_i^f$.

\subsubsection{Foliation charts for $\mathcal{W}^f_i$} \label{sec:16}
Since $\mathcal W_i^f$ is a \text{H\"older} foliation with uniformly
$C^{1+\alpha}$ leaves, we may choose a finite cover $\{U_k\}_{k\in I}$ of
$\TT^N$ by foliation charts
\[
\Gamma_k:O_k\subset
\RR^{\dim\mathcal E_i}\times \RR^{N-\dim\mathcal E_i}
\to U_k.
\]
We write points in $O_k$ as $(x,y)$, where $x\in\RR^{\dim\mathcal E_i}$ and $y\in\RR^{N-\dim\mathcal E_i}$.
For each fixed $y$, the plaque
\[
\mathfrak r_y
:=
\bigl(\RR^{\dim\mathcal E_i}\times\{y\}\bigr)\cap O_k
\]
is mapped by $\Gamma_k$ into a local leaf of $\mathcal W_i^f$. Moreover,
$\Gamma_k$ is $C^{1+\alpha}$ in the leaf variable $x$, uniformly in the
foliation chart.

Since $\mathcal W_i^f$ is absolutely continuous (see Section \ref{sec:15}), the restriction of
$\mathfrak{m}$ to each foliation box admits a disintegration along the local
plaques. Thus, for every continuous function $g$ compactly supported in $O_k$
and every bounded measurable function $\omega$ on $U_k$, we have
\begin{align}\label{for:167}
 \int_{U_k}
\omega(z)g(\Gamma_k^{-1}z)\,d\mathfrak m(z)
=
\int
\left(
\int
\omega(\Gamma_k(x,y))g(x,y)J_k(x,y)\,dx
\right)
d\nu_k(y),
\end{align}
where $\nu_k$ is a transverse measure and $J_k(x,y)>0$ is the leafwise
density of the conditional measure of $\mathfrak m$ on the plaque
$\Gamma_k(\mathfrak r_y)$.

Moreover, the leafwise densities may be chosen to be H\"older in the leaf
variable \(x\), uniformly in the transverse parameter \(y\). Indeed,
Ledrappier's criterion gives conditional densities for \(\mu\) along
\(\mathcal W_i^f\). As in \cite{KS20}, these densities are given, up to
normalization, by ratios of leafwise Jacobians. Since
\(\mu=\kappa\,\mathfrak m\), where \(\kappa\) is positive and \(C^\alpha\) by
Section~\ref{sec:13}, the corresponding conditional densities for
\(\mathfrak m\) are obtained, up to plaque-dependent normalization, by
dividing by \(\kappa\) along the plaques. The normalization constants depend
only on the plaque and therefore do not affect H\"older regularity in the leaf
variable. Hence the densities \(J_k(x,y)\) may be chosen H\"older in the leaf
variable \(x\), uniformly in the transverse parameter \(y\).

We say that a function $\omega$ on $\TT^N$ is $C^{1+\alpha}$ along
$\mathcal W_i^f$ if, in every foliation chart $\Gamma_k$, the function $x\mapsto \omega(\Gamma_k(x,y))$
is $C^{1+\alpha}$ for each fixed $y$, with uniform $C^{1+\alpha}$ bounds in
the leaf variable.

\subsubsection{Rearranging coordinates along a curve}\label{sec:11}

Let $\gamma:(-1,1)\to \TT^N$ be a $C^{1+\alpha}$ curve tangent to
$\mathcal E_i$, that is, $\gamma'(t)\in (\mathcal E_i)_{\gamma(t)}$. Assume that $\gamma'(0)\neq0$. Choose a foliation chart
$(U_k,\Gamma_k)$ such that $\gamma(0)\in U_k$. After shrinking the interval of
definition of $\gamma$, we may assume that $\gamma(t)\in U_k$ for all relevant
$t$.

Since $\gamma$ is tangent to $\mathcal E_i=T\mathcal W_i^f$, the curve
$\gamma$ is contained in one local leaf of $\mathcal W_i^f$. Hence, in the
foliation chart, we may write
\[
\Gamma_k^{-1}\gamma(t)=(a(t),y_0),
\qquad a(t)\in \RR^{\dim\mathcal E_i}.
\]
After a linear change of the leaf coordinate \(x\in \RR^{\dim\mathcal E_i}\), we may assume that the
first component of \(a'(0)\) is nonzero. Shrinking the interval again if
necessary, we may assume that the first component of \(a'(t)\) is nonzero for
all relevant \(t\). Write
\[
w=(x_2,\ldots,x_{\dim\mathcal E_i})\in\RR^{\dim\mathcal E_i-1}.
\]
Define a new chart by
\[
\Phi_k(t,w,y)
=
\Gamma_k\bigl(a(t)+(0,w),\,y_0+y\bigr),
\]
on a sufficiently small domain. Then
\[
\Phi_k(t,0,0)=\gamma(t).
\]
Moreover, the map $(t,w)
\mapsto
a(t)+(0,w)$ is a local \(C^{1+\alpha}\) diffeomorphism in the leaf variables. Therefore,
after possibly shrinking the domain, \(\Phi_k\) is again a foliation chart for
\(\mathcal W_i^f\), with the same regularity properties as \(\Gamma_k\). In
particular, \(\Phi_k\) is \(C^{1+\alpha}\) in the leaf variables and its
leafwise derivative is \(\alpha\)-H\"older.

The disintegration formula remains valid in the new chart. Namely, for every
continuous function \(g\) compactly supported in the domain of \(\Phi_k\) and
every bounded measurable function \(\omega\) on \(U_k\), we have
\begin{align}\label{for:138}
\int_{U_k}
\omega(z)g(\Phi_k^{-1}z)\,d\mathfrak m(z)
=
\int
\left(
\int
\omega(\Phi_k(t,w,y))g(t,w,y)J_k(t,w,y)\,dt\,dw
\right)
d\nu_k(y),
\end{align}
where, after renaming the density, \(J_k(t,w,y)>0\) is
\(\alpha\)-\text{H\"older} in the leaf variable \((t,w)\), uniformly in \(y\).

\subsubsection{Smooth approximations}\label{sec:17}  We will use standard smooth approximations by convolution. Let
\(\phi\in C_c^\infty(\mathbb R^N)\) be a nonnegative bump function supported
in the unit ball and satisfying \(\int_{\mathbb R^N}\phi=1\). For
\(0<\varepsilon<1\), set
\[
        \phi_\varepsilon(x)=\varepsilon^{-N}\phi(x/\varepsilon),
        \qquad
        \omega_\varepsilon=\omega*\phi_\varepsilon .
\]
We note that $\phi_\varepsilon$
is supported on the ball of radius $\varepsilon$ centered
at $0$ and $\int_{\RR^N} \phi_\varepsilon=1$. Moreover,
\begin{align}\label{for:108}
 \norm{(I-\Delta)^m\phi_\varepsilon}_{L^1}\leq c_m\varepsilon^{-2m},
\end{align}
where $\Delta$ is the Laplacian on $\RR^N$, and $c_m$ a constant dependent only on $m$. If $\omega\in C^a$, $a>0$, then $\omega_\varepsilon$ is $C^\infty$  with estimates
\begin{align}
 \norm{\omega_\varepsilon-\omega}_{C^0}&\leq C_a\varepsilon^a\norm{\omega}_{C^a},\label{for:107}\\
 \norm{\omega_\varepsilon}_{C^r}&\leq C_r\varepsilon^{-r}\norm{\omega}_{C^0},\quad\text{for all }r\geq0. \label{for:142}
\end{align}

\subsection{Proof of Lemma \ref{th:7}}

\subsubsection{Preparatory step}\label{sec:25}
Let $\gamma:(-1,1)\to \TT^N$ be a regular curve along
  $\mathcal F$.  Choose a
foliation chart $(U_k,\Phi_k)$ as in Section~\ref{sec:11}. We denote by $O_k$
the domain of the rearranged chart $\Phi_k$ associated with $\gamma$.  After
shrinking the interval if necessary, there exists \(\delta>0\) such that
\begin{align*}
 \gamma((-\delta,\delta))\subset U_k\quad\text{and}\quad \Phi_k(t,0,0)=\gamma(t),
\qquad t\in(-\delta,\delta).
\end{align*}
We write the coordinates of $\Phi_k$ as
\[
(t,w,y),
\qquad
w\in\RR^{\dim\mathcal E_i-1},
\qquad
y\in\RR^{N-\dim\mathcal E_i}.
\]
Let $\phi_1(w)\geq0$ and $\phi_2(y)\geq0$ be two $C^\infty$ nonnegative bump functions such that
\[
\phi_1\equiv 1 \quad \text{on } B(0,1/2),
\qquad
\operatorname{supp}\phi_1\subset B(0,1),
\]
and
\[
\phi_2\equiv 1 \quad \text{on } B(0,1/2),
\qquad
\operatorname{supp}\phi_2\subset B(0,1).
\]
For $\epsilon>0$, set
\begin{align*}
 \phi_1^\epsilon(w)
:=
\phi_1\left(\frac{w}{\epsilon}\right)\quad\text{and}\quad \phi_2^\epsilon(y)=\phi_2\left(\frac{y}{\epsilon} \right).
\end{align*}
Choose $\epsilon_0>0$ sufficiently small so that
\[
        O_{\epsilon_0}
        :=
        \{(t,w,y): |t|\leq \delta,\ \|w\|\leq \epsilon_0,\ \|y\|\leq \epsilon_0\}\subset O_k.
\]
Define the product measure in these coordinates by
\begin{align}\label{for:101}
        d\widehat\mu(t,w,y):=dt\,dw\,d\nu_k(y).
\end{align}

\begin{remark}\label{re:5}
In the computations below, we use the product measure
$d\widehat\mu(t,w,y)$ rather than the full foliation-box measure
\[
        J_k(t,w,y)\,dt\,dw\,d\nu_k(y).
\]
The reason is that we need to differentiate in the leaf direction and integrate
by parts. If the factor $J_k(t,w,y)$ were included in the measure, then
integration by parts would produce additional terms involving derivatives of
$J_k$.

This causes no loss. Whenever we need to compare the product measure
$d\widehat\mu$ with Lebesgue measure $\mathfrak m$, we absorb the Jacobian
factor $J_k(t,w,y)$ into the test function. Since $J_k$ is regular along the
leaves, with estimates uniform in the transverse variable $y$, multiplying by
$J_k$ preserves the H\"older class and the relevant estimates. Thus the use of
$d\widehat\mu$ is only a technical convenience and does not affect the
averaging argument.
\end{remark}

For $0<\epsilon\leq\epsilon_0$, define the normalizing constant
\[
c_\epsilon:= \Big(\int\phi_1^\epsilon(w)\phi_2^\epsilon(y)dw\,d\nu_k(y)\Big)^{-1}.
\]
Equivalently,
\[
c_\epsilon=
\epsilon^{-(\dim\mathcal E_i-1)}
\left(\int \phi_1(w)\,dw\right)^{-1}
\left(
\int \phi_2\left(\frac{y}{\epsilon}\right)d\nu_k(y)
\right)^{-1}.
\]
Recall the local identifications of nearby fibers from Section~\ref{sec:6}. Applying them to the bundle $\mathcal F$, and after possibly shrinking the
domain once more, we assume that, for every
$(t,w,y)\in O_k$, the two points
\[
        \Phi_k(t,0,0)
        \quad\text{and}\quad
        \Phi_k(t,w,y)
\]
lie in a common trivializing neighborhood for $\mathcal F$. Thus the local
fiber identification
\[
        \mathcal I^{\mathcal F}_{\Phi_k(t,0,0),\,\Phi_k(t,w,y)}
\]
is well defined throughout the chart.

Next, we introduce two $\alpha$-H\"older vector
fields on $\Phi_k(O_k)$. First, let
\[
\mathcal V_{\Phi_k(t,w,y)}
:=
D\Phi_k(t,w,y)(\partial_t).
\]
Since $t$ is a leaf coordinate, we have
\[
\mathcal V_{\Phi_k(t,w,y)}
\in
(\mathcal E_i)_{\Phi_k(t,w,y)}.
\]
However, except on the curve $\Phi_k(t,0,0)=\gamma(t)$, the vector field
$\mathcal V$ need not belong to $\mathcal F$. This is why we introduce a second
vector field.

 Define a vector
field $\mathcal U$ on $\Phi_k(O_{k})$ by
\[
\mathcal U_{\Phi_k(t,w,y)}
:=
\mathcal I^{\mathcal F}_{\Phi_k(t,0,0),\,\Phi_k(t,w,y)}
\bigl(\gamma'(t)\bigr).
\]
Then $\mathcal U_{\Phi_k(t,w,y)}
\in
\mathcal F_{\Phi_k(t,w,y)}$ for all $(t,w,y)\in O_{k}$. Moreover, since
\begin{align*}
 \Phi_k(t,0,0)=\gamma(t)\quad\text{and}\quad D\Phi_k(t,0,0)(\partial_t)=\gamma'(t),
\end{align*}
we have
\begin{equation}\label{for:105}
\mathcal U_{\Phi_k(t,0,0)}
=
\mathcal V_{\Phi_k(t,0,0)},
\qquad -\delta\leq t\leq \delta.
\end{equation}
Equivalently,
\begin{equation}\label{for:104}
\mathcal U_{\gamma(t)}
=
\mathcal V_{\gamma(t)},
\qquad -\delta\leq t\leq \delta.
\end{equation}
Both vector fields are $\alpha$-H\"older on $\Phi_k(O_{k})$. For
$\mathcal V$, this follows from the $\alpha$-H\"older regularity of the
leafwise derivative $D_{(t,w)}\Phi_k$. For $\mathcal U$, it follows from the
$C^{1+\alpha}$ regularity of $\gamma$, the $\alpha$-H\"older regularity of the
bundle $\mathcal F$, and the $\alpha$-H\"older dependence of the local fiber
identifications $\mathcal I^{\mathcal F}_{x,y}$.

At the end of this section, we prove the following decay estimate,  which will be used in the subsequent parts. We use the product measure $d\widehat\mu$ defined in \eqref{for:101}.
\begin{lemma}\label{le:6}  Let \(\mathcal Q\) be an \(\alpha\)-H\"older vector field on
\(\Phi_k(O_k)\) taking values in \(\mathcal E_i\). Then, for every
$\omega\in C_c^1(-\delta,\delta)$, for every \(m\geq0\)
and every sufficiently small \(\epsilon>0\), we have
\begin{align*}
 &\Big\|c_\epsilon\int_{O_\epsilon} A_i^{-(m+1)}D(R_i\circ f^{m})_{\Phi_k(t,w,y)}(\mathcal Q_{\Phi_k(t,w,y)})\omega(t)\phi_1^\epsilon(w)\phi_2^\epsilon(y)d\widehat\mu\Big\|\\
 &\leq C_{\phi_1}\epsilon^{-1}\rho_i^{-\frac{\alpha(m+1)}{2}}(|m|+1)^{\dim \mathcal{E}_i-1}\norm{R_i}_{C^1}\norm{\omega}_{C^1}\norm{\mathcal{Q}}_{C^\alpha}.
\end{align*}
\end{lemma}
\begin{proof}
It follows from \eqref{for:33} of Section~\ref{sec:12} and the fact that
\(A_i\) is diagonalizable with all eigenvalues of modulus \(\rho_i\) (see Section \ref{sec:14}) that, for
every \(m\geq0\),
\begin{align}\label{for:34}
 \norm{A_i^{-(m+1)}D(R_i\circ f^{m})|_{\mathcal{E}_i}}\leq C (|m|+1)^{\dim \mathcal{E}_i-1}\norm{R_i}_{C^1}.
\end{align}
For $(t,w,y)\in O_{k}$, define the leafwise coordinate vector fields
\begin{align*}
  \mathcal W^0_{\Phi_k(t,w,y)}
        :=
        D\Phi_k(t,w,y)(\partial_t),
\end{align*}
and, for $1\leq j\leq \dim \mathcal{E}_i-1$,
\begin{align*}
  \mathcal W^j_{\Phi_k(t,w,y)}
        :=
        D\Phi_k(t,w,y)(\partial_{w_j}).
\end{align*}
Since \(\mathcal Q\) takes values in \(\mathcal E_i\), we can write
\begin{align*}
 \mathcal Q_{\Phi_k(t,w,y)}=\sum_{j=0}^{\dim \mathcal{E}_i-1} a_j(t,w,y)\mathcal W^j_{\Phi_k(t,w,y)}.
\end{align*}
The coefficients \(a_j\) are \(\alpha\)-H\"older in the leaf variables
\((t,w)\), uniformly in \(y\), and satisfy
\[
\|a_j\|_{C^0}\leq C\|\mathcal Q\|_{C^0},
\qquad
\|a_j\|_{C^\alpha_{t,w}}\leq C\|\mathcal Q\|_{C^\alpha}.
\]
We smooth \(a_j\) only in the leaf variables \((t,w)\), treating \(y\) as a
parameter. Let \(a_{j,\eta}\) be the
corresponding \(C^\infty\) approximation as in Section~\ref{sec:17}. Then the corresponding estimates hold uniformly in
\(y\):
\begin{align}
 \norm{a_{j,\eta}-a_j}_{C^0}&\leq C\eta^\alpha\norm{a_j}_{C^\alpha},\label{for:109}\\
 \norm{a_{j,\eta}}_{C^1_{t,w}}&\leq C\eta^{-1}\norm{a_j}_{C^0}\label{for:115}.
\end{align}
For each \(j\), set
\[
b_j(t,w,y)
:=
a_j(t,w,y)\omega(t)\phi_1^\epsilon(w)\phi_2^\epsilon(y),
\]
and
\[
b_{j,\eta}(t,w,y)
:=
a_{j,\eta}(t,w,y)\omega(t)\phi_1^\epsilon(w)\phi_2^\epsilon(y).
\]
Then we have
\begin{align*}
 &\int_{O_\epsilon} A_i^{-(m+1)}D(R_i\circ f^{m})_{\Phi_k(t,w,y)}(\mathcal Q_{\Phi_k(t,w,y)})\omega(t)\phi_1^\epsilon(w)\phi_2^\epsilon(y)d\widehat\mu\\
 &=\sum_{j=0}^{\dim \mathcal{E}_i-1}\int_{O_\epsilon} A_i^{-(m+1)}D(R_i\circ f^{m})_{\Phi_k(t,w,y)}(\mathcal W^j_{\Phi_k(t,w,y)})b_j(t,w,y)d\widehat\mu\\
&=\sum_{j=0}^{\dim \mathcal{E}_i-1}\mathcal{X}_j+\sum_{j=0}^{\dim \mathcal{E}_i-1}\mathcal{Z}_j,
\end{align*}
where
\begin{align*}
 \mathcal{X}_j&=\int_{O_\epsilon} A_i^{-(m+1)}D(R_i\circ f^{m})_{\Phi_k(t,w,y)}(\mathcal W^j_{\Phi_k(t,w,y)})b_{j,\eta}(t,w,y)d\widehat\mu\\
 \mathcal{Z}_j&=\int_{O_\epsilon} A_i^{-(m+1)}D(R_i\circ f^{m})_{\Phi_k(t,w,y)}(\mathcal W^j_{\Phi_k(t,w,y)})(b_j-b_{j,\eta})(t,w,y)d\widehat\mu.
\end{align*}
We first estimate \(\|c_\epsilon\mathcal X_j\|\).  Write \(z_0=t\) and \(z_j=w_j\) for
\(1\leq j\leq \dim \mathcal{E}_i-1\). We note that
\[
D(R_i\circ f^m)_{\Phi_k(t,w,y)}
\bigl(\mathcal W^j_{\Phi_k(t,w,y)}\bigr)
=
\partial_{z_j}(R_i\circ f^m\circ\Phi_k)(t,w,y).
\]
Then we have
\begin{align*}
 \mathcal{X}_j=-\int_{O_\epsilon} A_i^{-(m+1)}(R_i\circ f^{m}\circ\Phi_k)(t,w,y)\big(\partial_{z_j}b_{j,\eta}(t,w,y)\big)d\widehat\mu.
\end{align*}
Since \(\omega\) is compactly supported in \(t\) and
\(\phi_1^\epsilon\) is compactly supported in \(w\), no boundary terms appear
when integrating by parts in the leaf variables.

Hence,
\begin{align*}
 \norm{c_\epsilon\mathcal{X}_j}&\leq \Big\|A_i^{-(m+1)}(R_i\circ f^{m}\circ\Phi_k)(t,w,y)\big(\partial_{z_j}b_{j,\eta}(t,w,y)\Big)\big\|_{C^0}\\
 &\leq C\norm{A_i^{-(m+1)}}\norm{R_i}_{C^0}\cdot \big(\epsilon^{-1}\norm{a_{j,\eta}}_{C^1_{t,w}} \norm{\omega}_{C^1}\norm{\phi_1}_{C^1}\big)\\
 &\overset{\text{(1)}}{\leq} C\epsilon^{-1}\eta^{-1}\rho_i^{-(m+1)} \norm{R_i}_{C^0}\|\mathcal Q\|_{C^0}\norm{\omega}_{C^1}\norm{\phi_1}_{C^1}.
\end{align*}
Here in $(1)$ we use \eqref{for:115}.

We next estimate each $\norm{c_\epsilon\mathcal{Z}_j}$. It follows from \eqref{for:34} and \eqref{for:109} that
\begin{align*}
\norm{c_\epsilon\mathcal{Z}_j} &\leq \Big\| A_i^{-(m+1)}D(R_i\circ f^{m})_{\Phi_k(t,w,y)}(\mathcal W^j_{\Phi_k(t,w,y)})(b_j-b_{j,\eta})(t,w,y)  \Big\|_{C^0}\\
&\leq C\big\|A_i^{-(m+1)}D(R_i\circ f^{m})|_{\mathcal{E}_i}\big\|\cdot \norm{a_{j,\eta}-a_j}_{C^0}\norm{\omega}_{C^0}\notag\\
 &\leq C_1\eta^\alpha(|m|+1)^{\dim \mathcal{E}_i-1}\norm{R_i}_{C^1}\norm{\omega}_{C^0}\norm{\mathcal{Q}}_{C^\alpha}.
\end{align*}
Combining the estimates for \(\mathcal X_j\) and \(\mathcal Z_j\), and summing
over \(0\leq j\leq \dim \mathcal{E}_i-1\), we obtain
\begin{align*}
 &\Big\|c_\epsilon\int_{O_\epsilon} A_i^{-(m+1)}D(R_i\circ f^{m})_{\Phi_k(t,w,y)}(\mathcal Q_{\Phi_k(t,w,y)})\omega(t)\phi_1^\epsilon(w)\phi_2^\epsilon(y)d\widehat\mu\Big\|\\
 &\leq C_{\phi_1}\big(\epsilon^{-1}\eta^{-1}\rho_i^{-(m+1)}+ \eta^\alpha(|m|+1)^{\dim \mathcal{E}_i-1}\big)\norm{R_i}_{C^1}\norm{\omega}_{C^1}\norm{\mathcal{Q}}_{C^\alpha}.
\end{align*}
Let $\eta=\rho_i^{-\frac{m+1}{2}}$. Then
\begin{align*}
 \rho_i^{-(m+1)}\eta^{-1}
=
\rho_i^{-\frac{m+1}{2}},\quad \eta^\alpha
=
\rho_i^{-\frac{\alpha(m+1)}{2}}, \quad\text{and}\quad \rho_i^{-\frac{m+1}{2}}
\leq
\rho_i^{-\frac{\alpha(m+1)}{2}}.
\end{align*}

Also, since \(\epsilon>0\) is small, we have
\begin{align*}
&\left\|
c_\epsilon
\int_{O_\epsilon}
A_i^{-(m+1)}
D(R_i\circ f^m)_{\Phi_k(t,w,y)}
\bigl(\mathcal Q_{\Phi_k(t,w,y)}\bigr)
\omega(t)\phi_1^\epsilon(w)\phi_2^\epsilon(y)
\,d\widehat\mu
\right\| \\
&\qquad\leq
C_{\phi_1}
\epsilon^{-1}
\rho_i^{-\frac{\alpha(m+1)}{2}}
(m+1)^{\dim\mathcal{E}_i-1}
\|R_i\|_{C^1}
\|\omega\|_{C^1}
\|\mathcal Q\|_{C^\alpha}.
\end{align*}
This proves the lemma.
\end{proof}
Lemma \ref{le:6} shows that for any $\alpha$-H\"older vector field $\mathcal{Q}$ on $\Phi_k(O_k)$ taking values in \(\mathcal E_i\), the
following series converges absolutely for every fixed sufficiently small
$\epsilon>0$:
\begin{align}\label{for:15}
 \mathcal{N}_{\mathcal{Q},\epsilon}&=-c_\epsilon \sum_{m=0}^\infty\int_{O_\epsilon} A_i^{-(m+1)}D(R_i\circ f^{m})_{\Phi_k(t,w,y)}(\mathcal Q_{\Phi_k(t,w,y)})\omega(t)\phi_1^\epsilon(w)\phi_2^\epsilon(y)d\widehat\mu.
\end{align}

\subsubsection{Step 1: Transfer to vector field}\label{sec:30}  Suppose $\omega(t)$ is compacted supported on $(-\delta,\delta)$.  In this step, we show that
\begin{align*}
 \int_{-\delta}^\delta H_i(\gamma(t))\omega'(t)dt=\lim_{\epsilon\to0 }\mathcal{N}_{\mathcal{V},\epsilon}-\int_{-\delta}^\delta p_i(\gamma'(t))\omega(t)dt.
\end{align*}
Recall that $\mathcal{V}$ is defined in Section \ref{sec:25}.

Here, as in \eqref{for:7} of Section~\ref{sec:39}, we write
\[
        H_i:=p_i\circ H=p_i+h_i .
\]
First, we show that
\begin{align}\label{for:111}
 \int_{-\delta}^\delta H_i(\gamma(t))\omega'(t)dt=\mathcal{Y}_1+\mathcal{Y}_2,
\end{align}
where
\begin{align*}
 \mathcal{Y}_1&=\lim_{\epsilon\to0 }c_\epsilon\int_{O_\epsilon} (h_i\circ \Phi_k)(t,w,y)\omega'(t)\phi_1^\epsilon(w)\phi_2^\epsilon(y)d\widehat\mu.\\
  \mathcal{Y}_2&=\lim_{\epsilon\to0 }c_\epsilon\int_{O_\epsilon} (p_i\circ \Phi_k)(t,w,y)\omega'(t)\phi_1^\epsilon(w)\phi_2^\epsilon(y)d\widehat\mu.
\end{align*}
Indeed, the normalized measures
\[
        c_\epsilon\phi_1^\epsilon(w)\phi_2^\epsilon(y)\,dw\,d\nu_k(y)
\]
converge weakly to the Dirac mass at $(w,y)=(0,0)$. Since
$\Phi_k(t,0,0)=\gamma(t)$, we obtain
\[
\begin{aligned}
        \int_{-\delta}^{\delta}
        H_i(\gamma(t))\omega'(t)\,dt
        &=
        \lim_{\epsilon\to0}
        c_\epsilon
        \int_{O_\epsilon}
        (H_i\circ\Phi_k)(t,w,y)\omega'(t)
        \phi_1^\epsilon(w)\phi_2^\epsilon(y)
        \,d\widehat\mu  \\
        &=
        \mathcal Y_1+\mathcal Y_2 .
\end{aligned}
\]
This proves \eqref{for:111}.

We now compute $\mathcal Y_1$. Using \eqref{for:92} of Section \ref{sec:14},  we have
\begin{align}\label{for:150}
 \mathcal{Y}_1&=\lim_{\epsilon\to0 }c_\epsilon \sum_{m=0}^\infty\int_{O_\epsilon} A_i^{-(m+1)}(R_i\circ f^{m}\circ \Phi_k)(t,w,y)\omega'(t)\phi_1^\epsilon(w)\phi_2^\epsilon(y)d\widehat\mu\notag\\
 &=-\lim_{\epsilon\to0 }c_\epsilon \sum_{m=0}^\infty\int_{O_\epsilon} A_i^{-(m+1)}D(R_i\circ f^{m})_{\Phi_k(t,w,y)}(\mathcal V_{\Phi_k(t,w,y)})\omega(t)\phi_1^\epsilon(w)\phi_2^\epsilon(y)d\widehat\mu\notag\\
 &=\lim_{\epsilon\to0 }\mathcal{N}_{\mathcal{V},\epsilon}.
\end{align}
Finally, we compute $\mathcal Y_2$. Since
$\Phi_k(t,0,0)=\gamma(t)$, the same averaging argument gives
\begin{align*}
\mathcal{Y}_2=\int_{-\delta}^{\delta}
        p_i(\gamma(t))\omega'(t)\,dt=-\int_{-\delta}^\delta p_i(\gamma'(t))\omega(t)dt.
\end{align*}
Combining this identity with \eqref{for:111} and \eqref{for:150} proves the result.
\subsubsection{Step 2: Replacing $\mathcal V$ by $\mathcal U$}\label{sec:31} In this step, we show that
\begin{align}\label{for:17}
 \lim_{\epsilon\to0 }\mathcal{N}_{\mathcal{V},\epsilon}&=\lim_{\epsilon\to0 }\mathcal{N}_{\mathcal{U},\epsilon}.
\end{align}
Recall that $\mathcal{V}$ and $\mathcal{U}$ are defined in Section \ref{sec:25}.  Since
\begin{align*}
 \mathcal N_{\mathcal V,\epsilon}
-
\mathcal N_{\mathcal U,\epsilon}
=
\mathcal N_{\mathcal V-\mathcal U,\epsilon},
\end{align*}
to prove \eqref{for:17}, it suffices to show that: for any $\epsilon>0$
\begin{align}\label{for:137}
\norm{\mathcal{N}_{(\mathcal{V}-\mathcal{U}),\epsilon}}\leq C_{\phi_1}\epsilon^{\frac{\alpha}{2}}\norm{R_i}_{C^1}\norm{\omega}_{C^1}\max\{\norm{\mathcal{U}}_{C^\alpha}, \,\norm{\mathcal{V}}_{C^\alpha}\}.
\end{align}
Let $n_\epsilon=\lfloor\epsilon^{-\frac{\alpha}{2\dim\mathcal{E}_i}}\rfloor+1$. Then
\begin{align}\label{for:132}
 (n_\epsilon+1)^{\dim \mathcal{E}_i}\epsilon^\alpha\leq C\epsilon^{\frac{\alpha}{2}}.
\end{align}
Moreover, since \(n_\epsilon+2\geq \epsilon^{-\frac{\alpha}{2\dim\mathcal{E}_i}}\), we have
\begin{align}\label{for:135}
 \epsilon^{-1}
\rho_i^{-\frac{\alpha(n_\epsilon+2)}{4}}
\leq
\epsilon^{-1}
e^{-c\epsilon^{-\frac{\alpha}{2\dim\mathcal{E}_i}}}\leq C\epsilon^{\frac{\alpha}{2}},
\end{align}
where \(c=\frac{\alpha\log\rho_i}{4}>0\), for all sufficiently small \(\epsilon>0\).

We split
\begin{align*}
\mathcal{N}_{(\mathcal{V}-\mathcal{U}),\epsilon}=\mathcal{M}_1+\mathcal{M}_2+\mathcal{M}_3,
\end{align*}
where
\begin{align*}
  \mathcal M_1&:=
        c_\epsilon
        \sum_{m=0}^{n_\epsilon}
        \int_{O_\epsilon}
        A_i^{-(m+1)}
        D(R_i\circ f^m)_{\Phi_k(t,w,y)}
        \bigl(\mathcal U_{\Phi_k(t,w,y)}
        -\mathcal V_{\Phi_k(t,w,y)}\bigr)  \\
        &\hspace{4.5cm}\cdot
        \omega(t)\phi_1^\epsilon(w)\phi_2^\epsilon(y)
        \,d\widehat\mu,\\
  \mathcal M_2
        &:=
        c_\epsilon
        \sum_{m=n_\epsilon+1}^{\infty}
        \int_{O_\epsilon}
        A_i^{-(m+1)}
        D(R_i\circ f^m)_{\Phi_k(t,w,y)}
        \bigl(\mathcal U_{\Phi_k(t,w,y)}\bigr)  \\
        &\hspace{4.5cm}\cdot
        \omega(t)\phi_1^\epsilon(w)\phi_2^\epsilon(y)
        \,d\widehat\mu,
\end{align*}
and
\[
\begin{aligned}
        \mathcal M_3
        &:=
        -c_\epsilon
        \sum_{m=n_\epsilon+1}^{\infty}
        \int_{O_\epsilon}
        A_i^{-(m+1)}
        D(R_i\circ f^m)_{\Phi_k(t,w,y)}
        \bigl(\mathcal V_{\Phi_k(t,w,y)}\bigr)  \\
        &\hspace{4.5cm}\cdot
        \omega(t)\phi_1^\epsilon(w)\phi_2^\epsilon(y)
        \,d\widehat\mu.
\end{aligned}
\]
The sign of $\mathcal M_3$ will play no role in the estimates.

We first estimate \(\mathcal M_1\).  On $\Phi_k(O_\epsilon)$ we have
\begin{align}\label{for:131}
\norm{\mathcal U_{\Phi_k(t,w,y)}&-\mathcal V_{\Phi_k(t,w,y)}}\overset{\text{(1)}}{=}\|\mathcal U_{\Phi_k(t,w,y)}-\mathcal U_{\Phi_k(t,0,0)}+\mathcal V_{\Phi_k(t,0,0)}-\mathcal V_{\Phi_k(t,w,y)}\|\notag\\
&\leq \|\mathcal U_{\Phi_k(t,w,y)}-\mathcal U_{\Phi_k(t,0,0)}\|+\|\mathcal V_{\Phi_k(t,0,0)}-\mathcal V_{\Phi_k(t,w,y)}\|\notag\\
&\overset{\text{(2)}}{\leq} C\epsilon^\alpha\max\{\norm{\mathcal{U}}_{C^\alpha}, \,\norm{\mathcal{V}}_{C^\alpha}\}.
\end{align}
Here in $(1)$ we use \eqref{for:105}; in $(2)$ we note that the points $\Phi_k(t,w,y)$ and $\Phi_k(t,0,0)$ are at distance
$O(\epsilon)$ in the chart, and both vector fields are $\alpha$-H\"older.

Then it follows that
\begin{align*}
 \norm{\mathcal{M}_1}&\leq \sum_{m=0}^{n_\epsilon}\big\|A_i^{-(m+1)}D(R_i\circ f^{m})|_{\mathcal{E}_i}\big\|\cdot \norm{\mathcal U-\mathcal V}_{C^0}\norm{\omega}_{C^0}\\
 &\overset{\text{(1)}}{\leq} \sum_{m=0}^{n_\epsilon} C (|m|+1)^{\dim \mathcal{E}_i-1}\norm{R_i}_{C^1}\cdot C\epsilon^\alpha\max\{\norm{\mathcal{U}}_{C^\alpha}, \,\norm{\mathcal{V}}_{C^\alpha}\}\cdot\norm{\omega}_{C^0}\\
 &\leq C_{1}(n_\epsilon+1)^{\dim \mathcal{E}_i}\epsilon^\alpha\norm{R_i}_{C^1}\norm{\omega}_{C^0}\max\{\norm{\mathcal{U}}_{C^\alpha}, \,\norm{\mathcal{V}}_{C^\alpha}\}\\
 &\overset{\text{(2)}}{\leq} C_{2}\epsilon^{\frac{\alpha}{2}}\norm{R_i}_{C^1}\norm{\omega}_{C^0}\max\{\norm{\mathcal{U}}_{C^\alpha}, \,\norm{\mathcal{V}}_{C^\alpha}\}.
\end{align*}
Here in $(1)$ we use \eqref{for:34} and \eqref{for:131}; in $(2)$ we use \eqref{for:132}.

Next, it follows from Lemma \ref{le:6} that
\begin{align*}
 &\norm{\mathcal{M}_2}+\norm{\mathcal{M}_3}\\
 &\leq \sum_{m=n_\epsilon+1}^{\infty}C_{\phi_1}\epsilon^{-1}\rho_i^{-\frac{\alpha(m+1)}{2}}(|m|+1)^{\dim \mathcal{E}_i-1}\norm{R_i}_{C^1}\norm{\omega}_{C^1}\max\{\norm{\mathcal{U}}_{C^\alpha}, \,\norm{\mathcal{V}}_{C^\alpha}\}\\
 &\leq C_{\phi_1,1}\epsilon^{-1}\rho_i^{-\frac{\alpha(n_\epsilon+2)}{4}}\norm{R_i}_{C^1}\norm{\omega}_{C^1}\max\{\norm{\mathcal{U}}_{C^\alpha}, \,\norm{\mathcal{V}}_{C^\alpha}\}\\
 &\overset{\text{(1)}}{\leq} C_{\phi_1,2}\epsilon^{\frac{\alpha}{2}}\norm{R_i}_{C^1}\norm{\omega}_{C^1}\max\{\norm{\mathcal{U}}_{C^\alpha}, \,\norm{\mathcal{V}}_{C^\alpha}\}.
\end{align*}
Here in $(1)$ we use \eqref{for:135}.

Combining the estimates for \(\mathcal M_1,\mathcal M_2,\mathcal M_3\), we get
\eqref{for:137}. Hence
\[
        \lim_{\epsilon\to0}
        \bigl(\mathcal N_{\mathcal V,\epsilon}
        -
        \mathcal N_{\mathcal U,\epsilon}\bigr)=0,
\]
which proves \eqref{for:17}.

\subsubsection{Step 3: Computation of  $\lim_{\epsilon\to0 }\mathcal{N}_{\mathcal{U},\epsilon}$} We first list two facts that will be used in the proof.
\begin{enumerate}
  \item\label{for:71} From \eqref{for:5} of Lemma \ref{for:8}, there is a sequence $k_n\to +\infty$ such that
\begin{align*}
 \rho_i^{-k_m}A_i^{k_m}\to I_{id}|_{E_i}\quad\text{and}\quad \rho_i^{k_m}A_i^{-k_m}\to I_{id}|_{E_i}
\end{align*}
as $k_m\to +\infty$.

\smallskip
  \item\label{for:133} It follows from Theorem \ref{th:9} that
there exist an $A$-invariant subspace $V_{i,1}\subset E_i$ and a bi-$\alpha$-\text{H\"older} subbundle   map
\[
\mathcal C_{i,1}(x):\mathcal F_x\to V_{i,1}
\]
(see \eqref{for:33} of Section~\ref{sec:12}) such that each $\mathcal C_{i,1}(x)$ is a linear isomorphism and
\begin{align*}
\mathcal C_{i,1}(fx)\circ Df_x|_{\mathcal F_x}
=
A_i|_{V_{i,1}}\circ \mathcal C_{i,1}(x),
\qquad x\in\mathbb T^N.
\end{align*}
\end{enumerate}
Define
\begin{align}\label{for:96}
 P(x):=
p_i\circ\bigl(\mathcal C_{i,1}(x)\bigr)^{-1}:V_{i,1}\to E_i\quad\text{and}\quad B_i:=
\int_{\TT^N}P(H^{-1}z)\,d\mathfrak m(z).
\end{align}
Thus $B_i:V_{i,1}\to E_i$ is a fixed linear map. Define
\begin{align}\label{for:145}
\mathcal T_i(x):
=p_i|_{\mathcal F_x}-B_i\circ \mathcal C_{i,1}(x)\quad\text{and}\quad \mathcal D_i(x):
=B_i\circ \mathcal C_{i,1}(x).
\end{align}
Since \(B_i\) is a fixed linear map and \(\mathcal C_{i,1}\) is
\(\alpha\)-H\"older, the maps \(\mathcal T_i\) and \(\mathcal D_i\) are both
\(\alpha\)-H\"older bundle maps on \(\mathcal F\).

In this step, we prove that
\begin{align}\label{for:146}
 \lim_{\epsilon\to0}\mathcal N_{\mathcal U,\epsilon}
=
\int_{-\delta}^{\delta}
\mathcal T_i(\gamma(t))(\gamma'(t))\omega(t)\,dt.
\end{align}
Let
\begin{align*}
J(x)=-p_i|_{\mathcal{F}_x}.
\end{align*}
Then $J$ satisfies the twisted coboundary equation
\begin{align}\label{for:69}
 A_i\circ J(x)-J(fx)\circ Df|_{\mathcal{F}_x}=DR_i|_{\mathcal{F}_x}, \qquad x\in \TT^N.
\end{align}
Iterating \eqref{for:69},  for any $m\geq1$ we have
\begin{align}\label{for:75}
 J(x)&+A_i^{-m}p_i\circ Df^{m}|_{\mathcal{F}_x}=\sum_{n=0}^{m-1}A_i^{-(n+1)}DR_i(f^{n}x)\circ Df^n(x)|_{\mathcal{F}_x}\notag\\
 &=\sum_{n=0}^{m-1}A_i^{-(n+1)}D(R_i\circ f^{n})|_{\mathcal{F}_x}
\end{align}
for any $x\in\TT^N$. Let
\begin{align*}
 \mathcal{S}_l(t,w,y)=\sum_{m=0}^lA_i^{-(m+1)}D(R_i\circ f^{m})_{\Phi_k(t,w,y)}(\mathcal U_{\Phi_k(t,w,y)}).
\end{align*}
Taking $m=k_n$ in \eqref{for:75}, we get
\begin{align*}
\mathcal S_{k_m-1}(t,w,y)
&=
-p_i\bigl(\mathcal U_{\Phi_k(t,w,y)}\bigr)+
A_i^{-k_m}p_i\circ Df_{\Phi_k(t,w,y)}^{k_m}
\bigl(\mathcal U_{\Phi_k(t,w,y)}\bigr).
\end{align*}
By the absolute convergence given by Lemma~\ref{le:6}, we may pass to the
subsequence of partial sums $k_m-1$. Hence
\begin{align}\label{for:141}
 \mathcal{N}_{\mathcal{U},\epsilon}&=-c_\epsilon \lim_{m\to +\infty}\int_{O_\epsilon}\mathcal{S}_{k_m-1}(t,w,y)\omega(t)\phi_1^\epsilon(w)\phi_2^\epsilon(y)d\widehat\mu\notag\\
 &=c_\epsilon \int_{O_\epsilon}p_i(\mathcal U_{\Phi_k(t,w,y)})\omega(t)\phi_1^\epsilon(w)\phi_2^\epsilon(y)d\widehat\mu-\mathcal{Z}_\epsilon
\end{align}
where
\begin{align*}
 \mathcal{Z}_\epsilon=c_\epsilon \lim_{m\to +\infty}\int_{O_\epsilon}\Big( A_i^{-k_m}p_i\circ Df^{k_m}|_{\mathcal{F}_{\Phi_k(t,w,y)}}(\mathcal U_{\Phi_k(t,w,y)}) \Big)\omega(t)\phi_1^\epsilon(w)\phi_2^\epsilon(y)d\widehat\mu.
\end{align*}
Moreover, by the approximate-identity argument, we have
\begin{align}\label{for:143}
&\lim_{\epsilon\to0 }c_\epsilon \int_{O_\epsilon}p_i(\mathcal U_{\Phi_k(t,w,y)})\omega(t)\phi_1^\epsilon(w)\phi_2^\epsilon(y)d\widehat\mu=\int_{-\delta}^\delta p_i(\mathcal U_{\Phi_k(t,0,0)})\omega(t)dt\notag\\
&=\int_{-\delta}^\delta p_i(\gamma'(t))\omega(t)dt.
\end{align}
It remains to compute \(\lim_{\epsilon\to0}\mathcal Z_\epsilon\).

Then, by \eqref{for:133}, for \(\tau=(t,w,y)\), we have
\begin{align*}
  &A_i^{-k_m}p_i\circ Df^{k_m}|_{\mathcal{F}_{\Phi_k(\tau)}}(\mathcal U_{\Phi_k(\tau)})\\
  &\overset{\text{(1)}}{=}A_i^{-k_m}p_i\circ \big(\mathcal C_{i,1}(f^{k_m}\Phi_k(\tau))\big)^{-1}\circ A_i^{k_m}\circ \mathcal C_{i,1}(\Phi_k(\tau))(\mathcal U_{\Phi_k(\tau)})\\
 &=\mathcal M_{1,m}(\tau)+\mathcal M_{2,m}(\tau)+\mathcal M_{3,m}(\tau),
\end{align*}
where
\begin{align*}
\mathcal M_{1,m}(\tau)
&:=
\bigl(\rho_i^{k_m}A_i^{-k_m}-I_{id}|_{E_i}\bigr)
\circ P(f^{k_m}\Phi_k(\tau))
\circ \bigl(\rho_i^{-k_m}A_i^{k_m}|_{V_{i,1}}\bigr)  \\
&\qquad\qquad\circ
\mathcal C_{i,1}(\Phi_k(\tau))
\bigl(\mathcal U_{\Phi_k(\tau)}\bigr),\\
\mathcal M_{2,m}(\tau)
&:=
P(f^{k_m}\Phi_k(\tau))
\circ
\bigl(\rho_i^{-k_m}A_i^{k_m}|_{V_{i,1}}-I_{id}|_{V_{i,1}}\bigr)\circ
\mathcal C_{i,1}(\Phi_k(\tau))
\bigl(\mathcal U_{\Phi_k(\tau)}\bigr),
\end{align*}
and
\[
\mathcal M_{3,m}(\tau)
:=
P(f^{k_m}\Phi_k(\tau))
\circ
\mathcal C_{i,1}(\Phi_k(\tau))
\bigl(\mathcal U_{\Phi_k(\tau)}\bigr).
\]
For \(\mathcal M_{1,m}\) and \(\mathcal M_{2,m}\), we have
\begin{align*}
\max\{\|\mathcal M_{1,m}\|_{C^0},\|\mathcal M_{2,m}\|_{C^0}\}
&\leq
C
\max\left\{
\|\rho_i^{k_m}A_i^{-k_m}-I_{id}|_{E_i}\|,
\|\rho_i^{-k_m}A_i^{k_m}-I_{id}|_{E_i}\|
\right\} \\
&\qquad\cdot
\|P\|_{C^0}
\|\mathcal C_{i,1}\|_{C^0}
\|\mathcal U\|_{C^0}.
\end{align*}
This, together with the convergence in \eqref{for:71}, implies that   for every fixed
$\epsilon>0$,
\begin{align}\label{for:D12-zero}
\lim_{m\to\infty}
\int_{O_\epsilon}
\mathcal M_{j,m}(t,w,y)
\omega(t)\phi_1^\epsilon(w)\phi_2^\epsilon(y)
\,d\widehat\mu
=
0,
\qquad j=1,2.
\end{align}
We now compute the contribution of \(\mathcal M_{3,n}\). This is where the
mixing property of \(A\) enters the argument.  Let
\[
\Psi_\epsilon(t,w,y)
:=
\frac{\omega(t)\phi_1^\epsilon(w)\phi_2^\epsilon(y)}
{J_k(t,w,y)},
\]
where \(J_k(t,w,y)\) is the density in the local disintegration formula (see \eqref{for:138}). We extend
$\Psi_\epsilon\circ\Phi_k^{-1}$ by zero outside $U_k$. Using
\eqref{for:138}, we get
\begin{align}\label{for:100}
&\lim_{m\to\infty}
\int_{O_\epsilon}
\mathcal M_{3,m}(t,w,y)
\omega(t)\phi_1^\epsilon(w)\phi_2^\epsilon(y)
\,d\widehat\mu \notag\\
&=
\lim_{m\to\infty}
\int_{\TT^N}
P(f^{k_m}z)
\circ\mathcal C_{i,1}(z)(\mathcal U_z)
\,
\Psi_\epsilon(\Phi_k^{-1}z)
\,d\mathfrak m(z)\notag\\
&\overset{\text{(1)}}{=}\lim_{m\to\infty}
\int_{\TT^N} P(H^{-1}A^{k_m}q)
\circ
\mathcal C_{i,1}(H^{-1}q)
\bigl(\mathcal U_{H^{-1}q}\bigr)\notag\\
&\qquad\qquad\cdot
\Psi_\epsilon\bigl(\Phi_k^{-1}(H^{-1}q)\bigr)
\frac{1}{\kappa(H^{-1}q)}
\,d\mathfrak m(q).
\end{align}
Here in $(1)$ we used:
\begin{enumerate}
\item[(a)] the conjugacy relation $H\circ f=A\circ H$, so that
\[
        f^{k_m}z=H^{-1}A^{k_m}H(z);
\]
\item[(b)] the change of variables $q=H(z)$;
\item[(c)] the identity
\[
        H_*\mathfrak m=(\kappa\circ H^{-1})^{-1}\mathfrak m,
\]
which follows from $(H^{-1})_*\mathfrak m=\kappa\,\mathfrak m$.
\end{enumerate}
Set
\[
G(q):=P(H^{-1}q)
\]
and
\[
F_\epsilon(q)
:=
\mathcal C_{i,1}(H^{-1}q)
\bigl(\mathcal U_{H^{-1}q}\bigr)
\Psi_\epsilon\bigl(\Phi_k^{-1}(H^{-1}q)\bigr)
\frac{1}{\kappa(H^{-1}q)}.
\]
For each fixed \(\epsilon>0\), the function \(F_\epsilon\) belongs to
\(L^\infty(\TT^N,\,V_{i,1})\). Moreover, \(G\) is H\"older, since \(P\) and
\(H^{-1}\) are H\"older.

Then the preceding expression is
\[
\lim_{m\to\infty}
\int_{\TT^N}
G(A^{k_m}q)F_\epsilon(q)\,d\mathfrak m(q).
\]
By definition
\[
B_i:=\int_{\TT^N}G(q)\,d\mathfrak m(q)
=
\int_{\TT^N}P(H^{-1}q)\,d\mathfrak m(q).
\]
Then \(G-B_i\) has zero average. By the mixing property of \(A\), applied componentwise, we have
\[
\lim_{m\to\infty}
\int_{\TT^N}
G(A^{k_m}q)F_\epsilon(q)\,d\mathfrak m(q)
=
B_i
\Big(\int_{\TT^N}F_\epsilon(q)\,d\mathfrak m(q)\Big).
\]
Finally, using again \((H^{-1})_*\mathfrak m=\kappa\,\mathfrak m\), we get
\begin{align*}
\int_{\TT^N}F_\epsilon(q)\,d\mathfrak m(q)
&=
\int_{\TT^N}
\mathcal C_{i,1}(z)(\mathcal U_z)
\Psi_\epsilon(\Phi_k^{-1}z)
\,d\mathfrak m(z) \\
&=
\int_{O_\epsilon}
\mathcal C_{i,1}(\Phi_k(t,w,y))
\bigl(\mathcal U_{\Phi_k(t,w,y)}\bigr)
\omega(t)\phi_1^\epsilon(w)\phi_2^\epsilon(y)
\,d\widehat\mu.
\end{align*}
Therefore,
\begin{align}\label{for:Zeps-limit}
\mathcal Z_\epsilon
=
c_\epsilon B_i
\int_{O_\epsilon}
&\mathcal C_{i,1}(\Phi_k(t,w,y))
\bigl(\mathcal U_{\Phi_k(t,w,y)}\bigr)
\omega(t)\phi_1^\epsilon(w)\phi_2^\epsilon(y)
\,d\widehat\mu .
\end{align}
Passing to the limit \(\epsilon\to0\), we obtain
\begin{align}\label{for:144}
\lim_{\epsilon\to0}\mathcal Z_\epsilon
&=
\int_{-\delta}^{\delta}
B_i\circ \mathcal C_{i,1}(\Phi_k(t,0,0))
\bigl(\mathcal U_{\Phi_k(t,0,0)}\bigr)
\omega(t)\,dt \notag\\
&=
\int_{-\delta}^{\delta}
B_i\circ \mathcal C_{i,1}(\gamma(t))
\bigl(\gamma'(t)\bigr)
\omega(t)\,dt.
\end{align}
Combining \eqref{for:141}, \eqref{for:143}, and \eqref{for:144}, we have
\begin{align*}
\lim_{\epsilon\to0}\mathcal N_{\mathcal U,\epsilon}
&=
\int_{-\delta}^{\delta}
p_i(\gamma'(t))\omega(t)\,dt  -
\int_{-\delta}^{\delta}
B_i\circ \mathcal C_{i,1}(\gamma(t))
\bigl(\gamma'(t)\bigr)
\omega(t)\,dt\\
&=
\int_{-\delta}^{\delta}
\mathcal T_i(\gamma(t))
\bigl(\gamma'(t)\bigr)
\omega(t)\,dt
\end{align*}
 This proves \eqref{for:146}.
\subsubsection{Step 4: Conclusion} It follows from \eqref{for:111}, \eqref{for:150} and \eqref{for:17} that
we obtain
\begin{align*}
\int_{-\delta}^{\delta}
H_i(\gamma(t))\omega'(t)\,dt
&=
\lim_{\epsilon\to 0}\mathcal{N}_{\mathcal V,\epsilon}
-
\int_{-\delta}^{\delta}
p_i(\gamma'(t))\omega(t)\,dt \\
&=\lim_{\epsilon\to 0}\mathcal{N}_{\mathcal U,\epsilon}
-
\int_{-\delta}^{\delta}
p_i(\gamma'(t))\omega(t)\,dt \\
&=
\int_{-\delta}^{\delta}
\mathcal T_i(\gamma(t))(\gamma'(t))\omega(t)\,dt
-
\int_{-\delta}^{\delta}
p_i(\gamma'(t))\omega(t)\,dt \\
&=
-\int_{-\delta}^{\delta}
B_i\circ\mathcal C_{i,1}(\gamma(t))
\bigl(\gamma'(t)\bigr)\omega(t)\,dt,\\
&=-\int_{-\delta}^{\delta}
\mathcal D_i(\gamma(t))(\gamma'(t))\omega(t)\,dt.
\end{align*}
By Remark~\ref{re:4}, the left-hand side agrees with
\[
        \int_{-\delta}^{\delta}
        H(\gamma(t))\omega'(t)\,dt .
\]
Therefore Lemma~\ref{th:7} follows.

\section{Partial $C^{1+\text{H\"older}}$ regularity of $H$ along $\mathcal{W}^f_i$}\label{sec:35}
We recall the following notation:
\begin{enumerate}

  \item The foliations $\mathcal W_i^f$ and $\mathcal W_i^A$ are defined in
  \eqref{for:23} of Section~\ref{sec:12}.
  \item The subbundle $\mathcal F_{i,1}$ is defined in \eqref{for:33} of
  Section~\ref{sec:12}.

  \item The maps $A_i$ and $p_i$ are defined in Section~\ref{sec:14}.
  \item \(\mathcal C_{i,1}\) is the bundle map given by
  Theorem~\ref{th:9}.
  \item \(B_i:V_{i,1}\to E_i\) is the linear map
  defined in Proposition~\ref{po:2}.

  \item $H_i$ is defined in \eqref{for:7} of Section \ref{sec:39}.
\end{enumerate}

\subsection{Main result}

\begin{theorem}\label{th:1} Suppose $H(\mathcal{W}^f_i) = \mathcal{W}^A_i$ for some $i_0\leq i\leq \ell$.   Then:
\begin{enumerate}
  \item \(\mathcal F_{i,1}\) is uniquely integrable to a foliation with
  uniformly \(C^{1+\alpha}\) leaves, denoted by
  \(\mathcal W_{\mathcal F_{i,1}}\).
  \item \(H\) is a \(C^{1+\alpha}\) diffeomorphism along
  \(\mathcal W_{\mathcal F_{i,1}}\). Moreover, for every
  \(x\in\TT^N\) and \(u\in \mathcal F_{i,1}(x)\),
  \[
  D(H_i)_x(u)=B_i\circ \mathcal C_{i,1}(x)(u).
  \]
   In addition, $B_i$ is injective.

\smallskip
  \item The subspace $W:=B_i(V_{i,1})$
is $A_i$-invariant, and \(H(\mathcal W_{\mathcal F_{i,1}})=W^{L}\), where $W^{L}$ denotes the linear foliation tangent to
  $W$.

\end{enumerate}
\end{theorem}
\begin{remark}
Since \(H(\mathcal W_i^f)=\mathcal W_i^A\), the components \(p_j\circ H\),
\(j\neq i\), are constant along the leaves of \(\mathcal W_i^f\). In
particular, they are constant along curves tangent to \(\mathcal F_{i,1}\).
Therefore, for every \(u\in\mathcal F_{i,1}(x)\),
\[
        p_j(DH_x(u))=0 \quad (j\neq i),
        \qquad
        p_i(DH_x(u))=D(H_i)_x(u).
\]
Equivalently, after identifying \(E_i\subset\mathbb R^N\), we have
\[
        DH_x(u)=D(H_i)_x(u),
        \qquad u\in\mathcal F_{i,1}(x).
\]
\end{remark}

\subsection{Role of Theorem~\ref{th:1}}

The partial differentiability obtained in Theorem~\ref{th:1}, together with
the \(A_i\)-invariant subspace \(W=B_i(V_{i,1})\), plays a crucial role in the
proof of Theorem~\ref{th:3}. In particular, the Diophantine property of \(W\)
is used to control the negative-time distributional series
\[
        h_i^-:=
        -\sum_{j=-1}^{-\infty}A_i^{-(j+1)}R_i\circ f^j .
\]
More precisely, the partial \(C^{1+\alpha}\) regularity of \(H^{-1}\) along the
linear foliation tangent to \(W\), together with the Diophantine property of
\(W\), allows us to prove that \(h_i^-\) is a well-defined distribution and that
\[
        h_i^-=h_i
        \qquad\text{as distributions}.
\]
This identity is the key step in proving the equality of the positive- and
negative-time distributions in Theorem~\ref{th:3}.

\subsection{Proof of Theorem \ref{th:1}}

Throughout the proof, we write $\mathcal F:=\mathcal F_{i,1}$ for simplicity. By Proposition \ref{po:2}, $H$ is curve differentiable along $\mathcal{F}$,
and
\[
dH_x^{\mathfrak c}=\mathcal D_i(x)=B_i\circ \mathcal C_{i,1}(x),
\qquad x\in\TT^N.
\]
Applying the curve derivative to the conjugacy equation
\[
H\circ f=A\circ H
\]
along \(\mathcal F\), we have
\begin{align}\label{for:4}
\mathcal D_i(fx)\circ Df|_{\mathcal{F}_x}=A_i\circ \mathcal D_i(x),\qquad \forall\,x\in\TT^N.
\end{align}
Since
\[
        \ker \mathcal D_i(x)
        =
        \mathcal C_{i,1}(x)^{-1}(\ker B_i),
\]
it is enough to prove that $\ker B_i=\{0\}$. Suppose, by contradiction, that
there exists $0\neq v\in\ker B_i$. Define a local vector
field \(X\) tangent to \(\mathcal F\) by
\[
X_x:=\mathcal C_{i,1}(x)^{-1}(v).
\]
By
Peano's existence theorem, there exists a local integral curve $\gamma$ of
$X$. Moreover, since $X$ is H\"older and nonvanishing, $\gamma$ is a nonconstant
regular curve along $\mathcal F$.  Along this curve,
\begin{align*}
\frac{d}{dt}(H\circ\gamma)(t)
=
\mathcal D_i(\gamma(t))(\gamma'(t))
=
B_i\circ\mathcal C_{i,1}(\gamma(t))
\bigl(\mathcal C_{i,1}(\gamma(t))^{-1}(v)\bigr)
=
B_iv
=
0.
\end{align*}
Hence \(H\circ\gamma\) is constant, contradicting the fact that \(H\) is a
homeomorphism and $\gamma$ is nonconstant. Therefore \(\ker B_i=\{0\}\), and consequently
\(\mathcal D_i(x)\) is injective on \(\mathcal F_x\) for every \(x\).

Since \(\mathcal{D}_i(x)\) is continuous and \(\mathcal F\) is a continuous
bundle over the compact space \(\TT^N\), the injectivity is uniform. Then
\begin{equation}\label{for:59}
\|\mathcal{D}_i(x)(u)\|
\geq
C^{-1}\|u\|,
\qquad
x\in\TT^N,\quad u\in\mathcal F_x.
\end{equation}
Define
\[
W:=B_i(V_{i,1})\subset E_i.
\]
Then, for every \(x\in\TT^N\),
\[
\mathcal D_i(x)(\mathcal F_x)=W.
\]
Using \eqref{for:4}, we see that
\begin{align*}
 A_i W
        =
        A_i\mathcal D_i(x)(\mathcal F_x)
        =
        \mathcal D_i(fx)(Df_x\mathcal F_x)
        =
        \mathcal D_i(fx)(\mathcal F_{fx})
        =
        W.
\end{align*}
Hence $W$ is $A_i$-invariant.

Let $W^L$ be the linear foliation tangent to $W$. We define
\[
        \mathcal W_{\mathcal F}(x)
        :=
        H^{-1}\bigl(W^L(H(x))\bigr).
\]
Since $H$ is
a homeomorphism, the sets $\mathcal W_{\mathcal F}(x)$ form a topological
foliation. We now show that this foliation is tangent to $\mathcal F$ and has
uniformly $C^{1+\text{H\"older}}$ leaves.

For each \(x\in\TT^N\), the map $\mathcal{D}_i:\mathcal F_x\to W$ is a linear isomorphism. By \eqref{for:59}, its inverse is uniformly bounded.
Since \(\mathcal{D}_i(x)\) is \(\alpha\)-H\"older, the family of
inverses
\[
\bigl(\mathcal{D}_i(x)\bigr)^{-1}:W\to\mathcal F_x
\]
is also \(\alpha\)-\text{H\"older} in \(x\). Fix \(v\in W\). Define a vector field \(X_v\) on \(\TT^N\) by
\[
X_v(x):=
\bigl(\mathcal{D}_i(x)\bigr)^{-1}(v).
\]
Then \(X_v\) is \(\alpha\)-H\"older and tangent to \(\mathcal F\). Let
\(\gamma\) be a local integral curve of \(X_v\), with \(\gamma(0)=x\). Then
\[
\frac{d}{dt}(H\circ\gamma)(t)
=
\mathcal{D}_i(\gamma(t))(\gamma'(t))
=
\mathcal{D}_i(\gamma(t))(X_v(\gamma(t)))
=
v.
\]
Therefore, for \(t\) sufficiently small,
\[
H(\gamma(t))=H(x)+tv.
\]
Equivalently,
\[
H^{-1}(H(x)+tv)=\gamma(t).
\]
Thus \(H^{-1}\) is differentiable along the affine \(W\)-leaves, and
\[
D(H^{-1})_{H(x)}|_W
=
\bigl(\mathcal{D}_i(x)\bigr)^{-1}.
\]
Since the right-hand side is \(\alpha\)-H\"older in \(x\), it follows
that \(H^{-1}\) is \(C^{1+\alpha}\) along affine \(W\)-leaves.
Consequently, the leaves of
\[
\mathcal W_{\mathcal F}(x)=H^{-1}(H(x)+W)
\]
are uniformly \(C^{1+\alpha}\) and tangent to \(\mathcal F\). Since \(H\)
maps each leaf of \(\mathcal W_{\mathcal F}\) diffeomorphically onto an
affine \(W\)-leaf, \(H\) is \(C^{1+\alpha}\) along
\(\mathcal W_{\mathcal F}\) and
\begin{align}\label{for:10}
 DH_x|_{\mathcal F_x}=\mathcal D_i(x).
\end{align}
It remains to prove unique integrability. Let \(\theta\) be any \(C^1\) curve tangent to \(\mathcal F\). By \eqref{for:10},  \(H\circ\theta\) is differentiable and \[ \frac{d}{dt}(H\circ\theta)(t)  = \mathcal D_i(\theta(t))(\theta'(t)) \in W . \] Hence \(H\circ\theta\) is contained in the affine subspace \(H(\theta(0))+W\). Therefore \[ \theta(t)\in H^{-1}(H(\theta(0))+W) = \mathcal W_{\mathcal F}(\theta(0)) \] for all \(t\). This proves unique integrability of \(\mathcal F\).

Taking $\mathcal W_{\mathcal F_{i,1}}:=\mathcal W_{\mathcal F}$,
we obtain the desired foliation. Moreover, \(H\) maps each leaf of
\(\mathcal W_{\mathcal F_{i,1}}\) onto an affine leaf parallel to the fixed
\(A_i\)-invariant subspace \(W\), and \(H\) is
\(C^{1+\alpha}\) along \(\mathcal W_{\mathcal F_{i,1}}\). This
completes the proof.

\section{Invariant distributions along $\mathcal{E}_i$}\label{sec:37}
We recall the following notation:
\begin{enumerate}
\item $H$ is a bi-$\eta$-H\"older conjugacy; see Section \ref{sec:5}.
  \item $\mathfrak m$ and $\mu$ are defined in Section~\ref{sec:12}.
  \item The numbers $\rho_i$, the Lyapunov blocks $\mathcal E_i$, and the
  linear subspaces $E_i$ are defined in Section~\ref{sec:12}.
   \item The maps $A_i$ and $p_i$ are defined in Section~\ref{sec:14}.
  \item The foliations $\mathcal W_i^f$ and $\mathcal W_i^A$ are defined in
  \eqref{for:23} of Section~\ref{sec:12}.
  \item The subbundle $\mathcal F_{i,j}$ is defined in \eqref{for:33} of
  Section~\ref{sec:12}.
\item $H_i$ is defined in \eqref{for:7} of Section \ref{sec:39}.
\end{enumerate}
\begin{theorem}\label{th:3}
Suppose that \(H(\mathcal W_i^f)=\mathcal W_i^A\) for some
\(i_0\leq i\leq \ell\). Let
\(0<\beta<1\). If \(\mathcal V\) is an \(\alpha\)-H\"older vector
field on \(\TT^N\) taking values in \(\mathcal E_i\), then the following
formulas define \(E_i\)-valued bounded linear functionals on
\(C^{\beta}(\TT^N)\):
\begin{align*}
\mathfrak D_{\mathcal V}^+(\omega)
&:=
\sum_{m=0}^{\infty}
\int_{\mathbb T^N}
A_i^{-(m+1)}
D(R_i\circ f^m)_{H^{-1}z}
\bigl(\mathcal V_{H^{-1}z}\bigr)
\,\omega(z)\,d\mathfrak m(z), \\
\mathfrak D_{\mathcal V}^-(\omega)
&:=
-
\sum_{m=-1}^{-\infty}
\int_{\mathbb T^N}
A_i^{-(m+1)}
D(R_i\circ f^m)_{H^{-1}z}
\bigl(\mathcal V_{H^{-1}z}\bigr)
\,\omega(z)\,d\mathfrak m(z), \\
\widetilde{\mathfrak D}_{\mathcal V}^+(\omega)
&:=
\sum_{m=0}^{\infty}
\int_{\mathbb T^N}
A_i^{-(m+1)}
D(R_i\circ f^m)_z
\bigl(\mathcal V_z\bigr)
\,\omega(z)\,d\mathfrak m(z),\\
\widetilde{\mathfrak D}_{\mathcal V}^-(\omega)
&:=
-
\sum_{m=-1}^{-\infty}
\int_{\mathbb T^N}
A_i^{-(m+1)}
D(R_i\circ f^m)_z
\bigl(\mathcal V_z\bigr)
\,\omega(z)\,d\mathfrak m(z).
\end{align*}
More precisely, for every \(\omega\in C^{\beta}(\TT^N)\),
\[
\max\{
\|\mathfrak D^+_{\mathcal V}(\omega)\|,
\|\mathfrak D^-_{\mathcal V}(\omega)\|, \|\widetilde{\mathfrak D}_{\mathcal V}^+(\omega)\|,\|\widetilde{\mathfrak D}_{\mathcal V}^-(\omega)\|
\}
\leq
C_{\beta}
\|\mathcal V\|_{C^{\alpha}}
\|\omega\|_{C^{\beta}}.
\]
Moreover,
\[
\mathfrak D^+_{\mathcal V}
=
\mathfrak D^-_{\mathcal V},\quad \widetilde{\mathfrak D}_{\mathcal V}^+=\widetilde{\mathfrak D}_{\mathcal V}^-
\]
as \(E_i\)-valued distributions on \(C^{\beta}(\TT^N)\).
\end{theorem}

We understand the integrals in the theorem componentwise. More precisely, if
\(g\) is vector-valued and \(\varphi\) is scalar-valued, then \(g\varphi\)
denotes their componentwise product, and the integral of \(g\varphi\) is taken
componentwise.

\begin{remark}
We will see that
\[
        D_{\mathcal V}h_i
        =
        \widetilde{\mathfrak D}_{\mathcal V}^+,
        \qquad
        D_{\mathcal V}h_i^-
        =
        \widetilde{\mathfrak D}_{\mathcal V}^-
\]
as distributions, where \(h_i^-\) is defined in \eqref{for:102}.

The distributions $\mathfrak D_{\mathcal V}^{\pm}$ are the conjugated versions
of $\widetilde{\mathfrak D}_{\mathcal V}^{\pm}$ under the Franks--Manning
conjugacy $H$. More precisely, let $\kappa$ denote the density determined by
\[
        H^{-1}_*\mathfrak m=\kappa\,\mathfrak m .
\]
Equivalently,
\[
        H_*\mathfrak m=(\kappa\circ H^{-1})^{-1}\mathfrak m .
\]
Then, for every test function $\omega\in C^\beta(\mathbb T^N)$,
\[
        \mathfrak D_{\mathcal V}^{\pm}(\omega)
        =
        \widetilde{\mathfrak D}_{\mathcal V}^{\pm}
        \bigl((\omega\circ H)\kappa\bigr).
\]
Indeed, this follows from the change of variables $z=H(x)$:
\begin{align*}
 &\int_{\mathbb T^N}
A_i^{-(m+1)}
D(R_i\circ f^m)_{H^{-1}z}
\bigl(\mathcal V_{H^{-1}z}\bigr)
\omega(z)\,d\mathfrak m(z) \\
&=
\int_{\mathbb T^N}
A_i^{-(m+1)}
D(R_i\circ f^m)_x
\bigl(\mathcal V_x\bigr)
(\omega\circ H)(x)\kappa(x)\,d\mathfrak m(x).
\end{align*}
Thus $\widetilde{\mathfrak D}_{\mathcal V}^{\pm}$ are the distributions in the
original $f$-coordinates, and they are the ones directly related to the
distributional derivative of $H_i$. The distributions
$\mathfrak D_{\mathcal V}^{\pm}$ are introduced because, after conjugating by
$H$, the stable and unstable foliations become the linear foliations of $A$.
This allows us to use translations on the linear torus and to apply the
distribution-to-H\"older criterion to the candidate derivative in the proof of
Theorem~\ref{th:4}.
\end{remark}

\subsection{Role of Theorem~\ref{th:3}}

Theorem~\ref{th:3} is the main distributional input in the proof of
Theorem~\ref{th:4}. It identifies the positive- and negative-time
distributional expressions for the candidate derivative of \(H_i\). Together
with the triangular reduction of \(Df|_{\mathcal E_i}\) (see Theorem \ref{th:9}), this equality allows
us to obtain stable-direction translation estimates from
\(\mathfrak D_{\mathcal V}^{+}\) and unstable-direction translation estimates
from \(\mathfrak D_{\mathcal V}^{-}\) on the linear side. Since
\[
        \mathfrak D_{\mathcal V}^{+}
        =
        \mathfrak D_{\mathcal V}^{-},
\]
these one-sided estimates combine to give the two-sided estimates required by
the distribution-to-H\"older criterion. This is what ultimately shows that the distribution
$D_{\mathcal V}h_i$ is in fact a H\"older map.

\subsection{Proof strategy for Theorem~\ref{th:3}}

The positive-time series are controlled directly by an averaging estimate along
the foliation $\mathcal W_i^f$. More precisely, Lemma~\ref{le:2} gives
summability for local test functions in foliation boxes, and
Corollary~\ref{cor:2} converts this local estimate into a global estimate.
This proves that the positive-time distributions $\mathfrak D_{\mathcal V}^+$ and $\widetilde{\mathfrak D}_{\mathcal V}^+$
are bounded $E_i$-valued linear maps on $C^\beta(\mathbb T^N)$.

The negative-time series are more delicate. We first define formally
\begin{align}\label{for:102}
  h_i^-
        =
        -
        \sum_{j=-1}^{-\infty}
        A_i^{-(j+1)}R_i\circ f^j .
\end{align}
The key point is to prove that this negative-time expression represents the
same distribution as $h_i$. To do this, we conjugate by $H$ and show that
\[
        h_i^-\circ H^{-1}
        =
        h_i\circ H^{-1}
\]
as distributions. The proof uses a two-sided telescoping identity coming from
the cohomological equation
\[
        A_i h_i-h_i\circ f=R_i.
\]
The positive endpoint of the telescoping sum converges to zero by expansion of
$A_i$. The negative endpoint is handled using the partial
$C^{1+\alpha}$ regularity of $H^{-1}$ along the linear foliation tangent to the
subspace $W\subset E_i$ obtained in Theorem~\ref{th:1}. The Diophantine
property of $W$ then allows us to integrate by parts against Fourier modes and
show that the negative endpoint also tends to zero. This gives
\[
        h_i^-\circ H^{-1}=h_i\circ H^{-1}.
\]
Changing variables back through $H$, we obtain $ h_i^-=h_i$ as distributions. This is one place where irreducibility of $A$ is used in an essential way: it guarantees the Diophantine property for the invariant subspaces that arise in the proof. A new feature of the proof is that the exponential mixing required to realize the formal negative-time
expression as a distribution is
not obtained from standard decay of correlations. Instead, partial
\(C^{1+\alpha}\) regularity along a Diophantine linear foliation is combined
with a Fourier integration-by-parts argument. The Diophantine property controls
the small divisors with only polynomial loss, while the contraction of the
linear dynamics along the same foliation provides the exponential gain needed
to control the twisted negative-time series; see Remark \ref{re:6}.

We then take distributional derivatives along an $\alpha$-H\"older vector field
$\mathcal V$. By the positive-time expansion of $h_i$, we have
\[
        D_{\mathcal V}h_i
        =
        \widetilde{\mathfrak D}_{\mathcal V}^+,
\]
while by the negative-time expansion of $h_i^-$,
\[
        D_{\mathcal V}h_i^-
        =
        \widetilde{\mathfrak D}_{\mathcal V}^-.
\]
Since $h_i^-=h_i$, these derivatives coincide. Hence $\widetilde{\mathfrak D}_{\mathcal V}^+
        =
        \widetilde{\mathfrak D}_{\mathcal V}^-$. This also gives boundedness of $\widetilde{\mathfrak D}_{\mathcal V}^-$. Finally, the conjugated distributions satisfy
\[
        \mathfrak D_{\mathcal V}^{\pm}(\omega)
        =
        \widetilde{\mathfrak D}_{\mathcal V}^{\pm}
        \bigl((\omega\circ H)\kappa\bigr),
\]
where $H^{-1}_*\mathfrak m=\kappa\,\mathfrak m$. Since $(\omega\circ H)\kappa\in C^{\beta\eta}(\mathbb T^N)$,
the equality for $\widetilde{\mathfrak D}_{\mathcal V}^{\pm}$, applied with
the exponent $\beta\eta$, gives $\mathfrak D_{\mathcal V}^+
        =
        \mathfrak D_{\mathcal V}^-$.
The same change-of-variables relation and the boundedness of the
$\widetilde{\mathfrak D}_{\mathcal V}^{\pm}$ give the estimates for
$\mathfrak D_{\mathcal V}^{\pm}$. This proves the boundedness and equality
claims in Theorem~\ref{th:3}.

\subsection{Notations and basic facts}

\subsubsection{Diophantine subspace} Suppose $L$ is a subspace of $\RR^N$. We say that $L$ has Diophantine property, i.e., for any $0\neq m\in\ZZ^N$ and any basis $\{v_1,\cdots, v_{\dim L}\}$ of $L$, we have
\begin{align*}
 \sum_{i=1}^{\dim L}|m\cdot v_i|\geq C_{v_1,\cdots,v_{\dim L}}\norm{m}^{-N}.
\end{align*}
We will use the following arithmetic fact later in the proof.
\begin{lemma}[Katznelson's Lemma]\label{Katz} Let $P$ be a $N\times N$ integer matrix.
Assume that $\R^N$ splits as $\R^N=V_1\oplus V_2$, where $V_1$ and $V_2$ are invariant under $P$,
and  $P|_{V_1}$ and $P|_{V_2}$ have no common eigenvalues. If $V_1\cap\Z^N=\{0\}$, then
there exists a constant $K$ such that
\begin{align*}
dist(n,V_1)\geq K\norm{n}^{-N} \quad\text{for all } \,0\ne n\in\Z^N,
\end{align*}
where $\norm{n}$ denotes
Euclidean norm and $dist$ is Euclidean distance.
\end{lemma}
\begin{lemma}\label{le:4} If \(L\) is an \(A\)-invariant subspace, then \(L\) has the Diophantine property.
\end{lemma}
\begin{proof} Since \(L\) is \(A\)-invariant, \(L^\perp\) is invariant under \(A^\tau\), the transpose of $A$. We claim that
\[
L^\perp\cap\ZZ^N=\{0\}.
\]
Indeed, if \(0\neq z\in L^\perp\cap\ZZ^N\), then the rational span of the
\(A^\tau\)-orbit of \(z\) gives a nonzero proper rational \(A^\tau\)-invariant
subspace. This contradicts the irreducibility of \(A^\tau\), which is equivalent
to the irreducibility of \(A\).

We have a generalize eigenspace decomposition of $A^\tau$:
\begin{align*}
 \RR^N=\bigoplus_{\lambda\in\Delta} V_\lambda,
\end{align*}
where
$\Delta$ denotes the set of eigenvalues of $A$. Since $A^\tau$ is irreducible, each $V_\lambda$ is minimal, i.e., $V_\lambda$ has no nontrivial $A^\tau$-invariant subspace. In fact,
each $V_\lambda$ is one or two dimensional and $A^\tau|_{V_\lambda}$ is conformal.  Let $\Delta_L$ denote the set of eigenvalues of $A^\tau|_{L^\bot}$. Let $L'$ be the subspace spanned by generalized eigenvectors corresponding to eigenvalues inside $\Delta\backslash \Delta_{L^\bot}$. It is clear that $L'$ is also an invariant subspace of $A^\tau$ and
\begin{align*}
\RR^N=L^\bot\bigoplus L'.
\end{align*}
Moreover, the restrictions of $A^\tau$ to $L^\perp$ and $L'$ have no common
eigenvalues.

Applying Lemma~\ref{Katz} to $P=A^\tau$, $V_1=L^\perp$, and
$V_2=L'$, we obtain
\begin{align}\label{for:57}
 d(m,L^\bot)\geq C_L\norm{m}^{-N},\qquad \forall\, \,0\neq m\in\ZZ^{N}.
\end{align}
Let $p_L$ denote the projection to $L$. Then for any $0\neq m\in\ZZ^{N}$ we have
\begin{align*}
 \sum_{i=1}^{\dim L}|m\cdot v_i|&\geq C_{v_1,\cdots,v_{\dim L}}\norm{p_{L}(m)}=C_{v_1,\cdots,v_{\dim L}}d(m,L^\bot)\\
 &\overset{\text{(1)}}{\geq} C_{v_1,\cdots,v_{\dim L},1}\norm{m}^{-N}.
\end{align*}
Here in $(1)$ we use \eqref{for:57}. Hence, we finish the proof.
\end{proof}

\subsubsection{Decay estimate}  We recall notations in Section \ref{sec:16}. In this subsection, we prove two decay estimates that will be used later.
The proof is similar to that of Lemma~\ref{le:6}, but we include the details
for completeness. As in Remark~\ref{re:5}, we use the product measure \(dx\,d\nu_k(y)\) in the foliation box, rather than the full foliation-box measure
 \begin{align*}
  J_k(x,y)\,dx\,d\nu_k(y).
 \end{align*}
 The Jacobian factor \(J_k(x,y)\) can be absorbed into the test function when one compares the product measure with \(\mathfrak m\).

\begin{lemma}\label{le:2} Fix a foliation chart $(\Gamma_k,\,O_k)$ for $\mathcal{W}^f_i$.  Let \(\mathcal Q\) be an \(\alpha\)-H\"older vector field on
\(\Gamma_k(O_k)\) taking values in \(\mathcal E_i\). Let \(\varphi\) be a function compactly
supported on \(O_k\), which is \(\beta\)-H\"older in the
leaf variable \(x\), uniformly in \(y\). Then, for every \(m\geq0\), we have
\begin{align*}
 &\Big\|\int_{O_k} A_i^{-(m+1)}D(R_i\circ f^{m})_{\Gamma_k(x,y)}(\mathcal Q_{\Gamma_k(x,y)})\varphi(x,y)dxd\nu_k(y)\Big\|\\
 &\leq C_{\beta}\rho_i^{-\frac{\varrho(m+1)}{2}}(|m|+1)^{\dim \mathcal{E}_i-1}\norm{R_i}_{C^1}\norm{\mathcal{Q}}_{C^\alpha}\norm{\varphi}_{C^\beta_x},
\end{align*}
where $\varrho=\min\{\alpha,\,\beta\}$.
\end{lemma}
\begin{proof}
It follows from \eqref{for:33} of Section~\ref{sec:12} and the fact that
\(A_i\) is diagonalizable with all eigenvalues of modulus \(\rho_i\) (see Section \ref{sec:14}) that, for
every \(m\geq0\),
\begin{align}\label{for:147}
 \norm{A_i^{-(m+1)}D(R_i\circ f^{m})|_{\mathcal{E}_i}}\leq C (|m|+1)^{\dim \mathcal{E}_i-1}\norm{R_i}_{C^1}.
\end{align}
Set
\begin{align*}
 \mathcal W^j_{\Gamma_k(x,y)}
:=
D\Gamma_k(x,y)(\partial_{x_j}),\qquad 1\leq j\leq \dim \mathcal{E}_i.
\end{align*}
Since \(\mathcal Q\) takes values in \(\mathcal E_i\), we can write
\begin{align*}
 \mathcal Q_{\Gamma_k(x,y)}=\sum_{j=1}^{\dim \mathcal{E}_i} a_j(x,y)\mathcal W^j_{\Gamma_k(x,y)}.
\end{align*}
The coefficients \(a_j\) are \(\alpha\)-\text{H\"older} in the leaf variables
\(x\), uniformly in \(y\), and satisfy
\[
\|a_j\|_{C^0}\leq C\|\mathcal Q\|_{C^0},
\qquad
\|a_j\|_{C^\alpha_x}\leq C\|\mathcal Q\|_{C^\alpha}.
\]
For each $j$, define
\[
        b_j(x,y):=a_j(x,y)\varphi(x,y).
\]
Then $b_j$ is $\varrho$-H\"older in the leaf variable $x$, uniformly in $y$,
and
\[
        \|b_j\|_{C^\varrho_x}
        \leq
        C_\beta
        \|\mathcal Q\|_{C^\alpha}
        \|\varphi\|_{C^\beta_x}.
\]
We smooth $b_j$ only in the leaf variables $x$, treating $y$ as a parameter.
Let $b_{j,\eta}$ be the corresponding $C^\infty$ approximation as in Section~\ref{sec:17}. Since
$\varphi$ is compactly supported in $O_k$, after extending $b_j$ by zero in
the leaf variables and taking $\eta$ sufficiently small, there are no boundary
terms in the integration by parts below. The standard smoothing estimates give
\begin{align}
        \|b_{j,\eta}-b_j\|_{C^0}
        &\leq
        C_\beta\eta^\varrho
        \|\mathcal Q\|_{C^\alpha}
        \|\varphi\|_{C^\beta_x},
        \label{for:155}\\
        \|b_{j,\eta}\|_{C^1_x}
        &\leq
        C\eta^{-1}
        \|\mathcal Q\|_{C^0}
        \|\varphi\|_{C^0}.
        \label{for:6}
\end{align}
We split
\begin{align*}
 &\int_{O_k} A_i^{-(m+1)}D(R_i\circ f^{m})_{\Gamma_k(x,y)}(\mathcal Q_{\Gamma_k(x,y)})\varphi(x,y)dxd\nu_k(y)\\
 &=\sum_{j=1}^{\dim \mathcal{E}_i}\int_{O_k} A_i^{-(m+1)}D(R_i\circ f^{m})_{\Gamma_k(x,y)}(\mathcal W^j_{\Gamma_k(x,y)})b_j(x,y)dxd\nu_k(y)\\
&=\sum_{j=1}^{\dim \mathcal{E}_i}\mathcal{X}_j+\sum_{j=1}^{\dim \mathcal{E}_i}\mathcal{Z}_j,
\end{align*}
where
\begin{align*}
 \mathcal{X}_j&=\int_{O_k} A_i^{-(m+1)}D(R_i\circ f^{m})_{\Gamma_k(x,y)}(\mathcal W^j_{\Gamma_k(x,y)})b_{j,\eta}(x,y)dxd\nu_k(y)\\
 \mathcal{Z}_j&=\int_{O_k} A_i^{-(m+1)}D(R_i\circ f^{m})_{\Gamma_k(x,y)}(\mathcal W^j_{\Gamma_k(x,y)})(b_j-b_{j,\eta})(x,y)dxd\nu_k(y).
\end{align*}
We first estimate \(\|\mathcal X_j\|\).   We note that
\[
D(R_i\circ f^m)_{\Gamma_k(x,y)}
\bigl(\mathcal W^j_{\Gamma_k(x,y)}\bigr)
=
\partial_{x_j}(R_i\circ f^m\circ\Gamma_k)(x,y).
\]
Then we have
\begin{align*}
 \mathcal{X}_j=-\int_{O_k} A_i^{-(m+1)}(R_i\circ f^{m}\circ\Gamma_k)(x,y)\big(\partial_{x_j}b_{j,\eta}(x,y)\big)dxd\nu_k(y).
\end{align*}
Hence,
\begin{align*}
 \norm{\mathcal{X}_j}&\leq C\norm{A_i^{-(m+1)}}\norm{R_i}_{C^0}\norm{b_{j,\eta}}_{C^1_{x}}
 \overset{\text{(1)}}{\leq} C\eta^{-1}\rho_i^{-(m+1)} \norm{R_i}_{C^0}\|\mathcal Q\|_{C^0}\norm{\varphi}_{C^0}.
\end{align*}
Here in $(1)$ we use \eqref{for:6}.

We next estimate each $\norm{\mathcal{Z}_j}$. It follows from \eqref{for:147} and \eqref{for:155} that
\begin{align*}
\norm{\mathcal{Z}_j} &\leq C\big\|A_i^{-(m+1)}D(R_i\circ f^{m})|_{\mathcal{E}_i}\big\|\cdot \norm{b_{j,\eta}-b_j}_{C^0}\notag\\
 &\leq C_\beta\eta^\varrho(|m|+1)^{\dim \mathcal{E}_i-1}\norm{R_i}_{C^1}\norm{\mathcal{Q}}_{C^\alpha}\norm{\varphi}_{C^\beta_x}.
\end{align*}
Combining the estimates for \(\mathcal X_j\) and \(\mathcal Z_j\), and summing
over \(1\leq j\leq \dim \mathcal{E}_i\), we obtain
\begin{align*}
 &\Big\|\int_{O_k} A_i^{-(m+1)}D(R_i\circ f^{m})_{\Gamma_k(x,y)}(\mathcal Q_{\Gamma_k(x,y)})\varphi(x,y)dxd\nu_k(y)\Big\|\\
 &\leq C_{\beta}\big(\eta^{-1}\rho_i^{-(m+1)}+ \eta^\varrho(|m|+1)^{\dim \mathcal{E}_i-1}\big)\norm{R_i}_{C^1}\norm{\mathcal{Q}}_{C^\alpha}\norm{\varphi}_{C^\beta_x}.
\end{align*}
Let $\eta=\rho_i^{-\frac{m+1}{2}}$. Then
\begin{align*}
 \rho_i^{-(m+1)}\eta^{-1}
=
\rho_i^{-\frac{m+1}{2}},\quad \eta^\varrho
=
\rho_i^{-\frac{\varrho(m+1)}{2}}, \quad\text{and}\quad \rho_i^{-\frac{m+1}{2}}
\leq
\rho_i^{-\frac{\varrho(m+1)}{2}}.
\end{align*}
Thus, we have
\begin{align*}
&\left\|
\int_{O_k}
A_i^{-(m+1)}
D(R_i\circ f^m)_{\Gamma_k(x,y)}
\bigl(\mathcal Q_{\Gamma_k(x,y)}\bigr)
\varphi(x,y)dxd\nu_k(y)
\right\| \\
&\qquad\leq
C_{\beta}
\rho_i^{-\frac{\varrho(m+1)}{2}}
(m+1)^{\dim\mathcal{E}_i-1}
\|R_i\|_{C^1}
\|\mathcal Q\|_{C^\alpha}\norm{\varphi}_{C^\beta_x}.
\end{align*}
This proves the lemma.
\end{proof}
The following is a direct consequence of Lemma \ref{le:2}.
\begin{corollary}\label{cor:2} Fix a foliation chart $(\Gamma_k,\,O_k)$ for $\mathcal{W}^f_i$.  Let \(\mathcal Q\) be an \(\alpha\)-H\"older vector field on
\(\Gamma_k(O_k)\) taking values in \(\mathcal E_i\). Let \(\varphi\) be a function compactly
supported on \(\Gamma_k(O_k)\), such that $\varphi\circ \Gamma_k$ is \(\beta\)-H\"older in the
leaf variable \(x\), uniformly in \(y\). Then, for every \(m\geq0\), we have
\begin{align*}
 &\Big\|\int_{\TT^N} A_i^{-(m+1)}D(R_i\circ f^{m})_{z}(\mathcal Q_{z})\varphi(z)d\mathfrak{m}\Big\|\\
 &\leq C_{\beta,k}\rho_i^{-\frac{\varrho(m+1)}{2}}(|m|+1)^{\dim \mathcal{E}_i-1}\norm{R_i}_{C^1}\norm{\mathcal{Q}}_{C^\alpha}\norm{\varphi\circ \Gamma_k}_{C^\beta_x},
\end{align*}
where $\varrho=\min\{\alpha,\,\beta\}$.
\end{corollary}
\begin{proof} By using the measure-change formula \eqref{for:167} in Section \ref{sec:16}, we have
\begin{align*}
 &\int_{\TT^N} A_i^{-(m+1)}D(R_i\circ f^{m})_{z}(\mathcal Q_{z})\varphi(z)d\mathfrak{m}\\
 &=\int_{O_k} A_i^{-(m+1)}D(R_i\circ f^{m})_{\Gamma_k(x,y)}(\mathcal Q_{\Gamma_k(x,y)})\varphi(\Gamma_k(x,y))J_k(x,y)dxd\nu_k(y).
\end{align*}
The function
\[
(\varphi\circ\Gamma_k)(x,y)J_k(x,y)
\]
is \(\varrho\)-H\"older in the leaf variable \(x\), uniformly in \(y\). Moreover,
since $J_k$ is uniformly $\alpha$-H\"older in the leaf variable,
\begin{align*}
 \norm{(\varphi\circ \Gamma_k)J_k}_{C^\varrho_x}\leq C\norm{\varphi\circ \Gamma_k}_{C^\beta_x}\norm{J_k}_{C^\alpha_x}.
\end{align*}
The desired estimate now follows from Lemma~\ref{le:2}, with the part $\norm{J_k(x,y)}_{C^\alpha_x}$
absorbed into \(C_{\beta,k}\).
\end{proof}

\subsubsection{A decay estimate from the Diophantine property} Recall the  $A_i$-invariant subspace $W$ defined in Theorem \ref{th:1}. Let \(C^{1+\alpha,W}(\TT^N)\) denote the space of continuous functions
\(\psi\) such that
\[
\begin{aligned}
 \|\psi\|_{C^{1+\alpha,W}}
 &:=
 \|\psi\|_{C^0}
 +
 \sup_{w\in W,\ \|w\|=1}\|\partial_w\psi\|_{C^0}  \\
 &\quad+
 \sup_{\substack{x\in\TT^N,\ 0\neq v\in W\\ w\in W,\ \|w\|=1}}
 \frac{|\partial_w\psi(x+v)-\partial_w\psi(x)|}{\|v\|^\alpha}
 <\infty .
\end{aligned}
\]
\begin{lemma}\label{le:8} Suppose $\xi:\TT^N\to\RR^N$ is $C^{1+\alpha}$ along $W$. Then for any $l>0$ and any $0\neq n\in\ZZ^N$, we have
\begin{align*}
 \Big\|\int_{\TT^N}
\xi(A^{-l}z)e_n(z)\,d\mathfrak m(z)\Big\|\leq
        C\norm{n}^{(1+\alpha)N}\rho_i^{-(\alpha+1) l}\|\xi\|_{C^{1+\alpha,W}}.
\end{align*}
\end{lemma}
\begin{proof}
Fix a basis $\{w_1,\cdots, w_{\dim W}\}$ of $W$.
It follows from Lemma \ref{le:4} that $W$ has Diophantine property. Hence, for each $0\neq n\in\mathbb Z^N$, there exists \(j_0=j_0(n)\) such that
\begin{align}\label{for:112}
 |n\cdot w_{j_0}|\geq C_{W}\norm{n}^{-N}.
\end{align}
Let
\[
I_{l,n}
:=
\int_{\TT^N}
\xi(A^{-l}z)e_n(z)\,d\mathfrak m(z).
\]
Since
\[
\partial_{w_{j_0}}e_n(z)
=
2\pi\mathbf i(n\cdot w_{j_0})e_n(z),
\]
we have
\begin{align*}
 I_{l,n}
&=
\frac{1}{2\pi\mathbf i(n\cdot w_{j_0})}
\int_{\TT^N}
\xi(A^{-l}z)\partial_{w_{j_0}}e_n(z)\,d\mathfrak m(z)\\
&=-\frac{1}{2\pi\mathbf i(n\cdot w_{j_0})}
\int_{\TT^N}
D\xi_{A^{-l}z}(A^{-l}w_{j_0})
e_n(z)\,d\mathfrak m(z).
\end{align*}
Set
\begin{align*}
 a_{l}:=\|A^{-l}w_{j_0}\|,
\qquad
\widetilde w_{l}:=
\frac{A^{-l}w_{j_0}}{\|A^{-l}w_{j_0}\|},\qquad g_{l}:=\partial_{\widetilde w_{l}}\xi.
\end{align*}
Then
\begin{align}\label{for:62}
I_{l,n}
=
-\frac{1}{2\pi\mathbf i(n\cdot w_{j_0})}
a_{l_2}
\int_{\TT^N}
g_{l_2}(A^{-l}z)e_n(z)\,d\mathfrak m(z).
\end{align}
Since $\xi$ is \(C^{1+\alpha}\) along \(W\),
\[
\sup_{l\geq1}\|g_{l}\|_{C^{\alpha,W}}<\infty.
\]
Let \(t=\frac{1}{2(n\cdot w_{j_0})}\). Then
\[
e^{2\pi\mathbf i\,t(n\cdot w_{j_0})}-1=-2.
\]
Then
\begin{align*}
&\int_{\TT^N}g_{l}(A^{-l}z)e_n(z)\,d\mathfrak m(z) \\
&\quad =
\frac{1}{e^{2\pi\mathbf i\,t(n\cdot w_{j_0})}-1}
\int_{\TT^N}
\Big[
g_{l}(A^{-l}(z-tw_{j_0}))
-
g_{l_2}(A^{-l}z)
\Big]
e_n(z)\,d\mathfrak m(z)\\
&=-\frac{1}{2}
\int_{\TT^N}
\Big[
g_{l}(A^{-l}(z-tw_{j_0}))
-
g_{l}(A^{-l}z)
\Big]
e_n(z)\,d\mathfrak m(z).
\end{align*}
Therefore,
\begin{align*}
 \left\|
\int_{\TT^N}g_{l_2}(A^{-l}z)e_n(z)\,d\mathfrak m(z)
\right\|
&\overset{\text{(1)}}{\leq} C\|g_{l}\|_{C^{\alpha,W}}
\|tA^{-l}w_{j_0}\|^\alpha\\
&\leq C_1|n\cdot w_{j_0}|^{-\alpha}\|\xi\|_{C^{1+\alpha,W}}
\|A^{-l}w_{j_0}\|^\alpha.
\end{align*}
Here in $(1)$ we note that $A^{-l}w_{j_0}\in W$ since $W\subset E_i$ is $A_i$-invariant.

This, together with \eqref{for:62} gives
\begin{align}\label{for:110}
 \|I_{l,n}\|
&\leq
C|n\cdot w_{j_0}|^{-(1+\alpha)}\|\xi\|_{C^{1+\alpha,W}}
\|A^{-l}w_{j_0}\|^{1+\alpha}\notag\\
&\overset{\text{(1)}}{\leq}C_1|n\cdot w_{j_0}|^{-(1+\alpha)}\rho_i^{-l(\alpha+1)}\|\xi\|_{C^{1+\alpha,W}}\notag\\
&\overset{\text{(2)}}{\leq}C_2\norm{n}^{(1+\alpha)N}\rho_i^{-l(\alpha+1)}\|\xi\|_{C^{1+\alpha,W}}.
\end{align}
Here in  $(1)$ we use \eqref{for:93} of Section \ref{sec:8} and the fact that $W\subset E_i$ is $A_i$-invariant; in $(2)$ we use \eqref{for:112}.
\end{proof}
\begin{remark}\label{re:6}
The Diophantine property of \(W\) is crucial in this argument. It allows us,
for each nonzero Fourier mode \(n\), to choose a direction \(w_{j_0}\in W\)
such that
\[
        |n\cdot w_{j_0}|
        \geq
        C\|n\|^{-N}.
\]
Thus the small divisor \(n\cdot w_{j_0}\) has only polynomial loss in the
Fourier frequency.

Equivalently, this property allows us to solve the elementary equation
\[
        e_n(z)
        =
        \partial_{w_{j_0}}
        \bigl(\varphi(z+tw_{j_0})-\varphi(z)\bigr)
\]
with controlled loss. Indeed, for \(t=\frac{1}{2(n\cdot w_{j_0})}\), we have
\[
        e^{2\pi\mathbf i\,t(n\cdot w_{j_0})}-1=-2,
\]
and the function
\[
        \varphi(z)
        =
        \frac{e_n(z)}{2\pi\mathbf i(n\cdot w_{j_0})}
        \frac{1}{e^{2\pi\mathbf i\,t(n\cdot w_{j_0})}-1}
        =
        -\frac{1}{2}
        \frac{e_n(z)}{2\pi\mathbf i(n\cdot w_{j_0})}
\]
satisfies the desired identity.

Applying this identity and integrating by parts gives
\[
\begin{aligned}
 \int_{\TT^N}\xi(A^{-l}z)e_n(z)\,d\mathfrak m(z)
 &=
 -
 \int_{\TT^N}
 \partial_{w_{j_0}}\bigl(\xi(A^{-l}z)\bigr)
 \bigl(\varphi(z+tw_{j_0})-\varphi(z)\bigr)
 \,d\mathfrak m(z).
\end{aligned}
\]
The derivative \(\partial_{w_{j_0}}\bigl(\xi(A^{-l}z)\bigr)\) produces the
factor \(\|A^{-l}w_{j_0}\|\), which is of order \(\rho_i^{-l}\), because
\(W\subset E_i\) is \(A_i\)-invariant. The finite-difference step then uses
the \(C^\alpha\) regularity of the leafwise derivative of \(\xi\) and produces
an additional factor of order \(\rho_i^{-\alpha l}\). Together these two
effects give the decay rate $ \rho_i^{-(1+\alpha)l}$.
After multiplication by the twist, whose norm grows at most like
\(\rho_i^l\), one still obtains the decay rate \(\rho_i^{-\alpha l}\). This
is the key point that allows us to control the formal negative-time expression
and to identify it with \(h_i\circ H^{-1}\) as a distribution.

We emphasize that this estimate is not obtained from a standard exponential
mixing estimate for the toral automorphism \(A\). Standard mixing estimates
usually require global H\"older regularity of the observables in order to
obtain exponential decay. Moreover, obtaining prescribed decay rates typically
requires controlling higher global regularity norms; see
\cite{FKS,Damjanovic4}. In the present argument, the available regularity is
only leafwise: \(h_i\circ H^{-1}\) is \(C^{1+\alpha}\) along the linear
foliation tangent to \(W\). Thus the required decay follows from partial
leafwise regularity together with the Diophantine property of \(W\), rather
than from a standard global mixing estimate.

\end{remark}
\subsection{Proof of Theorem \ref{th:3}} Integrating the cohomological equation \eqref{for:156} of Section \ref{sec:14} with respect to \(\mu\),
we get
\begin{align*}
 A_i\int_{\TT^N}h_i\,d\mu-\int_{\TT^N}h_i\circ f\,d\mu
=
\int_{\TT^N}R_i\,d\mu.
\end{align*}
Since \(\mu\) is \(f\)-invariant, this gives
\begin{align*}
 A_i\int_{\TT^N}h_id\mu-\int_{\TT^N}h_id\mu=\int_{\TT^N}R_id\mu.
\end{align*}
Therefore, replacing \(h_i\) and $R_i$ by
\begin{align*}
 h_i-\int_{\TT^N}h_id\mu\quad\text{and}\quad R_i-\int_{\TT^N}R_id\mu
\end{align*}
respectively, preserves the cohomological equation. This normalization does not affect the
argument below. Thus, from now on, we assume
\begin{align*}
 \int_{\TT^N}h_id\mu=\int_{\TT^N}R_id\mu=0.
\end{align*}

\subsubsection{Step 1:  Construction of the negative-time distribution} Define formally
\begin{align*}
 h_i^-\circ H^{-1}=-\sum_{j=-1}^{-\infty}A_i^{-(j+1)}R_i\circ f^j\circ H^{-1}.
\end{align*}
In this step, we show that this series converges as a distribution and that
\begin{align}\label{for:158}
 h_i^-\circ H^{-1}=h_i\circ H^{-1}\qquad\text{as distributions.}
\end{align}
First, we show that
\begin{align}\label{for:94}
  \sum_{j=-\infty}^{\infty}
        A_i^{-(j+1)}R_i\circ f^j\circ H^{-1}
        =
        0
\end{align}
as a distribution.

From equation \eqref{for:156} of Section \ref{sec:14}
for any $l_1\geq0$ and $l_2\geq1$ we have
\begin{align*}
 \sum_{j=-l_2}^{l_1} A_i^{-(j+1)}R_i\circ f^j=A_i^{l_2}h_i\circ f^{-l_2}-A_i^{-(l_1+1)}  h_i\circ f^{l_1+1}.
\end{align*}
This together with the conjugacy equation \eqref{for:140} of Section \ref{sec:5} give
\begin{align*}
 \sum_{j=-l_2}^{l_1} A_i^{-(j+1)}R_i\circ f^j\circ H^{-1}&=A_i^{l_2}(h_i\circ H^{-1})\circ A^{-l_2}-A_i^{-(l_1+1)}  (h_i\circ H^{-1})\circ A^{l_1+1}.
\end{align*}
Let
\[
e_n(z):=e^{2\pi\mathbf i\,n\cdot z},
\qquad n\in\ZZ^N.
\]
To prove \eqref{for:94}, it suffices to show that for any $n\in\ZZ^N$, the two endpoint terms
converge to zero on each Fourier mode. The estimates below are polynomial in
the Fourier frequency \(n\), and therefore this convergence extends from
Fourier modes to arbitrary \(C^\infty\) test functions.
\begin{align}
 &\lim_{l_1\to \infty}\int_{\TT^N} A_i^{-(l_1+1)}  (h_i\circ H^{-1})(A^{l_1+1}z)e_n(z) d\mathfrak{m}=0,\qquad \text{and}\label{for:60}\\
 &\lim_{l_2\to \infty}\int_{\TT^N} A_i^{l_2}(h_i\circ H^{-1})(A^{-l_2}z)e_n(z)d\mathfrak{m}=0, \label{for:61}
\end{align}
We recall \eqref{for:93} of Section \ref{sec:8}. Therefore \eqref{for:60} follows from
\begin{align*}
&\left\|
\int_{\TT^N}
A_i^{-(l_1+1)}
(h_i\circ H^{-1})(A^{l_1+1}z)e_n(z)\,d\mathfrak m(z)
\right\|  \\
&\qquad\leq
C\rho_i^{-(l_1+1)}
\|h_i\circ H^{-1}\|_{L^1}
\to 0.
\end{align*}
We now prove \eqref{for:61}. If \(n=0\), then by the normalization above, we have
\begin{align*}
 \int_{\mathbb T^N}
A_i^{l_2}(h_i\circ H^{-1})(A^{-l_2}z)\,d\mathfrak m(z)
&=
A_i^{l_2}
\int_{\mathbb T^N}h_i\circ H^{-1}\,d\mathfrak m =
A_i^{l_2}
\int_{\mathbb T^N}h_i\,d\mu
=
0.
\end{align*}
Now assume \(0\neq n\in\ZZ^N\).

By Theorem~\ref{th:1}, there exists an \(A_i\)-invariant subspace
\(W\subset E_i\) such that \(H\) maps the foliation
\(\mathcal W_{\mathcal F_{i,1}}\) onto the linear foliation \(W^L\), and
\(H\) is a \(C^{1+\alpha}\) diffeomorphism along these leaves. Hence \(H^{-1}\) is \(C^{1+\alpha}\) along
\(W\). Moreover, since
\[
p_i\circ H=p_i+h_i,
\]
we have
\[
h_i\circ H^{-1}(z)
=
p_i(z)-p_i\circ H^{-1}(z).
\]
This identity is understood after choosing compatible local lifts. Since we
use it only to obtain leafwise regularity along the linear \(W\)-foliation,
the conclusion is independent of the choice of lift.

Since \(p_i\) is linear and $H^{-1}$ is $C^{1+\alpha}$ along
$W$-leaves, it follows that \(h_i\circ H^{-1}\) is
\(C^{1+\alpha}\) along \(W\)-leaves.

It follows from Lemma~\ref{le:8}, applied to
\(\xi=h_i\circ H^{-1}\), that for each \(0\neq n\in\ZZ^N\),
\begin{align*}
 &\left\|
 \int_{\TT^N}
 (h_i\circ H^{-1})(A^{-l_2}z)e_n(z)\,d\mathfrak m(z)
 \right\|\notag\\
 &\leq
 C\|n\|^{(1+\alpha)N}
 \rho_i^{-(1+\alpha)l_2}
 \|h_i\circ H^{-1}\|_{C^{1+\alpha,W}} .
\end{align*}
Therefore
\begin{align*}
 &\Big\|\int_{\TT^N}
A_i^{l_2}(h_i\circ H^{-1})(A^{-l_2}z)e_n(z)\,d\mathfrak m(z)
\Big\|  \\
&\leq
\|A_i^{l_2}\|
\left\|
 \int_{\TT^N}
 (h_i\circ H^{-1})(A^{-l_2}z)e_n(z)\,d\mathfrak m(z)
\right\| \\
&\overset{\text{(1)}}{\leq}C\rho_i^{l_2}\|n\|^{(1+\alpha)N}
\rho_i^{-(\alpha+1) l_2}
\|h_i\circ H^{-1}\|_{C^{1+\alpha,W}}\\
&=C\|n\|^{(1+\alpha)N}
\rho_i^{-\alpha l_2}
\|h_i\circ H^{-1}\|_{C^{1+\alpha,W}}\to0 .
\end{align*}
Here in $(1)$ we use \eqref{for:93} of Section \ref{sec:8}.

This proves \eqref{for:61}. Hence, we obtain \eqref{for:94}.  On the other hand, the positive-time expansion gives (see \eqref{for:92} of Section \ref{sec:14})
\[
h_i\circ H^{-1}
=
\sum_{j=0}^{\infty}
A_i^{-(j+1)}R_i\circ f^j\circ H^{-1}.
\]
Therefore,
\[
h_i^-\circ H^{-1}
=
h_i\circ H^{-1}
\]
as distributions. This proves \eqref{for:158}.
\begin{remark}\label{re:2}   By \eqref{for:158}, the distribution \(h_i^-\circ H^{-1}\) extends
continuously to \(L^\infty(\TT^N)\). More precisely, for every
\(\omega\in L^\infty(\TT^N)\), we have
\[
(h_i^-\circ H^{-1})(\omega)=
\int_{\TT^N}
h_i(H^{-1}z)\omega(z)\,d\mathfrak m(z).
\]
Consequently,
\[
\|(h_i^-\circ H^{-1})(\omega)\|
\leq
\|h_i\|_{C^0}\|\omega\|_{L^\infty}.
\]
\end{remark}

\subsubsection{Step 2:  Definition of $h_i^-$}\label{sec:1} Define formally
\begin{align*}
 h_i^-=-\sum_{j=-1}^{-\infty}A_i^{-(j+1)}R_i\circ f^j.
\end{align*}
In this step, we show that this series converges as a distribution  and that
\begin{align}\label{for:164}
h_i^-=h_i
\end{align}
as bounded linear functionals on \(L^\infty(\TT^N)\).
More precisely, for every
\(\omega\in L^\infty(\TT^N)\), we have
\[
h_i^-(\omega)
=
\int_{\TT^N}
h_i(z)\omega(z)\,d\mathfrak m(z).
\]
Consequently,
\[
\|h_i^-(\omega)\|
\leq
\|h_i\|_{C^0}\|\omega\|_{L^\infty}.
\]
Let \(\omega\in L^\infty(\TT^N)\). Set
\[
\Psi_\omega(z)
:=
\omega(H^{-1}z)\bigl(\kappa(H^{-1}z)\bigr)^{-1}.
\]
We note that \(\Psi_\omega\in L^\infty(\TT^N)\) since \(d\mu=\kappa\,d\mathfrak m\) with $\kappa\in C^{\alpha}(\TT^N)$ and $\kappa>0$.

Then we have
\begin{align*}
h_i^-(\omega)
&=
-\sum_{j=-1}^{-\infty}
\int_{\TT^N}
A_i^{-(j+1)}R_i\circ f^j(x)\omega(x)\,d\mathfrak m(x) \\
&\overset{\text{(1)}}{=}
-\sum_{j=-1}^{-\infty}
\int_{\TT^N}
A_i^{-(j+1)}
R_i\circ f^j\circ H^{-1}(z)
\Psi_\omega(z)\,d\mathfrak m(z) \\
&\overset{\text{(2)}}{=}
(h_i^-\circ H^{-1})(\Psi_\omega)\\
&\overset{\text{(3)}}{=}(h_i\circ H^{-1})(\Psi_\omega) \\
&=
\int_{\TT^N}
h_i(H^{-1}z)\Psi_\omega(z)\,d\mathfrak m(z) \\
&\overset{\text{(4)}}{=}
\int_{\TT^N}
h_i(x)\omega(x)\,d\mathfrak m(x).
\end{align*}
Here, in $(1)$ we use the change of variables $z=H(x)$ and
\begin{align*}
 H_*\mathfrak m
        =
        (\kappa\circ H^{-1})^{-1}\mathfrak m;
\end{align*}
in \((2)\), we use the fact that \(\Psi_\omega\in L^\infty(\TT^N)\) and
Remark~\ref{re:2}, which allows \(h_i^-\circ H^{-1}\) to be evaluated on
$L^\infty(\TT^N)$ test functions; in \((3)\), we use \eqref{for:158} together with
Remark~\ref{re:2}; in $(4)$ we recall that \(d\mu=\kappa\,d\mathfrak m\) and \(H_*\mu=\mathfrak m\).

Hence \(h_i^-=h_i\) as bounded linear functionals on \(L^\infty(\TT^N)\), and
hence also as distributions.

\subsubsection{Step 3: Equality of the local positive and negative distributions}\label{sec:3} We recall notations in Section \ref{sec:16}.  Fix a foliation chart $(\Gamma_k,\,O_k)$ for $\mathcal{W}^f_i$.  Let \(\mathcal Q\) be an \(\alpha\)-H\"older vector field on
\(\Gamma_k(O_k)\) taking values in \(\mathcal E_i\). Let \(C^{c,\beta}_x(O_k)\) denote the space of functions
compactly supported in \(O_k\) that are \(\beta\)-H\"older in the leaf
variable \(x\), uniformly in \(y\). When \(\beta=1\), we interpret
\(C^{c,1}_x(O_k)\) as the space of functions compactly supported in \(O_k\)
that are \(C^1\) in the leaf variable \(x\), uniformly in \(y\).

It follows from Lemma \ref{le:2} that for every \(\varphi\in C^{c,\beta}_x(O_k)\), the series
\begin{align*}
\mathfrak{K}^+_\mathcal Q(\varphi):=\sum_{m=0}^{\infty}\int_{O_k} A_i^{-(m+1)}D(R_i\circ f^{m})_{\Gamma_k(x,y)}\big(\mathcal Q_{\Gamma_k(x,y)})\varphi(x,y)dxd\nu_k(y)
\end{align*}
is convergent and satisfies
\begin{align}\label{for:165}
\norm{\mathfrak{K}^+_\mathcal Q(\varphi)}\leq C_{\beta}\norm{R_i}_{C^1}\norm{\mathcal{Q}}_{C^\alpha}\norm{\varphi}_{C^\beta_x}.
\end{align}
In this step, we show that the series
\begin{align*}
\mathfrak{K}^-_\mathcal Q(\varphi):=-\sum_{m=-1}^{-\infty}\int_{O_k} A_i^{-(m+1)}D(R_i\circ f^{m})_{\Gamma_k(x,y)}\big(\mathcal Q_{\Gamma_k(x,y)})\varphi(x,y)dxd\nu_k(y)
\end{align*}
is convergent as well for any $\varphi\in C^{c,\beta}_x(O_k)$; and
\begin{align*}
 \mathfrak{K}^-_\mathcal Q(\varphi)=\mathfrak{K}^+_\mathcal Q(\varphi).
\end{align*}
Set
\begin{align*}
 \mathcal W^j_{\Gamma_k(x,y)}
:=
D\Gamma_k(x,y)(\partial_{x_j}),\qquad 1\leq j\leq \dim \mathcal{E}_i.
\end{align*}
Firstly, we show that for any $1\leq j\leq \dim \mathcal{E}_i$ and any $\varphi\in C^{c,1}_x(O_k)$, we have
\begin{align*}
 \mathfrak{K}^-_{\mathcal W^j}(\varphi)=\mathfrak{K}^+_{\mathcal W^j}(\varphi).
\end{align*}
Moreover,
\begin{align*}
\Big\|\mathfrak{K}^-_{\mathcal W^j}(\varphi)\Big\|\leq C_{\beta}\norm{R_i}_{C^1}\norm{\varphi}_{C^\beta_x}.
\end{align*}
for any $0<\beta<1$.

For $\varphi\in C_x^{c,1}(O_k)$, the negative-time series satisfies
\begin{align*}
\mathfrak{K}^-_{\mathcal W^j}(\varphi)&=-\sum_{m=-1}^{-\infty}\int_{O_k} A_i^{-(m+1)}D(R_i\circ f^{m})_{\Gamma_k(x,y)}\big(\mathcal W^j_{\Gamma_k(x,y)})\varphi(x,y)dxd\nu_k(y)\\
&=-\sum_{m=-1}^{-\infty}\int_{O_k} A_i^{-(m+1)} \partial_{x_j}\big(R_i\circ f^{m}\circ \Gamma_k(x,y)\big) \varphi(x,y) dxd\nu_k(y)\\
&\overset{\text{(1)}}{=}\sum_{m=-1}^{-\infty}\int_{O_k} A_i^{-(m+1)} R_i\circ f^{m}\circ \Gamma_k(x,y)\psi(x,y) dxd\nu_k(y)\\
&\overset{\text{(2)}}{=}\sum_{m=-1}^{-\infty}\int_{\TT^N} A_i^{-(m+1)} R_i\circ f^{m}(z)\chi_\psi(z)d\mathfrak{m}\\
&=-h_i^-( \chi_\psi ) \\
&\overset{\text{(3)}}{=}-h_i( \chi_\psi  ).
\end{align*}
We explain steps:
\begin{enumerate}

  \item Set $\psi=\partial_{x_j}\varphi$. Then $\psi$ is uniformly bounded and compactly supported on $O_k$.
  \item We recall the measure-change formula \eqref{for:167} in Section \ref{sec:16}. Define
  \[
        \chi_\psi(z)
        :=
        \begin{cases}
        \displaystyle
        \frac{\psi(\Gamma_k^{-1}z)}
        {J_k(\Gamma_k^{-1}z)}, & z\in U_k=\Gamma_k(O_k),\\[1.2em]
        0, & z\notin U_k .
        \end{cases}
\]
Then $\chi_\psi\in L^\infty(\mathbb T^N)$.

  \item By step 2, (see Section \ref{sec:1}) $h_i^-=h_i$ as bounded linear functionals on $L^\infty(\TT^N)$.
 \end{enumerate}
 On the other hand, using \eqref{for:92} of Section \ref{sec:14}, we have
\begin{align*}
 -h_i(\chi_\psi)
&=
-\sum_{m=0}^{\infty}
\int_{\mathbb T^N}
A_i^{-(m+1)}
R_i\circ f^m(z)\chi_\psi(z)\,d\mathfrak m(z) \\
&=
-\sum_{m=0}^{\infty}
\int_{O_k}
A_i^{-(m+1)}
R_i\circ f^m\circ\Gamma_k(x,y)
\psi(x,y)\,dx\,d\nu_k(y) \\
&=
\sum_{m=0}^{\infty}
\int_{O_k}
A_i^{-(m+1)}
D(R_i\circ f^m)_{\Gamma_k(x,y)}
\bigl(\mathcal W^j_{\Gamma_k(x,y)}\bigr)
\varphi(x,y)\,dx\,d\nu_k(y) \\
&=
\mathfrak K_{\mathcal W^j}^+(\varphi).
\end{align*}
 Therefore,
\begin{align*}
 \mathfrak K_{\mathcal W^j}^-(\varphi)
        =
        \mathfrak K_{\mathcal W^j}^+(\varphi),
        \qquad
        \varphi\in C_x^{c,1}(O_k).
\end{align*}
 Combining this equality with \eqref{for:165}, applied to
$\mathcal Q=\mathcal W^j$, gives, for every $0<\beta<1$,
 \begin{align}\label{for:166}
\|\mathfrak K^-_{\mathcal W^j}(\varphi)\|
\leq
C_\beta
\|R_i\|_{C^1}
\|\varphi\|_{C^\beta_x}.
\end{align}
 Since \(C_x^{c,1}(O_k)\) is dense in \(C_x^{c,\beta}(O_k)\) in the
\(C_x^\beta\)-topology, the estimate \eqref{for:166} allows
\(\mathfrak K^-_{\mathcal W^j}\) to extend uniquely to
\(C_x^{c,\beta}(O_k)\). Moreover,
\begin{align}\label{for:95}
 \mathfrak K_{\mathcal W^j}^-(\varphi)
        =
        \mathfrak K_{\mathcal W^j}^+(\varphi),
        \qquad
        \varphi\in C_x^{c,\beta}(O_k).
\end{align}
Now let \(\mathcal Q\) be an \(\alpha\)-H\"older vector field taking
values in \(\mathcal E_i\). We can write
\[
\mathcal Q_{\Gamma_k(x,y)}
=
\sum_{j=1}^{\dim\mathcal E_i}
a_j(x,y)\mathcal W^j_{\Gamma_k(x,y)},
\]
where the coefficients \(a_j\) are \(\alpha\)-H\"older in the leaf
variable \(x\), uniformly in \(y\). Hence, if $\varphi\in C_x^{c,\beta}(O_k)$, then
$a_j\varphi\in C_x^{c,\varrho}(O_k)$, where $\varrho=\min\{\alpha,\beta\}$.
Since \(a_j\varphi\) is \(\varrho\)-H\"older in the leaf
variable, \eqref{for:95} gives
\[
\mathfrak K^-_{\mathcal W^j}(a_j\varphi)
=
\mathfrak K^+_{\mathcal W^j}(a_j\varphi)
\]
for every \(j\). Therefore
\[
\mathfrak K^-_{\mathcal Q}(\varphi)
=
\mathfrak K^+_{\mathcal Q}(\varphi).
\]
The convergence and the estimate for \(\mathfrak K^-_{\mathcal Q}\) follow
from \eqref{for:165}.

 \subsubsection{Step 4: Equality of the global positive and negative distributions}\label{sec:18}
 We recall notations in Section \ref{sec:16}.  Fix a foliation chart $(\Gamma_k,\,O_k)$ for $\mathcal{W}^f_i$.
 Let \(\mathcal Q\) be an \(\alpha\)-\text{H\"older} vector field on
\(\Gamma_k(O_k)\) taking values in \(\mathcal E_i\). Let \(\varphi\) be a function compactly supported in
\(\Gamma_k(O_k)\), and assume that \(\varphi\circ\Gamma_k\) is
\(\beta\)-\text{H\"older} in the leaf variable \(x\), uniformly in \(y\). From Corollary \ref{cor:2} the series
\begin{align*}
\widetilde{\mathfrak D}_{\mathcal Q}^+(\varphi):=\sum_{m=0}^{\infty}
\int_{\TT^N}
A_i^{-(m+1)}
D(R_i\circ f^m)_z(\mathcal Q_z)
\,\varphi(z)\,d\mathfrak m(z)
\end{align*}
is convergent and satisfies
\begin{align}\label{for:168}
\norm{\widetilde{\mathfrak D}_{\mathcal Q}^+(\varphi)}\leq C_{\beta}\norm{R_i}_{C^1}\norm{\mathcal{Q}}_{C^\alpha}\norm{\varphi}_{C^\beta}.
\end{align}
In this part, we show that the series
\begin{align*}
\widetilde{\mathfrak D}_{\mathcal Q}^-(\varphi):=-\sum_{m=-1}^{-\infty}\int_{\TT^N}
A_i^{-(m+1)}
D(R_i\circ f^m)_z(\mathcal Q_z)
\,\varphi(z)\,d\mathfrak m(z)
\end{align*}
is convergent as well; and satisfies
\begin{align*}
 \widetilde{\mathfrak D}_{\mathcal Q}^-(\varphi)=\widetilde{\mathfrak D}_{\mathcal Q}^+(\varphi).
\end{align*}
By using the measure-change formula \eqref{for:167} in Section \ref{sec:16}, we have
\begin{align*}
 &\int_{\TT^N} A_i^{-(m+1)}D(R_i\circ f^{m})_{z}(\mathcal Q_{z})\varphi(z)d\mathfrak{m}\\
 &=\int_{O_k} A_i^{-(m+1)}D(R_i\circ f^{m})_{\Gamma_k(x,y)}(\mathcal Q_{\Gamma_k(x,y)})\varphi(\Gamma_k(x,y))J_k(x,y)dxd\nu_k(y).
\end{align*}
Since $J_k$ is H\"older in the leaf variable, we have
\[
        (\varphi\circ\Gamma_k)(x,y)J_k(x,y)\in C_x^{c,\varrho}(O_k),\qquad \varrho=\min\{\alpha,\beta\}.
\]
Then it follows from step 3 (see Section \ref{sec:3}) that
\begin{align*}
 \widetilde{\mathfrak D}_{\mathcal Q}^-(\varphi)=\mathfrak{K}^-_{\mathcal Q}\bigl((\varphi\circ\Gamma_k)J_k\bigr)=\mathfrak{K}^+_{\mathcal Q}\bigl((\varphi\circ\Gamma_k)J_k\bigr)=
 \widetilde{\mathfrak D}_{\mathcal Q}^+(\varphi).
\end{align*}
It remains to pass from functions supported in one foliation box to arbitrary
\(\beta\)-H\"older functions on \(\TT^N\). Let
\(\{\Gamma_k(O_k)\}_{k=1}^M\) be a finite cover of \(\TT^N\) by foliation
boxes, and let \(\{\chi_k\}_{k=1}^M\) be a \(C^\infty\) partition of unity
subordinate to this cover. For a global \(\alpha\)-\text{H\"older} vector field
\(\mathcal V\) taking values in \(\mathcal E_i\) and for
\(\varphi\in C^\beta(\TT^N)\), we write
\[
\varphi=\sum_{k=1}^M \chi_k\varphi.
\]
Each \(\chi_k\varphi\) is supported in \(\Gamma_k(O_k)\), and
\[
\|(\chi_k\varphi)\circ\Gamma_k\|_{C^\beta_x}
\leq
C_k\|\varphi\|_{C^\beta}.
\]
Applying the local equality to
\(\mathcal Q=\mathcal V|_{\Gamma_k(O_k)}\) and to \(\chi_k\varphi\), and then
summing over \(k\), gives
\[
\widetilde{\mathfrak D}_{\mathcal V}^-(\varphi)
=
\widetilde{\mathfrak D}_{\mathcal V}^+(\varphi).
\]
Moreover, summing the local estimates \eqref{for:168} over the finite cover
gives
\[
\|\widetilde{\mathfrak D}_{\mathcal V}^-(\varphi)\|
\leq
C_\beta
\|R_i\|_{C^1}
\|\mathcal V\|_{C^\alpha}
\|\varphi\|_{C^\beta}.
\]
Together with the corresponding estimate for \(\widetilde{\mathfrak D}_{\mathcal V}^+\),
this completes the proof.

\subsubsection{Step 5: Conclusion} We now return to the distributions in Theorem~\ref{th:3}. For every
$m\in\mathbb Z$, changing variables $z=H(x)$ gives
\begin{align*}
&\int_{\TT^N}
A_i^{-(m+1)}
D(R_i\circ f^m)_{H^{-1}z}(\mathcal V_{H^{-1}z})
\,\omega(z)\,d\mathfrak m(z)\\
&=\int_{\TT^N}
A_i^{-(m+1)}
D(R_i\circ f^m)_{z}(\mathcal V_{z})
\,\omega\circ H(z)\kappa(z)\,d\mathfrak m(z).
\end{align*}
Thus
\[
        \mathfrak D_{\mathcal V}^{\pm}(\omega)
        =
        \widetilde{\mathfrak D}_{\mathcal V}^{\pm}
        \bigl((\omega\circ H)\kappa\bigr).
\]
Recall that $H$ is bi-$\eta$ H\"older (see Section \ref{sec:5}). Since
$\omega\in C^\beta(\mathbb T^N)$ and $\kappa\in C^{\alpha}(\mathbb T^N)$ is
positive, we have
\[
       (\omega\circ H)\kappa\in C^{\bar\beta},
        \qquad
        \bar\beta:=\min\{\beta\eta,\alpha\}.
\]
Moreover,
\[
        \|(\omega\circ H)\kappa\|_{C^{\bar\beta}}
        \leq
        C_\beta\|\omega\|_{C^\beta}.
\]
Applying step 4 (see Section \ref{sec:18}) with exponent \(\beta\eta\), we obtain
\[
\widetilde{\mathfrak D}_{\mathcal V}^-\big((\omega\circ H)\cdot\kappa\big)
=
\widetilde{\mathfrak D}_{\mathcal V}^+\big((\omega\circ H)\cdot\kappa\big),
\]
which is exactly
\[
\mathfrak D^-_{\mathcal V}(\omega)
=
\mathfrak D^+_{\mathcal V}(\omega).
\]
The same argument, together with \eqref{for:168} in Step~4, gives
\[
\max\{
\|\mathfrak D^+_{\mathcal V}(\omega)\|,
\|\mathfrak D^-_{\mathcal V}(\omega)\|
\}
\leq
C_\beta
\|\mathcal V\|_{C^\alpha}
\|\omega\|_{C^\beta}.
\]
This completes the proof of Theorem~\ref{th:3}.

\section{$C^{1+\text{H\"older}}$ regularity of $H$ along $\mathcal{W}^f_i$}\label{sec:38}
We recall the following notation:
\begin{enumerate}
  \item $\mathfrak m$ and $\mu$ are defined in Section~\ref{sec:12}.
  \item The numbers $\rho_i$, the Lyapunov blocks $\mathcal E_i$, and the
  linear subspaces $E_i$ are defined in Section~\ref{sec:12}.

  \item The maps $A_i$ and $p_i$ are defined in Section~\ref{sec:14}.

  \item $H_i$ is defined in \eqref{for:7} of Section \ref{sec:39}.

  \item The foliations $\mathcal W_i^f$ and $\mathcal W_i^A$ are defined in
  \eqref{for:23} of Section~\ref{sec:12}.
  \item The subbundles $\mathcal F_{i,j}$ are defined in \eqref{for:33} of
  Section~\ref{sec:12}.
  \item The subspaces $V_{i,j}$ are defined in Theorem \ref{th:9}.
  \item The distribution $\widetilde{\mathfrak D}_{\mathcal V}^+$ is defined in Theorem~\ref{th:3}.
\end{enumerate}
\begin{theorem}\label{th:4}
Suppose that \(H(\mathcal W_i^f)=\mathcal W_i^A\) for some
\(i_0\leq i\leq \ell\). Then \(H\) is a \(C^{1+\alpha}\) diffeomorphism along
  \(\mathcal W^f_{i}\).
\end{theorem}
The key step in the proof of Theorem~\ref{th:4} is the following proposition.
\begin{proposition}\label{le:3}
For each $1\leq j\leq j_i$, there exist an
\(A_i\)-invariant subspace \(S_{i,j}\subset E_i\), a constant linear
isomorphism
\[
L_j:V_{i,j}\to S_{i,j},
\]
and a bi-\(\alpha\)-H\"older bundle
isomorphism
\[
\mathcal B_{i,j}:\mathcal F_{i,j}\to S_{i,j}
\]
 such that for any \(\alpha\)-H\"older vector
field \(\mathcal V\) on \(\TT^N\)  taking values in \(\mathcal{F}_{i,j}\), we have
\begin{align*}
 \mathcal{B}_{i,j}(\cdot)(\mathcal V)=\widetilde{\mathfrak D}_{\mathcal V}^++p_{i}(\mathcal V)\qquad\text{as distributions}.
\end{align*}
More precisely, for every \(\omega\in C^\beta(\TT^N)\),
\[
\widetilde{\mathfrak D}_{\mathcal V}^+(\omega)
+
\int_{\TT^N}p_{i}(\mathcal V_x)\omega(x)\,d\mathfrak m(x)
=
\int_{\TT^N}
\mathcal B_{i,j}(x)(\mathcal V_x)\omega(x)\,d\mathfrak m(x).
\]
\end{proposition}

\begin{remark}
Equivalently, $\mathcal B_{i,j}$ represents the derivative of
$H_i=p_i\circ H=p_i+h_i$ along $\mathcal F_{i,j}$.
\end{remark}

\subsection{Proof of Theorem \ref{th:4}} We prove Theorem \ref{th:4} assuming Proposition \ref{le:3}. By Proposition \ref{le:3},  applied with \(j=j_i\), there exists an
\(A_i\)-invariant subspace \(S_{i,j_i}\subset E_i\), a constant linear
isomorphism
\[
L_{j_i}:V_{i,j_i}\to S_{i,j_i},
\]
and a bi-\(\alpha\)-H\"older bundle isomorphism
\[
\mathcal B_{i,j_i}:\mathcal F_{i,j_i}\to S_{i,j_i}
\]
such that \(\mathcal B_{i,j_i}\) represents the derivative of
\(H_i=p_i\circ H\) along \(\mathcal F_{i,j_i}\).

Since
\[
\mathcal F_{i,j_i}=\mathcal E_i=T\mathcal W_i^f
\qquad\text{and}\qquad
V_{i,j_i}=E_i,
\]
we have
\[
\dim S_{i,j_i}=\dim V_{i,j_i}=\dim E_i.
\]
Because \(S_{i,j_i}\subset E_i\), it follows that
\[
S_{i,j_i}=E_i.
\]
Therefore
\[
\mathcal B_{i,j_i}:\mathcal E_i\to E_i
\]
is a bi-\(\alpha\)-\text{H\"older} bundle isomorphism, and for every
\(x\in\TT^N\),
\[
D(H_i)_x|_{\mathcal E_i(x)}
=
\mathcal B_{i,j_i}(x).
\]
Hence \(H_i\) is \(C^{1+\alpha}\) along \(\mathcal W_i^f\), and its
leafwise derivative is everywhere invertible.

Since \(H(\mathcal W_i^f)=\mathcal W_i^A\), the components
\(p_k\circ H\), \(k\neq i\), are constant along each leaf of
\(\mathcal W_i^f\). Thus, for every
\(u\in\mathcal E_i(x)\),
\[
DH_x(u)
=
D(H_i)_x(u)
=
\mathcal B_{i,j_i}(x)(u).
\]
Consequently \(H\) is \(C^{1+\alpha}\) along \(\mathcal W_i^f\), and
its leafwise derivative is an isomorphism at every point. The inverse function
theorem on the leaves implies that \(H\) is a
\(C^{1+\alpha}\) diffeomorphism along \(\mathcal W_i^f\).

It remains to prove Proposition~\ref{le:3}; this will occupy the rest of the section.

\subsection{Proof strategy for Proposition~\ref{le:3}}

We prove Proposition~\ref{le:3} by induction on $j$. The base case $j=1$
follows from Theorem~\ref{th:1}; see Section~\ref{sec:27}.

Assume that Proposition~\ref{le:3} holds for some $j<j_i$. Then
$\mathcal B_{i,j}$ represents the derivative of $H_i$ along
$\mathcal F_{i,j}$. Differentiating the conjugacy equation
\[
        H_i\circ f=A_i\circ H_i
\]
along $\mathcal F_{i,j}$, we obtain
\[
        \mathcal B_{i,j}(fx)\circ Df_x|_{\mathcal F_{i,j}(x)}
        =
        A_i\circ \mathcal B_{i,j}(x).
\]
This gives a H\"older conjugacy between $Df|_{\mathcal F_{i,j}}$ and the
constant cocycle $A_i|_{S_{i,j}}$.

Using this conjugacy, together with the quotient trivialization supplied by
Theorem~\ref{th:9}, we choose a suitable trivialization
\[
        \mathcal O_{i,j+1}:\mathcal F_{i,j+1}\to V_{i,j+1}
\]
in which the cocycle $Df|_{\mathcal F_{i,j+1}}$ becomes upper triangular:
\[
        \widehat A_{i,j+1}(x)
        =
        \begin{pmatrix}
        \rho_i M_j & \theta_j(x)\\
        0 & \rho_i L_{i,j+1}
        \end{pmatrix},
\]
where the diagonal blocks have uniformly bounded normalized iterates. This
triangular form gives precise formulas for both positive and negative iterates
of $Df|_{\mathcal F_{i,j+1}}$; see Sections~\ref{sec:26} and~\ref{sec:28}.
It is the main input for controlling the growth of
$Df^m|_{\mathcal F_{i,j+1}}$ as $m\to\pm\infty$.

A key observation is that, for every $\alpha$-H\"older vector field
$\mathcal V$ taking values in $\mathcal F_{i,j+1}$, the positive-time
normalized expression satisfies
\[
        \lim_{m\to\infty}
        A_i^{-m}p_i\circ Df_z^m(\mathcal V_z)
        =
        p_i(\mathcal V_z)+\widetilde{\mathfrak D}_{\mathcal V}^+
\]
in the sense of distributions. Similarly, the negative-time expression
satisfies
\[
        \lim_{m\to-\infty}
        A_i^{-m}p_i\circ Df_z^m(\mathcal V_z)
        =
        p_i(\mathcal V_z)+\widetilde{\mathfrak D}_{\mathcal V}^- .
\]
The positive-time formula has good H\"older control along stable leaves of
$f$, while the negative-time formula has good H\"older control along unstable
leaves of $f$. However, the distribution-to-H\"older criterion is most
naturally applied on the linear torus, where the stable and unstable foliations
are linear and translation differences can be used.

For this reason, we trivialize the bundle using $\mathcal O_{i,j+1}$ and pass
to the linear base using $H^{-1}$. In these coordinates we work with the
conjugated distributions
\[
        \mathfrak D_{\mathcal V}^{+},
        \qquad
        \mathfrak D_{\mathcal V}^{-},
\]
and with the candidate distribution
\[
\begin{aligned}
 \mathcal T
 &:=
 \lim_{m\to \infty}
 A_i^{-m}p_i\circ Df^{m}_{H^{-1}z}\circ
 \mathcal O_{i,j+1}^{-1}(H^{-1}z)|_{V_{i,j+1}}  \\
 &=
 \lim_{m\to -\infty}
 A_i^{-m}p_i\circ Df^{m}_{H^{-1}z}\circ
 \mathcal O_{i,j+1}^{-1}(H^{-1}z)|_{V_{i,j+1}},
\end{aligned}
\]
where the equality is proved in Section~\ref{sec:20}. The identity
\[
        \mathfrak D_{\mathcal V}^{+}
        =
        \mathfrak D_{\mathcal V}^{-}
\]
combines the stable-direction estimates coming from the positive-time formula
with the unstable-direction estimates coming from the negative-time formula.
This gives the two-sided translation estimates required by the
distribution-to-H\"older criterion. Hence $\mathcal T$ is represented by a
H\"older function; see Section~\ref{sec:19}.

We then define
\[
        \mathcal B_{i,j+1}(x)
        :=
        \mathcal T(Hx)\circ \mathcal O_{i,j+1}(x).
\]
By construction, for every $\alpha$-H\"older vector field $\mathcal V$ taking
values in $\mathcal F_{i,j+1}$,
\[
\widetilde{\mathfrak D}_{\mathcal V}^+(\omega)
+
\int_{\TT^N}p_i(\mathcal V_x)\omega(x)\,d\mathfrak m(x)
=
\int_{\TT^N}
\mathcal B_{i,j+1}(x)(\mathcal V_x)\omega(x)\,d\mathfrak m(x).
\]
Thus $\mathcal B_{i,j+1}$ represents
$D(H_i)|_{\mathcal F_{i,j+1}}$.

It remains to prove that $\mathcal B_{i,j+1}$ is a bi-$\alpha$-H\"older bundle
isomorphism onto a constant $A_i$-invariant subspace. First, one proves
injectivity on a full-measure set of generic points. The corresponding
measurable image bundle is $A_i$-invariant. Since $A_i|_{E_i}$ has only one
Lyapunov exponent and is fiber bunched, the measurable image bundle is
H\"older by the regularity theorem for invariant subbundles of
Kalinin--Sadovskaya~\cite{KS13}. This regularity, together with continuity and
full support of $\mu$, promotes injectivity to every point.

Finally, using periodic points of suitable iterates of $A$, one shows that the
H\"older image bundle over the linear system is actually constant. Thus the
image is a fixed $A_i$-invariant subspace $S_{i,j+1}\subset E_i$.
Theorem~2.7 of \cite{S15} then upgrades $\mathcal B_{i,j+1}$ to an
$\alpha$-H\"older conjugacy, and compactness gives $\alpha$-H\"older
regularity of the inverse. Hence $\mathcal B_{i,j+1}$ is bi-$\alpha$-H\"older.
This completes the induction step, and therefore proves
Proposition~\ref{le:3}. The details of this final step are given in
Section~\ref{sec:29}.

\subsection{Proof of Proposition \ref{le:3}}
\subsubsection{A criterion for distributions to be H\"older}

\begin{lemma}\label{le:7}
Let \(\mathcal{T}\) be a distribution on \(\TT^N\). Suppose that, for every
\(v\in\RR^N\) with \(\|v\|\leq 1\), the distribution
\[
F_v(x):=\mathcal{T}(x)-\mathcal{T}(x+v)
\]
is a \(C^\alpha\) function, and assume moreover that
\[
\sup_{\|v\|\leq 1}\|F_v\|_{C^\alpha}<\infty.
\]
Then \(\mathcal{T}\) is a \(C^\alpha\) function.
\end{lemma}

\begin{proof}
Let \(\phi\in C_c^\infty(\RR^N)\) be supported in the ball
\(\{\|v\|\leq 1\}\), with \(\phi\geq 0\) and
\[
\int_{\RR^N}\phi(v)\,dv=1.
\]
Translations are understood modulo \(\ZZ^N\). Define
\[
G(x)
:=
\mathcal{T}(x)-\int_{\RR^N}\mathcal{T}(x+v)\phi(v)\,dv.
\]
Then, in the sense of distributions,
\[
G(x)
=
\int_{\RR^N}\bigl(\mathcal{T}(x)-\mathcal{T}(x+v)\bigr)\phi(v)\,dv
=
\int_{\RR^N}F_v(x)\phi(v)\,dv.
\]
By the uniform \(C^\alpha\) bound on \(F_v\), it follows that \(G\in C^\alpha\).
Indeed,
\[
\|G\|_{C^\alpha}
\leq
\int_{\RR^N}\|F_v\|_{C^\alpha}\phi(v)\,dv
\leq
\sup_{\|v\|\leq 1}\|F_v\|_{C^\alpha}.
\]
On the other hand,
\[
x\mapsto \int_{\RR^N}\mathcal{T}(x+v)\phi(v)\,dv
\]
is the convolution of the distribution \(T\) with a smooth compactly supported
kernel, and is therefore \(C^\infty\). Hence
\[
\mathcal{T}
=
G+
\int_{\RR^N}\mathcal{T}(\cdot+v)\phi(v)\,dv
\]
is represented by a \(C^\alpha\) function.
\end{proof}

\subsubsection{Step 1: Preparatory step}\label{sec:26} We recall notations in Theorem \ref{th:9}. By Theorem~\ref{th:9}, there exists a bi-\(\alpha\)-H\"older bundle
isomorphism
\[
\mathcal C_i:\mathcal E_i\to E_i
\]
such that, for every \(x\in\TT^N\), the map
\[
\mathcal C_i(x):\mathcal E_i(x)\to E_i
\]
is a linear isomorphism satisfying
\begin{align}\label{for:67}
 \mathcal C_i(fx)\circ Df_x|_{\mathcal E_i(x)}
=
\widetilde A_i(x)\circ \mathcal C_i(x),
\qquad x\in\TT^N,
\end{align}
where \(\widetilde A_i(x):E_i\to E_i\) is block upper triangular and its
diagonal blocks coincide with the corresponding diagonal blocks of \(A_i\).
Moreover, for every \(0\leq j\leq j_i\),
\[
\mathcal C_i(x)(\mathcal F_{i,j}(x))=V_{i,j}.
\]
We denote the restriction by
\[
\mathcal C_{i,j}(x):=\mathcal C_i(x)|_{\mathcal F_{i,j}(x)}.
\]
We prove Proposition \ref{le:3} by induction in the following steps.
\subsubsection{Step 2: The base case \(j=1\) and the inductive assumption}\label{sec:27} By Theorem~\ref{th:1}, \(H\) is a \(C^{1+\text{H\"older}}\) diffeomorphism along
  \(\mathcal W_{\mathcal F_{i,1}}\). Moreover, for every
  \(x\in\TT^N\) and \(u\in \mathcal F_{i,1}(x)\),
  \[
  D(H_i)_x(u)=B_i\circ \mathcal C_{i,1}(x)(u).
  \]
Set
\begin{align*}
 L_1:=B_i,\quad \mathcal{B}_{i,1}:=B_i\circ \mathcal C_{i,1},\quad S_{i,1}:=B_i(V_{i,1}).
\end{align*}
Since
\begin{align*}
 H_i=p_i+h_i=p_i+\sum_{m=0}^\infty A_i^{-(m+1)}R_i\circ f^{m},
\end{align*}
we have, in the sense of distributions along \(\mathcal F_{i,1}\),
 \begin{align*}
  D(H_i)_x(u)=p_{i}(u)+\sum_{m=0}^\infty A_i^{-(m+1)}D(R_i\circ f^{m})(u).
 \end{align*}
Therefore, for every \(\alpha\)-H\"older vector field
\(\mathcal V\) taking values in \(\mathcal F_{i,1}\) and every
\(\omega\in C^\beta(\TT^N)\),
\[
\widetilde{\mathfrak D}_{\mathcal V}^+(\omega)
+
\int_{\TT^N}p_i(\mathcal V_x)\omega(x)\,d\mathfrak m(x)
=
\int_{\TT^N}
\mathcal B_{i,1}(x)(\mathcal V_x)\omega(x)\,d\mathfrak m(x).
\]
Thus Proposition~\ref{le:3} holds for \(j=1\).

Now suppose that Proposition~\ref{le:3} holds for some
\(1\leq j\leq j_i-1\). Then \(H_i\) is \(C^{1+\alpha}\) along \(\mathcal F_{i,j}\). Differentiating the conjugacy equation
\[
H_i\circ f=A_i\circ H_i
\]
along $\mathcal{F}_{i,j}$, we have
\begin{align*}
\mathcal B_{i,j}(fx)\circ Df_x|_{\mathcal{F}_{i,j}}=(A_i|_{S_{i,j}})\circ\mathcal B_{i,j}(x)|_{\mathcal{F}_{i,j}},\qquad \forall\,x\in\TT^N.
\end{align*}
Equivalently,
\begin{align}\label{for:64}
B_j^{-1}\mathcal B_{i,j}(fx)\circ Df_x|_{\mathcal{F}_{i,j}}=(B_j^{-1}A_iB_j)|_{V_{i,j}}\circ B_j^{-1}\mathcal B_{i,j}(x)|_{\mathcal{F}_{i,j}},\qquad \forall\,x\in\TT^N.
\end{align}
\subsubsection{Step 3: Expressions for $Df^m|_{\mathcal{F}_{i,j+1}}$ }\label{sec:28}
Choose an \(\alpha\)-\text{H\"older} splitting
\[
\mathcal F_{i,j+1}
=
\mathcal F_{i,j}\oplus \mathcal G_{i,j+1}.
\]
By Theorem~\ref{th:9}, we also have the \(A\)-invariant decomposition
\[
V_{i,j+1}=V_{i,j}\oplus W_{i,j+1}.
\]
On \(\mathcal F_{i,j}\), define
\[
\mathcal O_{i,j+1}(x)|_{\mathcal F_{i,j}(x)}
=
B_j^{-1}\mathcal B_{i,j}(x).
\]
On the quotient \(\mathcal F_{i,j+1}/\mathcal F_{i,j}\), the map
\(\mathcal C_i\) induces a bi-\(\alpha\)-\text{H\"older} bundle isomorphism
\[
\overline{\mathcal C}_{i,j+1}(x):
\mathcal F_{i,j+1}(x)/\mathcal F_{i,j}(x)
\to
V_{i,j+1}/V_{i,j}.
\]
Using the decomposition \(V_{i,j+1}=V_{i,j}\oplus W_{i,j+1}\), we identify
\(V_{i,j+1}/V_{i,j}\) with \(W_{i,j+1}\). Also, using the splitting
\[
\mathcal F_{i,j+1}
=
\mathcal F_{i,j}\oplus\mathcal G_{i,j+1},
\]
we identify \(\mathcal G_{i,j+1}(x)\) with
\(\mathcal F_{i,j+1}(x)/\mathcal F_{i,j}(x)\). We then define
\[
\mathcal O'_{i,j+1}(x):\mathcal G_{i,j+1}(x)\to W_{i,j+1}
\]
to be the corresponding quotient trivialization. With this choice, the
induced quotient cocycle $\overline{Df}:\mathcal F_{i,j+1}/\mathcal F_{i,j}\to \mathcal F_{i,j+1}/\mathcal F_{i,j}$ is conjugated to the next diagonal block
\(\rho_iL_{i,j+1}\) of \(\widetilde A_i(x)\), as follows from \eqref{for:67}. Define
\begin{align*}
 \mathcal O_{i,j+1}(x)(u+v)
=
B_j^{-1}\mathcal B_{i,j}(x)u
+
\mathcal O'_{i,j+1}(x)v,
\end{align*}
where $u\in\mathcal F_{i,j}(x)$, $v\in\mathcal G_{i,j+1}(x)$.

Then \(\mathcal O_{i,j+1}\) is a bi-\(\alpha\)-H\"older bundle
isomorphism
\[
\mathcal O_{i,j+1}:\mathcal F_{i,j+1}\to V_{i,j+1}
\]
satisfying
\[
\mathcal O_{i,j+1}(x)|_{\mathcal F_{i,j}(x)}
=
B_j^{-1}\mathcal B_{i,j}(x).
\]
Consider the cocycle
\[
\widehat A_{i,j+1}(x)
:=
\mathcal O_{i,j+1}(fx)\circ Df_x|_{\mathcal F_{i,j+1}(x)}
\circ\mathcal O_{i,j+1}(x)^{-1}.
\]
With respect to the decomposition
\[
V_{i,j+1}=V_{i,j}\oplus W_{i,j+1},
\]
this cocycle has block upper-triangular form
\[
\widehat A_{i,j+1}(x)
=
\begin{pmatrix}
\rho_iM_j & \theta_j(x)\\
0 & \rho_iL_{i,j+1}
\end{pmatrix}.
\]
Indeed, the lower-left block is zero because
\[
Df_x\mathcal F_{i,j}(x)=\mathcal F_{i,j}(fx).
\]
The upper-left block is
\[
\rho_iM_j
=
(B_j^{-1}A_iB_j)|_{V_{i,j}},
\]
by \eqref{for:64}, where
\[
M_j:=B_j^{-1}\circ(\rho_i^{-1}A_i)|_{S_{i,j}}\circ B_j.
\]
The lower-right block is \(\rho_iL_{i,j+1}\) by the choice of
\(\mathcal O'_{i,j+1}\) and the conjugacy relation \eqref{for:67}. Finally,
the off-diagonal term \(\theta_j(x)\) is \(\alpha\)-\text{H\"older}, because
\(Df\), \(\mathcal O_{i,j+1}\), and \(\mathcal O_{i,j+1}^{-1}\) are
\(\alpha\)-H\"older.

Therefore, for every \(m\geq1\),
\begin{align}\label{for:68}
 Df_x^m|_{\mathcal F_{i,j+1}(x)}
=
\mathcal O_{i,j+1}(f^m x)^{-1}
\widetilde A_{i,j+1,m}(x)
\mathcal O_{i,j+1}(x),
\end{align}
where
\[
\widetilde A_{i,j+1,m}(x)
=
\prod_{k=0}^{m-1}
\begin{pmatrix}
\rho_i M_j & \theta_j(f^k x)\\
0 & \rho_i L_{i,j+1}
\end{pmatrix}.
\]
Equivalently,
\[
\widetilde A_{i,j+1,m}(x)
=\rho_i^m
\begin{pmatrix}
M_j^m & \rho_i^{-1}\theta_{j,m}(x)\\
0 & L_{i,j+1}^m
\end{pmatrix},
\]
where
\begin{align}\label{for:74}
 \theta_{j,m}(x)
=
\sum_{k=0}^{m-1}
M_j^{m-1-k}\theta_j(f^k x)L_{i,j+1}^k.
\end{align}
This together with \eqref{for:68} show that for every \(m\geq1\),
\begin{align}\label{for:66}
Df_x^m|_{\mathcal F_{i,j+1}(x)}
=\rho_i^m
\mathcal O_{i,j+1}(f^m x)^{-1}
\begin{pmatrix}
M_j^m & \rho_i^{-1}\theta_{j,m}(x)\\
0 & L_{i,j+1}^m
\end{pmatrix}
\mathcal O_{i,j+1}(x),
\end{align}
By construction, the normalized blocks $M_j$ and $L_{i,j+1}$ have uniformly bounded
integer powers. Thus
\begin{align}\label{for:76}
 \max\{\norm{M_j^m},\, \norm{L_{i,j+1}^m}\}\leq C,\qquad \forall\,m\in\ZZ.
\end{align}
Similarly, for negative iterates, writing \(n\geq1\), we have
\begin{align}
Df_x^{-n}|_{\mathcal F_{i,j+1}(x)}
=
\rho_i^{-n}
\mathcal O_{i,j+1}(f^{-n}x)^{-1}
\begin{pmatrix}
M_j^{-n} & \eta_{j,n}(x)\\
0 & L_{i,j+1}^{-n}
\end{pmatrix}
\mathcal O_{i,j+1}(x),
\end{align}
where
\[
\eta_{j,n}(x)
=
-\rho_i^{-1}
\sum_{k=1}^{n}
M_j^{-(n-k+1)}\theta_j(f^{-k}x)L_{i,j+1}^{-k}.
\]

\subsubsection{Step 4: Equality of the positive- and negative-time limits}\label{sec:20} In this part, we show that
\begin{align*}
 \mathcal{T}:&=\lim_{m\to \infty}A_i^{-m}p_i\circ Df^{m}_{H^{-1}z}\circ \mathcal{O}_{i,j+1}^{-1}(H^{-1}z)|_{V_{i,j+1}}\\
 &=\lim_{m\to -\infty}A_i^{-m}p_i\circ Df^{m}_{H^{-1}z}\circ \mathcal{O}_{i,j+1}^{-1}(H^{-1}z)|_{V_{i,j+1}}
\end{align*}
as distributions.

We first recall that the trivialization
\[
        \mathcal O_{i,j+1}:\mathcal F_{i,j+1}\to V_{i,j+1}
\]
identifies $\alpha$-H\"older vector fields tangent to
$\mathcal F_{i,j+1}$ with $\alpha$-H\"older
$V_{i,j+1}$-valued functions on $\mathbb T^N$. More precisely, if
$\mathcal V$ is an $\alpha$-H\"older vector field taking values in
$\mathcal F_{i,j+1}$, then
\[
        \mathfrak v_x:=\mathcal O_{i,j+1}(x)(\mathcal V_x)
\]
is an $\alpha$-H\"older $V_{i,j+1}$-valued function. Conversely, if
$\mathfrak v$ is an $\alpha$-H\"older $V_{i,j+1}$-valued function, then
\[
        \mathcal V_x:=\mathcal O_{i,j+1}^{-1}(x)(\mathfrak v_x)
\]
is an $\alpha$-H\"older vector field taking values in
$\mathcal F_{i,j+1}$. Moreover, since both $\mathcal O_{i,j+1}$ and
$\mathcal O_{i,j+1}^{-1}$ are H\"older bundle maps over the compact base, the
corresponding $C^\alpha$ norms are comparable. Thus it is equivalent to test
the candidate limit against $\alpha$-H\"older
$V_{i,j+1}$-valued functions or against $\alpha$-H\"older vector fields
tangent to $\mathcal F_{i,j+1}$.

More precisely, let \(\mathfrak v\) be an \(\alpha\)-\text{H\"older}
\(V_{i,j+1}\)-valued function on \(\TT^N\), and let
\(\omega\in C^\beta(\TT^N)\). We prove that
\begin{align*}
&\lim_{m\to \infty} \int_{\TT^N}A_i^{-m}p_i\circ Df^{m}_{H^{-1}z}\circ \mathcal{O}_{i,j+1}^{-1}(H^{-1}z)(\mathfrak{v}_{H^{-1}z})\omega(z)d\mathfrak{m}\\
&=\lim_{m\to -\infty} \int_{\TT^N}A_i^{-m}p_i\circ Df^{m}_{H^{-1}z}\circ \mathcal{O}_{i,j+1}^{-1}(H^{-1}z)(\mathfrak{v}_{H^{-1}z})\omega(z)d\mathfrak{m}.
\end{align*}
Let
\begin{align*}
J(x)=-p_i|_{(\mathcal{F}_{i,j+1})_x}.
\end{align*}
Then \(J\) satisfies the following twisted cohomological equation over \(f\):
\begin{align}\label{for:65}
 A_i\circ J(x)-J(fx)\circ Df|_{(\mathcal{F}_{i,j+1})_x}=DR_i|_{(\mathcal{F}_{i,j+1})_x}, \qquad x\in \TT^N.
\end{align}
Iterating \eqref{for:65} forward, for every \(m\geq1\), we have
\[
J(x)
+
A_i^{-m}p_i\circ Df_x^m|_{\mathcal F_{i,j+1}(x)}
=
\sum_{n=0}^{m-1}
A_i^{-(n+1)}
D(R_i\circ f^n)_x|_{\mathcal F_{i,j+1}(x)}.
\]
Now take \(x=H^{-1}z\) and compose on the right with
\[
\mathcal O_{i,j+1}^{-1}(H^{-1}z)|_{V_{i,j+1}}.
\]
Then we have
\begin{align}\label{for:72}
 &-p_i\circ \mathcal{O}_{i,j+1}^{-1}(H^{-1}z)|_{V_{i,j+1}}+A_i^{-m}p_i\circ Df^{m}_{H^{-1}z}\circ \mathcal{O}_{i,j+1}^{-1}(H^{-1}z)|_{V_{i,j+1}}\notag\\
 &=\sum_{n=0}^{m-1}A_i^{-(n+1)}D(R_i\circ f^{n})_{H^{-1}z}\circ \mathcal{O}_{i,j+1}^{-1}(H^{-1}z)|_{V_{i,j+1}}.
\end{align}
On the other hand, iterating \eqref{for:65} backward, for every
\(m\geq 1\), we obtain
\[
J(x)
+
A_i^{m}p_i\circ Df_x^{-m}|_{\mathcal F_{i,j+1}(x)}
=
-
\sum_{n=-1}^{-m}
A_i^{-(n+1)}
D(R_i\circ f^n)_x|_{\mathcal F_{i,j+1}(x)}.
\]
Now take again \(x=H^{-1}z\) and compose on the right with
\[
\mathcal O_{i,j+1}^{-1}(H^{-1}z)|_{V_{i,j+1}}.
\]
Then
we get
\begin{align}\label{for:73}
 &-
p_i\circ
\mathcal O_{i,j+1}^{-1}(H^{-1}z)|_{V_{i,j+1}}+A_i^{m}p_i\circ Df^{-m}_{H^{-1}z}
\circ
\mathcal O_{i,j+1}^{-1}(H^{-1}z)|_{V_{i,j+1}}\notag\\
&=-
\sum_{n=-m}^{-1}
A_i^{-(n+1)}
D(R_i\circ f^n)_{H^{-1}z}
\circ
\mathcal O_{i,j+1}^{-1}(H^{-1}z)|_{V_{i,j+1}}.
\end{align}
Now define an \(\alpha\)-H\"older vector field
\[
\mathcal V_y
:=
\mathcal O_{i,j+1}^{-1}(y)(\mathfrak v_y),
\qquad y\in\TT^N.
\]
Then \(\mathcal V\) takes values in \(\mathcal F_{i,j+1}\). By using \eqref{for:72} and the notation of Theorem \ref{th:3}, we have
\begin{align*}
 &\lim_{m\to \infty} \int_{\TT^N}A_i^{-m}p_i\circ Df^{m}_{H^{-1}z}\circ \mathcal{O}_{i,j+1}^{-1}(H^{-1}z)(\mathfrak{v}_{H^{-1}z})\omega(z)d\mathfrak{m}\\
  &=\int_{\TT^N}p_i (\mathcal{V}_{H^{-1}z})\omega(z)d\mathfrak{m}+\mathfrak D^+_{\mathcal V}(\omega).
 \end{align*}
Similarly, by using \eqref{for:73} we have
\begin{align*}
 &\lim_{m\to\infty}
\int_{\TT^N}
A_i^{m}p_i\circ Df^{-m}_{H^{-1}z}
\circ
\mathcal O_{i,j+1}^{-1}(H^{-1}z)(\mathfrak v_{H^{-1}z})
\,\omega(z)\,d\mathfrak m(z)\\
&=
\int_{\TT^N}
p_i(\mathcal V_{H^{-1}z})\omega(z)\,d\mathfrak m(z)
+
\mathfrak D^-_{\mathcal V}(\omega).
\end{align*}
By Theorem~\ref{th:3}, we have
\[
\mathfrak D^+_{\mathcal V}=\mathfrak D^-_{\mathcal V}.
\]
Therefore the positive- and negative-time limits are equal as distributions.
\subsubsection{Step 5: H\"older regularity of \(\mathcal T\)}\label{sec:19}
In this step, we show that the distribution \(\mathcal T\) is
a \(C^{\alpha\eta}\) function (see Section \ref{sec:5}). Let
\begin{align*}
  F_v(z):=\mathcal{T}(z)-\mathcal{T}(z+v),\qquad v\in\RR^N.
\end{align*}
We first prove that
\begin{align}\label{for:80}
\norm{ F_v}_{C^0}\leq C\norm{v}^{\alpha\eta}, \qquad v\in\RR^N,\quad \norm{v}\leq1.
\end{align}
We recall for any $v\in E^{s,A}$ and any $w\in E^{u,A}$ and any $m\geq0$ we have
\begin{align}\label{for:81}
 \norm{A^mv}\leq C\nu_0^m\norm{v},\quad \norm{A^{-m}w}\leq C\nu_0^{m}\norm{w},
\end{align}
see \eqref{for:77} of Section \ref{sec:8}.

First, we show that \eqref{for:80} holds for any $v\in E^{s,A}$. Using \eqref{for:66}, the
positive-time expression for $\mathcal T$ is
\begin{align*}
 \mathcal{T}&=\lim_{m\to \infty}A_i^{-m}p_i\circ Df^{m}_{H^{-1}z}\circ \mathcal{O}_{i,j+1}^{-1}(H^{-1}z)|_{V_{i,j+1}}\\
 &=\lim_{m\to \infty}\rho_i^mA_i^{-m}
\big(p_i\circ \mathcal O_{i,j+1}\circ H^{-1}\big)(A^mz)
\begin{pmatrix}
M_j^m & \rho_i^{-1}\theta_{j,m}(H^{-1}z)\\
0 & L_{i,j+1}^m
\end{pmatrix}.
\end{align*}
Then we have
\begin{align*}
 \mathcal{T}(z)-\mathcal{T}(z+v)=\mathcal{Y}_1+\mathcal{Y}_2,
\end{align*}
where
\begin{align*}
 \mathcal{Y}_1&=\lim_{m\to \infty}\rho_i^mA_i^{-m}
\big(p_i\circ \mathcal O_{i,j+1}\circ H^{-1}\big)(A^mz)
\begin{pmatrix}
M_j^m & \rho_i^{-1}\theta_{j,m}(H^{-1}z)\\
0 & L_{i,j+1}^m
\end{pmatrix}\\
&-\lim_{m\to \infty}\rho_i^mA_i^{-m}
\big(p_i\circ \mathcal O_{i,j+1}\circ H^{-1}\big)(A^m(z+v))
\begin{pmatrix}
M_j^m & \rho_i^{-1}\theta_{j,m}(H^{-1}z)\\
0 & L_{i,j+1}^m
\end{pmatrix}
\end{align*}
and
\begin{align*}
   \mathcal{Y}_2&=\lim_{m\to \infty}\rho_i^mA_i^{-m}
\big(p_i\circ \mathcal O_{i,j+1}\circ H^{-1}\big)(A^m(z+v))
\begin{pmatrix}
M_j^m & \rho_i^{-1}\theta_{j,m}(H^{-1}z)\\
0 & L_{i,j+1}^m
\end{pmatrix}\\
&-\lim_{m\to \infty}\rho_i^mA_i^{-m}
\big(p_i\circ \mathcal O_{i,j+1}\circ H^{-1}\big)(A^m(z+v))
\begin{pmatrix}
M_j^m & \rho_i^{-1}\theta_{j,m}(H^{-1}(z+v))\\
0 & L_{i,j+1}^m
\end{pmatrix}.
\end{align*}
Now we estimate $\mathcal{Y}_1$. It follows from \eqref{for:74} and \eqref{for:76} that
\begin{align*}
\Big \|\begin{pmatrix}
M_j^m & \rho_i^{-1}\theta_{j,m}(H^{-1}z)\\
0 & L_{i,j+1}^m
\end{pmatrix}\Big\|\leq Cm, \qquad \forall\,m\geq1.
\end{align*}
Since \(H^{-1}\) is \(\eta\)-H\"older and
\(\mathcal C_{i,j+1}^{-1}\) is \(\alpha\)-H\"older, the map
\[
p_i\circ\mathcal O_{i,j+1}^{-1}\circ H^{-1}
\]
is \(\alpha\eta\)-H\"older. Then we have
\begin{align*}
 &\left\|
\bigl(p_i\circ\mathcal O_{i,j+1}^{-1}\circ H^{-1}\bigr)(A^m(z+v))
-
\bigl(p_i\circ\mathcal O_{i,j+1}^{-1}\circ H^{-1}\bigr)(A^mz)
\right\|\\
&\leq \big\|p_i\circ\mathcal O_{i,j+1}^{-1}\circ H^{-1}\big\|_{C^{\alpha\eta}}\|A^mv\|^{\alpha\eta}\\
&\overset{\text{(1)}}{\leq} C\nu^{m\alpha\eta}\|v\|^{\alpha\eta}.
\end{align*}
Here in $(1)$ we use \eqref{for:81}. The above discussion, together with \eqref{for:93} of Section \ref{sec:8}, gives
\begin{align*}
 \norm{\mathcal{Y}_1}\leq \lim_{m\to \infty} C\cdot C\nu^{m\alpha\eta}\|v\|^{\alpha\eta}\cdot Cm=0.
\end{align*}
Next, we estimate $\mathcal{Y}_2$. From \eqref{for:74}, we have
\begin{align*}
& \big\|\theta_{j,m}(H^{-1}(z+v))-\theta_{j,m}(H^{-1}(z))\big\|\\
 &=\Big\|
\sum_{k=0}^{m-1}
M_j^{m-1-k}(\theta_j\circ H^{-1})(A^k(z+v))L_{i,j+1}^k\\
&\qquad-\sum_{k=0}^{m-1}
M_j^{m-1-k}(\theta_j\circ H^{-1})(A^k(z))L_{i,j+1}^k\Big\|\\
&\leq  \sum_{k=0}^{m-1}\|M_j^{m-1-k}\|\cdot \big\|(\theta_j\circ H^{-1})(A^k(z+v))-(\theta_j\circ H^{-1})(A^kz)\big\| \cdot \|L_{i,j+1}^k\| \\
&\overset{\text{(1)}}{\leq} \sum_{k=0}^{m-1} C\cdot \norm{\theta\circ H^{-1}}_{C^{\alpha\eta}}\norm{A^kv}^{\alpha\eta}\cdot C\\
&\overset{\text{(2)}}{\leq} \sum_{k=0}^{m-1} C_1\norm{\theta\circ H^{-1}}_{C^{\alpha\eta}}\nu^{k\alpha\eta}\|v\|^{\alpha\eta}\\
&\leq C_2\|v\|^{\alpha\eta}.
\end{align*}
Here in $(1)$ we use \eqref{for:76}; in $(2)$ we use \eqref{for:81}.

The above discussion shows that
\begin{align*}
 \norm{\mathcal{Y}_2}&\leq \lim_{m\to \infty} \norm{\rho_i^mA_i^{-m}}\cdot \norm{p_i\circ \mathcal O_{i,j+1}\circ H^{-1}}_{C^0}\\
 &\cdot \big\|\theta_{j,m}(H^{-1}(z+v))-\theta_{j,m}(H^{-1}(z))\big\|_{C^0}\\
 &\leq C\|v\|^{\alpha\eta}.
\end{align*}
This proves \eqref{for:80} for $v\in E^{s,A}$. Using the negative-time expression
for $\mathcal T$ and the second estimate in \eqref{for:81}, the same argument
proves \eqref{for:80} for $v\in E^{u,A}$.

Now let \(v\in\RR^N\), \(\|v\|\leq1\). Write
\[
v=v^s+v^u,
\qquad v^s\in E^s,\quad v^u\in E^u.
\]
Since the projections onto \(E^{s,A}\) and \(E^{u,A}\) are bounded,
\[
\|v^s\|+\|v^u\|\leq C\|v\|.
\]
Then
\begin{align*}
F_v(z)
=
F_{v^s}(z)+F_{v^u}(z+v^s).
\end{align*}
Therefore
\[
\|F_v\|_{C^0}
\leq
C\|v^s\|^{\alpha\eta}
+
C\|v^u\|^{\alpha\eta}
\leq
C_1\|v\|^{\alpha\eta}.
\]
Thus \eqref{for:80} holds for every \(v\in\RR^N\) with \(\|v\|\leq1\).

We next show that $F_v$ is uniformly $C^{\alpha\eta}$. Let $z,z_1\in\mathbb T^N$ with $\norm{z-z_1}\leq1$. Then
\[
\begin{aligned}
F_v(z)-F_v(z_1)
&=
F_{z_1-z}(z)-F_{z_1-z}(z+v).
\end{aligned}
\]
Using \eqref{for:80} with the translation vector $z_1-z$, we obtain
\[
        \|F_v(z)-F_v(z_1)\|
        \leq
        C\|z-z_1\|^{\alpha\eta}.
\]
Together with \eqref{for:80}, this gives
\[
        \sup_{\|v\|\leq1}\|F_v\|_{C^{\alpha\eta}}<\infty .
\]
Therefore, by Lemma \ref{le:7}, applied componentwise,
$\mathcal T$ is represented by a $C^{\alpha\eta}$ function.

\subsubsection{Step 6: Conclusion}\label{sec:29}

By the definition of \(\mathcal T\) (see Section \ref{sec:20}) and the H\"older regularity of
\(\mathcal T\) (see Section \ref{sec:19}), the distributional derivative of \(H_i\) along
\(\mathcal F_{i,j+1}\) is represented by a H\"older bundle map.
Thus \(H_i\) is \(C^{1+\text{H\"older}}\) along \(\mathcal F_{i,j+1}\).

More precisely, define
\[
\mathcal B_{i,j+1}(x)
:=
\mathcal T(Hx)\circ \mathcal O_{i,j+1}(x),
\qquad x\in\TT^N.
\]
Then \(\mathcal B_{i,j+1}:\mathcal F_{i,j+1}\to E_i\) is
\text{H\"older}, and for every \(\alpha\)-\text{H\"older} vector field
\(\mathcal V\) taking values in \(\mathcal F_{i,j+1}\), we have
\[
\widetilde{\mathfrak D}_{\mathcal V}^+(\omega)
+
\int_{\TT^N}p_i(\mathcal V_x)\omega(x)\,d\mathfrak m(x)
=
\int_{\TT^N}
\mathcal B_{i,j+1}(x)(\mathcal V_x)\omega(x)\,d\mathfrak m(x).
\]
Therefore \(\mathcal B_{i,j+1}\) represents
\(D(H_i)|_{\mathcal F_{i,j+1}}\), and \(H_i\) is
\(C^{1+\text{H\"older}}\) along \(\mathcal F_{i,j+1}\).

Since \(H(\mathcal W_i^f)=\mathcal W_i^A\), the components
\(p_k\circ H\), \(k\neq i\), are constant along \(\mathcal W_i^f\). Hence the
regularity of \(H_i=p_i\circ H\) implies the corresponding leafwise regularity
of \(H\) along \(\mathcal F_{i,j+1}\).

Next, we  show that \(D(H_i)|_{\mathcal F_{i,j+1}}\) is injective.
Differentiating the conjugacy equation
\[
H_i\circ f=A_i\circ H_i
\]
along \(\mathcal F_{i,j+1}\), we obtain
\begin{align}\label{for:83}
 \mathcal B_{i,j+1}(fx)\circ Df_x|_{\mathcal F_{i,j+1}(x)}
=
A_i\circ \mathcal B_{i,j+1}(x),
\qquad x\in\TT^N.
\end{align}
Consequently,
\begin{align}\label{for:82}
Df_x^n
\bigl(\ker \mathcal B_{i,j+1}(x)\bigr)
=
\ker \mathcal B_{i,j+1}(f^n x),
\qquad n\in\mathbb Z.\end{align}
Choose a generic $x_0$ with respect to $\mu$. Suppose, for
contradiction, that
\[
\ker \mathcal B_{i,j+1}(x_0)\neq \{0\}.
\]
Since \(x_0\) has dense \(f\)-orbit, \eqref{for:82} implies that
\(\ker\mathcal B_{i,j+1}(f^n x_0)\neq\{0\}\) for every \(n\in\mathbb Z\).
By continuity of \(\mathcal B_{i,j+1}\), we have
\begin{align*}
  \ker \mathcal B_{i,j+1}(x)\neq \{0\},\qquad \forall\,x\in\TT^N.
\end{align*}
Choose a point where $\operatorname{rank}\mathcal B_{i,j+1}$ is maximal. On a
small neighborhood of this point, the rank is locally constant, and hence
$\ker\mathcal B_{i,j+1}$ is a nontrivial continuous subbundle. Therefore, we
may choose a nonzero continuous vector field $\mathcal X$ tangent to
$\ker\mathcal B_{i,j+1}$ on this neighborhood. Let $\gamma(t)$ be a
nonconstant integral curve of $\mathcal X$. Then
\[
D(H_i)_{\gamma(t)}(\gamma'(t))=0.
\]
Since the other components \(p_k\circ H\), \(k\neq i\), are constant along
\(\mathcal W_i^f\), we also have
\[
DH_{\gamma(t)}(\gamma'(t))=0.
\]
Thus \(H(\gamma(t))\) is constant. This contradicts the injectivity of \(H\).
Hence
\[
\ker \mathcal B_{i,j+1}(x_0)=\{0\}.
\]
Let $\mathcal R$ denote the set of $\mu$-generic points.  On $\mathcal R$, the
map $\mathcal B_{i,j+1}$ is injective. Define the measurable
image bundle over $\mathcal R$ by
\[
        \mathcal S_x
        :=
        \mathcal B_{i,j+1}(x)
        \bigl(\mathcal F_{i,j+1}(x)\bigr).
\]
Then \eqref{for:83} gives
\[
        A_i(\mathcal S_x)=\mathcal S_{fx},
        \qquad x\in\mathcal R.
\]
Thus $\mathcal S$ is a $\mu$-measurable invariant subbundle of the constant
cocycle $A_i$ on \(E_i\).

Since all eigenvalues of $A_i$ have modulus $\rho_i$, the constant cocycle
$A_i|_{E_i}$ has only one Lyapunov exponent. Hence
\[
        \lambda^+(A_i|_{E_i},\mu)
        =
        \lambda^-(A_i|_{E_i},\mu).
\]
Moreover, $A_i|_{E_i}$ is fiber bunched. Therefore, by Theorem~3.3 and
Corollary~3.8 of \cite{KS13}, the measurable invariant subbundle $\mathcal S$
coincides $\mu$-almost everywhere with a H\"older continuous invariant
subbundle. We still denote this H\"older subbundle by $\mathcal S$.

We now show that $\mathcal B_{i,j+1}(x)$ is injective for every
$x\in\mathbb T^N$. Since $\mathcal S$ is H\"older and $\mu$ has full support,
the equality
\[
        \mathcal S_x
        =
        \mathcal B_{i,j+1}(x)
        \bigl(\mathcal F_{i,j+1}(x)\bigr)
\]
extends from $\mathcal R$ to every $x\in\mathbb T^N$. Since
\[
        \dim \mathcal S_x=\dim\mathcal F_{i,j+1}(x),
\]
it follows that $\mathcal B_{i,j+1}(x)$ has full rank, and hence is injective,
for every $x\in\mathbb T^N$.

Define a subbundle \(\tau\subset E_i\) by
\begin{align*}
 \tau_x=\mathcal B_{i,j+1}(H^{-1}x)(\mathcal{F}_{i,j+1}|_{H^{-1}x}),\qquad \forall\,x\in\TT^N.
\end{align*}
Using \eqref{for:83}, we see that $\tau$ is $A$-invariant. Since \(\mathcal F_{i,j+1}\), \(\mathcal B_{i,j+1}\), and \(H^{-1}\) are
H\"older, the bundle \(\tau\) is H\"older. Moreover, by
the injectivity of \(\mathcal B_{i,j+1}\),
\[
\dim \tau_x=\dim \mathcal F_{i,j+1},
\qquad x\in\TT^N.
\]
 Let
\begin{align*}
 \mathbb{S}=\{n\in\NN:\, A\text{ and }A^n\text{ have exactly the same invariant subspaces}\}.
\end{align*}
Choose a sequence $s_n\in \mathbb{S}$ such that $s_n\to \infty$.  Let
\begin{align*}
 P_{s_n}=\{x\in \TT^N:\,A^{s_n}(x)=x\}\quad\text{and}\quad P=\bigcup_{n}P_{s_n}.
\end{align*}
By Lemma \ref{th:2}, $P$ is dense in $\TT^N$. The $A$-invariancy of $\tau$ implies that
\begin{align*}
 A_i^{s_n}(\tau_x)=\tau_x.
\end{align*}
Since \(A\) and \(A^{s_n}\) have the same invariant subspaces, it follows that
\begin{align*}
 A_i(\tau_x)=\tau_x,\qquad \forall\,x\in P.
\end{align*}
Since $A$ has no repeated eigenvalues,  $A_i$ has only finitely many $A_i$-invariant subspaces of dimension
$\dim\mathcal F_{i,j+1}$ inside $E_i$. The H\"older continuity of the bundle \(\tau\) and density of $P$
implies that
\(\tau\) is constant. Hence there exists
an \(A_i\)-invariant subspace \(S_{i,j+1}\subset E_i\) such that
\[
\tau_x=S_{i,j+1},
\qquad x\in\TT^N.
\]
Thus, for every \(x\in\TT^N\),
\[
\mathcal B_{i,j+1}(x):\mathcal F_{i,j+1}(x)\to S_{i,j+1}
\]
is a linear isomorphism.  Next, we show that $\mathcal B_{i,j+1}$ is $\alpha$-H\"older as a bundle map.
The cocycle $Df|_{\mathcal F_{i,j+1}}$ is $\alpha$-H\"older and fiber bunched,
while the target cocycle $A_i|_{S_{i,j+1}}$ is uniformly quasiconformal, indeed
conformal after choosing an adapted inner product. By \eqref{for:83}, the measurable bundle map \(\mathcal B_{i,j+1}\) conjugates
the cocycle \(Df|_{\mathcal F_{i,j+1}}\) to the constant cocycle
\(A_i|_{S_{i,j+1}}\). Since $\mu$ is ergodic, has
full support, and has local product structure, Theorem~2.7 of \cite{S15}
implies that the measurable conjugacy $\mathcal B_{i,j+1}$ coincides
$\mu$-almost everywhere with an $\alpha$-H\"older continuous conjugacy. Since
$\mathcal B_{i,j+1}$ is already continuous and $\mu$ has full support, the two
conjugacies agree everywhere. Hence $\mathcal B_{i,j+1}$ is $\alpha$-H\"older
on $\mathbb T^N$.

Since $\mathbb T^N$ is compact and $\mathcal B_{i,j+1}(x)$ is injective for
every $x$, there exists $c>0$ such that
\[
        \|\mathcal B_{i,j+1}(x)u\|
        \geq
        c\|u\|,
        \qquad
        x\in\mathbb T^N,\quad u\in\mathcal F_{i,j+1}(x).
\]
Equivalently,
\[
        \sup_{x\in\mathbb T^N}
        \|\mathcal B_{i,j+1}(x)^{-1}\|
        <\infty .
\]
This implies that the inverse bundle map is $\alpha$-H\"older, and so
$\mathcal B_{i,j+1}$ is bi-$\alpha$-H\"older.

Since \(\mathcal C_{i,j+1}(x):\mathcal F_{i,j+1}(x)\to V_{i,j+1}\) is a
linear isomorphism, we have
\[
\dim S_{i,j+1}
=
\dim \mathcal F_{i,j+1}
=
\dim V_{i,j+1}.
\]
Therefore we may choose a linear isomorphism
\[
L_{j+1}:V_{i,j+1}\to S_{i,j+1}.
\]
Hence the triple
\[
\bigl(S_{i,j+1},\,L_{j+1},\,\mathcal B_{i,j+1}\bigr)
\]
satisfies the conclusion of Proposition~\ref{le:3} for \(j+1\). This completes the
induction step.

\section{Proof of Theorems~\ref{th:6} and \ref{th:5}}
We first recall the following form of Journ\'e's lemma.
\begin{lemma}[Journ\'e \cite{J88}]\label{le:journe}
Let \(M_1\) and \(M_2\) be manifolds, and for \(j=1,2\), let
\(\mathcal F_j^s\) and \(\mathcal F_j^u\) be continuous transverse foliations
on \(M_j\) with uniformly smooth leaves. Suppose that
\(h:M_1\to M_2\) is a homeomorphism such that
\[
        h(\mathcal F_1^s)\subset \mathcal F_2^s,
        \qquad
        h(\mathcal F_1^u)\subset \mathcal F_2^u.
\]
Assume moreover that the restrictions of \(h\) to the leaves of
\(\mathcal F_1^s\) and \(\mathcal F_1^u\) are uniformly \(C^{r+\nu}\), where
\(r\in\mathbb N\) and \(0<\nu<1\). Then \(h\) is \(C^{r+\nu}\).
\end{lemma}

The final induction-and-Journ\'e argument is standard and follows the same
strategy as in \cite{GKS11}. We include the details for completeness.

It follows from \eqref{for:87} in Section~\ref{sec:12} that
\[
        H(\mathcal W_{i_0}^f)=\mathcal W_{i_0}^A.
\]
Hence, by Theorem~\ref{th:4}, $H$ is a
$C^{1+\text{H\"older}}$ diffeomorphism along $\mathcal W_{i_0}^f$. This gives
the base case of the induction.

We now proceed by induction. Suppose that, for some
$i_0\leq j<\ell$, $H$ is a $C^{1+\text{H\"older}}$ diffeomorphism along
$\mathcal W_{i_0,j}^f$. By \eqref{for:86} in Section~\ref{sec:12}, applied successively inside \(\mathcal W_{i_0,j+1}^f\), we have
\[
        H(\mathcal W_{j+1}^f)=\mathcal W_{j+1}^A.
\]
Therefore, by Theorem~\ref{th:4}, $H$ is a
$C^{1+\text{H\"older}}$ diffeomorphism along $\mathcal W_{j+1}^f$.

The foliations $\mathcal W_{i_0,j}^f$ and $\mathcal W_{j+1}^f$ are transverse
subfoliations inside $\mathcal W_{i_0,j+1}^f$. Similarly,
$\mathcal W_{i_0,j}^A$ and $\mathcal W_{j+1}^A$ are transverse subfoliations
inside $\mathcal W_{i_0,j+1}^A$. Since $H$ maps each of these foliations to
the corresponding linear foliation and is $C^{1+\text{H\"older}}$ along both
subfoliations, Journ\'e's lemma implies that $H$ is
$C^{1+\text{H\"older}}$ along $\mathcal W_{i_0,j+1}^f$. Applying the same
argument to the leafwise inverse, we obtain that $H$ is a
$C^{1+\text{H\"older}}$ diffeomorphism along $\mathcal W_{i_0,j+1}^f$. This
proves the induction step.

Thus $H$ is a $C^{1+\text{H\"older}}$ diffeomorphism along
\[
        \mathcal W^{u,f}=\mathcal W_{i_0,\ell}^f.
\]
Applying the same argument to $f^{-1}$ and $A^{-1}$, and using the fact that
$H$ also conjugates $f^{-1}$ to $A^{-1}$, we obtain that $H$ is a
$C^{1+\text{H\"older}}$ diffeomorphism along
\[
        \mathcal W^{s,f}=\mathcal W^{u,f^{-1}}.
\]
Finally, since the stable and unstable foliations are transverse and have
uniformly smooth leaves, Journ\'e's lemma implies that $H$ is globally
$C^{1+\text{H\"older}}$ on $\mathbb T^N$. Applying the same reasoning to the
inverse conjugacy gives the same regularity for $H^{-1}$. Hence $H$ is a
$C^{1+\text{H\"older}}$ diffeomorphism of $\mathbb T^N$. This completes the
proof.

\end{document}